\documentclass[a4paper,12pt]{article}
\usepackage{amsmath,amsfonts,amssymb,amsthm,amscd}
\usepackage[arrow, matrix, curve]{xy}
\usepackage[colorlinks=true,citecolor=blue]{hyperref}

\font\cmssl=cmss10 at 12 pt

\usepackage{slashed}
\usepackage[normalem]{ulem}

\newtheorem{thm}{Theorem}
\newtheorem{lem}[thm]{Lemma}
\newtheorem{prop}[thm]{Proposition}
\newtheorem{defn}[thm]{Definition}
\newtheorem{cor}[thm]{Corollary}
\newtheorem{rem}[thm]{Remark}

\newtheorem{ass}[thm]{Assumption}

\title{Classification of odd generalized Einstein metrics  on  $3$-dimensional Lie groups}

\date{\today}

\author{Vicente Cort\'es and  Liana David}

\begin{document}

\maketitle

{\bf Abstract:}  
An odd generalized metric   $E_{-}$ on a  Lie group  $G$  of dimension $n$  is a left-invariant generalized metric on a Courant algebroid 
$E_{H, F}$  of type $B_{n}$ over $G$ with left-invariant  twisting forms $H\in \Omega^{3}(G)$ and $F\in \Omega^{2}(G)$. 
Given an odd generalized 
metric $E_{-}$ on $G$ we determine the affine space of left-invariant Levi-Civita generalized connections of $E_{-}.$ 
Given in addition a left-invariant divergence operator  $\delta$  we 
show that there is  a left-invariant Levi-Civita generalized connection  of $E_{-}$ with divergence $\delta$ and we 
compute the  corresponding Ricci tensor 
$\mathrm{Ric}^{\delta}$  of  the pair $(E_{-}, \delta )$. 
The  odd generalized metric  $E_{-}$ is called odd generalized Einstein with  divergence $\delta$ 
if $\mathrm{Ric}^{\delta}=0$.
We describe all 
odd generalized Einstein metrics of arbitrary left-invariant divergence  on  all $3$-dimensional Lie groups. 
 
\section{Introduction} 

Generalized geometry is an active research area in mathematics with applications in mathematical physics.
The starting  idea proposed by N. Hitchin   in \cite{hitchin} was to include complex and symplectic structures on a manifold $M$ 
into a more general type of structure defined  on the generalized tangent bundle
$\mathbb{T} M = TM\oplus T^{*}M$, called a generalized  complex structure. Many notions from classical geometry were extended to generalized geometry. They  were initially considered  on generalized tangent bundles   twisted by closed $3$-forms (known 
in the literature as  exact Courant algebroids) and afterwards on arbitrary Courant algebroids, which were defined for the first time in  \cite{w}. \

The present paper may be seen as the odd analogue of \cite{c-k-paper}, where left-invariant generalized metrics 
on exact Courant algebroids over Lie groups 
were systematically  studied and, when the Lie  groups are  $3$-dimensional, a classification
of 
such metrics which satisfy the generalized Einstein equations 
was obtained.  Our aim in this paper is to develop an  analogous theory,
with exact Courant algebroids replaced by their odd analogue. Odd exact Courant algebroids were introduced in \cite{rubio}.
From the view-point of the classification of Courant algebroids developed in \cite{chen}, they  represent
the simplest class of Courant algebroids, after the exact ones.
In this paper we only consider (unless otherwise stated) odd exact Courant algebroids $E_{H, F}$
in the standard form (often called  Courant algebroids of type $B_{n}$)
over Lie groups $G$,  with left-invariant twisting forms $H$ and $F$ (where $n:= \dim G$).
The Lie group action  of $G$  by left-multiplication  lifts to the underlying bundle $TG\oplus T^{*}G\oplus \mathbb{R}$ 
of $E_{H, F}$  and preserves the Courant algebroid structure of $E_{H, F}.$ Tensor fields, generalized connections,  operators etc.\
 invariant under this action will be  called left-invariant.\

In Section  \ref{odd-metrics-section}   we consider an arbitrary  left-invariant generalized metric $E_{-}$ on  $E_{H, F}$.
By using a  Courant algebroid isomorphism,  we can (and will)  assume that $E_{-} $
is in the standard form,  i.e.\  
$E_{-} = \{ X - g(X) \ \mid X\in TG\}$, where $g$ is a (left-invariant) pseudo-Riemannian metric on $G$ (see Proposition \ref{standard-form}). 
We show  that given a left-invariant divergence operator $\delta \in  (\mathfrak{g}\oplus \mathfrak{g}^{*} \oplus \mathbb{R})^{*}$
(where $\mathfrak{g} := \mathrm{Lie}\, G$), 
 there is always a left-invariant Levi-Civita generalized connection 
$D$ of $E_{-}$ with divergence  $\delta$ (see Theorem \ref{left-inv-div}).  Along the way we describe the 
affine space of all left-invariant Levi-Civita generalized connections of $E_{-}$. It is worth to remark that similar  results are known for generalized metrics on
arbitrary Courant algebroids (not necessarily over Lie groups and not necessarily  of type $B_{n}$).  The specific 
point emphasized here is the left-invariance: we prove that a 
Levi-Civita generalized connection 
with prescribed divergence operator (whose existence is known from the general theory, see e.g.\ \cite{mario}) 
can be chosen to be left-invariant, when the generalized metric and divergence operator are so.

There are various approaches to define curvature in generalized geometry (see e.g. \cite{baraglia, hu, goto, mario}). In this paper we 
adopt view-point of \cite{mario}, where a generalized Ricci tensor was 
defined using Levi-Civita generalized connections with prescribed divergence operator. 
Following the general theory of \cite{mario}, 
given a  left-invariant generalized metric $E_{-}$ and a left-invariant divergence operator  $\delta$ on $E_{H, F}$,
we compute
the  generalized Ricci tensor 
of the pair  $(E_{-}, \delta )$
(see Theorem \ref{2-prime}). 
We define  an odd generalized Einstein metric on $G$ to be a left-invariant 
generalized metric  $E_{-}$ on a Courant algebroid $E_{H, F}$ of type $B_{n}$ over $G$, 
which satisfies the generalized Einstein equation 
\[ \mathrm{Ric}^{\delta}=0\] 
for a left-invariant divergence operator $\delta$.
We  often refer to $\delta$ as the divergence of $E_{-}.$\ 
In Corollary  \ref{3} we express  the generalized Einstein condition in terms of, on the one hand,
the coefficients of $H$ and $F$ with respect to  a $g$-orthonormal basis $(v_{a})$ of $\mathfrak{g}$ and, on the other hand, the coefficients of the Dorfman bracket and of $\delta$ with respect to 
the induced basis $(v_{A})$ of $\mathfrak{g}\oplus \mathfrak{g}^{*}\oplus \mathbb{R}.$

In Section \ref{three-dim} we apply the results developed in the previous section to the case when $G$ is $3$-dimensional. 
Corollary \ref{forapplic}, which is  a rewriting of Corollary  \ref{3} in three dimensions,  
represents  a main computational tool, which will be used all along the next sections. 
There are several  important features in three dimensions: 
$3$-dimensional Lie algebras have a relatively simple classification (see below); 
the constraints on the twisting forms $H$ and $F$ reduce to $ dF =0$; 
the classification of generalized Einstein metrics on Courant algebroids of type $B_3$ with $F=0$  reduces to that of
 ordinary generalized Einstein metrics on $3$-dimensional Lie groups, which was developed  in \cite{c-k-paper} (compare Remark~\ref{F} ii)). 
Owing to  this,   in the results summarized below we  assume that $F\neq 0$,  see Remark \ref{F} ii).

In Section \ref{unimod-section} we develop a complete description of  odd generalized Einstein metrics on $3$-dimensional  unimodular Lie groups.
Recall the classification of $3$-dimensional unimodular Lie algebras  (see e.g.\  \cite{milnor}, Section 4):  
such a Lie algebra is isomorphic to 
$\mathfrak{so}(3)$, $\mathfrak{so}(2,1)$, $\mathfrak{e}(2)$,  $\mathfrak{e}(1,1)$ 
(where $\mathfrak{e}(p, q)$ denotes the Lie algebra of the isometry group of $\mathbb{R}^{p,q}$
and $\mathfrak{e}(p):= \mathfrak{e}(p,0)$),
$\mathfrak{heis}$ (the Lie algebra of the Heisenberg group) or $\mathbb{R}^{3}$ (the abelian  $3$-dimensional Lie algebra). 
The results from Section  \ref{unimod-section} 
can be summarized as follows (for a  precise description of all  odd generalized Einstein metrics on unimodular Lie groups,
see  Theorems \ref{classif-diag}, \ref{classif-l2}, \ref{classif-l3}, \ref{classif-l5}):

\begin{thm}\label{short-unimod}  Let $G$ be a $3$-dimensional unimodular Lie group with Lie algebra $\mathfrak{g}.$ 
If  $\mathfrak{g}$ is isomorphic to 
$\mathfrak{so}(3)$, $\mathfrak{so}(2,1)$,  $\mathfrak{e}(1,1)$  or $\mathfrak{heis}$,  then there are divergence-free odd generalized  Einstein metrics on $G$, as well as odd generalized Einstein metrics with non-zero divergence. 
If $\mathfrak{g}$ is  abelian or isomorphic to $\mathfrak{e}(2)$,  then there are no
odd generalized Einstein metrics on $G$. 
\end{thm}

In order to prove Theorem  \ref{short-unimod} let 
$E_{-} =\{ X- g(X) \mid  X\in TG\}$ be an odd generalized Einstein metric on  a $3$-dimensional unimodular Lie group $G$
with Lie algebra $\mathfrak{g}.$
As shown in \cite{c-k-paper}, 
the pair $(\mathfrak{g}, g )$ defines an operator $L\in \mathrm{End}\, \mathfrak{g}$ (unique up to multiplication by $\pm1$),
which relates the Lie bracket of $\mathfrak{g}$ with the cross product defined by $g$ and a choice of orientation on $\mathfrak{g}.$ 
Since  $\mathfrak{g}$ is unimodular, the operator $L$  is ${g}$-symmetric and admits five classes of  canonical forms.
The description of odd generalized Einstein metrics on $G$  (see Theorems \ref{classif-diag}, \ref{classif-l2}, \ref{classif-l3} and 
\ref{classif-l5}) follows by considering each of these classes  separately and applying Corollary \ref{forapplic} in a basis $(v_{a})$
of $\mathfrak{g}$,  adapted to these canonical forms.
The identification of the Lie algebra 
$\mathfrak{g}$
(and Theorem \ref{short-unimod} above) 
follows  then from Lemma \ref{identif-lie-alg}, which compares 
the classification   of  $3$-dimensional unimodular Lie algebras mentioned above with the classification 
(by means of canonical forms for the operator $L$) 
of such  Lie algebras endowed with a  non-degenerate  scalar product.\

In Section  \ref{non-unimod-section} we consider the same question as in the previous section, but under the assumption
that  the $3$-dimensional Lie group $G$ is non-unimodular.   
Then the  Lie algebra $\mathfrak{g} =\mathrm{Lie}\, G$ is a semi-direct product 
$\mathbb{R}\ltimes_{A} \mathbb{R}^{2}$ 
 of $\mathbb{R}$ and an abelian Lie algebra $\mathfrak{u} =\mathbb{R}^{2}$
 (the unimodular kernel of $\mathfrak{g}$), 
 where $\mathbb{R}$ acts on $\mathbb{R}^{2}$ 
by an endomorphism $A\in \mathrm{End}\, \mathbb{R}^{2}$ with non-zero trace. The isomorphism classes of 3-dimensional non-unimodular  Lie algebras correspond to 
a list of canonical forms for the operator $A$, see  (\ref{canonical-forms-A}), which,
in the notation of \cite{vinb},  define  the Lie algebras 
$\tau_{2} (\mathbb{R})\oplus \mathbb{R}$, $\tau_{3}(\mathbb{R})$, $\tau_{3, \lambda}(\mathbb{R})$ ($\lambda \in \mathbb{R}\setminus \{ -1\}$, 
$0<  | \lambda | \leq 1$) and 
$\tau_{3, \lambda}^{\prime} (\mathbb{R}) $ ($\lambda \in \mathbb{R}\setminus \{ 0\}$). Our main results from this section are summarized 
as follows:

\begin{thm}\label{short-non-unimod}
Let $G$ be a $3$-dimensional non-unimodular Lie group with Lie algebra $\mathfrak{g}$ and let $\mathfrak{u}$  be the  unimodular kernel of 
$\mathfrak{g}.$\

i)  There   are 
odd generalized Einstein metrics 
$E_{-} =  \{ X - g(X) \mid X\in TG\}$  on $G$ with  $g\vert_{\mathfrak{u}\times \mathfrak{u}}$ degenerate.
If   $\mathfrak{g}$ is isomorphic to $\tau_{2} (\mathbb{R})\oplus \mathbb{R}$,
$\tau_{3}(\mathbb{R})$ or  $\tau_{3, \lambda }(\mathbb{R})$,  then $E_{-}$ can be chosen to be divergence-free or with non-zero divergence.    If 
$\mathfrak{g}$ is isomorphic to $\tau_{3, \lambda}^{\prime} (\mathbb{R})$ then $E_{-}$ has necessarily non-zero divergence.\

ii)  There are 
odd generalized Einstein metrics  $E_{-} =\{X - g(X) \mid  X\in TG\}$ on $G$ with $g\vert_{\mathfrak{u}\times \mathfrak{u}}$ non-degenerate,
if and only if $\mathfrak{g}$ is isomorphic to   $\tau_{2} (\mathbb{R})\oplus \mathbb{R}$, $\tau_{3, \frac{1}{2}} (\mathbb{R})$ or 
$\tau_{3, -\frac{1}{2}} (\mathbb{R})$. All such metrics have non-zero divergence. 
 \end{thm}

In order to prove Theorem \ref{short-non-unimod} let  $E_{-} = \{ X - g(X)\mid  X\in TG\}$ be an odd generalized Einstein metric on 
a $3$-dimensional non-unimodular Lie group $G$ with Lie algebra $\mathfrak{g}.$  
As proved in \cite{c-k-paper}, $\mathfrak{g}$ admits a basis $(v_{a})$ adapted to the decomposition of $\mathfrak{g}$ into a semi-direct product. The form of the basis $(v_{a})$ depends on the (non)-degeneracy 
of $g\vert_{\mathfrak{u}\times \mathfrak{u}}.$ For this reason, the cases when $g\vert_{\mathfrak{u}\times \mathfrak{u}}$ 
is degenerate or not are treated separately (see Subsections \ref{degenerate-section} and  \ref{nondeg-sect}). 
Applying Corollary \ref{forapplic} in the   basis $(v_{a})$
we  arrive at  Theorems \ref{deg-thm}, \ref{non-deg-thm} and Corollary \ref{case-7}, which 
provide a complete description of odd generalized  Einstein metrics on all $3$-dimensional non-unimodular
Lie groups.
The identification of the corresponding Lie algebras is done in
Propositions \ref{identif-lie} and \ref{identif-lie-non-deg}, which lead to Theorem \ref{short-non-unimod}.\\

{\bf Acknowledgements.} Research of V.C.\ was supported by the German Science Foundation (DFG) under
Germany's Excellence Strategy -- EXC 2121 ``Quantum Universe'' -- 390833306. L.D.\ was partially supported by the UEFISCDI research grant PN-III-P4-ID-PCE-2020-0794, project title  ``Spectral Methods in Hyperbolic Geometry''.

\section{Odd  generalized metrics  on  Lie groups}\label{odd-metrics-section}

\subsection{Basic definitions}

Let $G$ be a Lie group  of dimension $n$ with Lie algebra $\mathfrak{g}$ and 
$$ 
E_{H, F}:= (TG \oplus T^{*}G \oplus \mathbb{R} , \pi ,  \langle \cdot , \cdot \rangle , 
 [\cdot , \cdot ]_{H, F} )
$$ 
a Courant algebroid  of type $B_{n}$ over $G$,  where $H\in \Omega^{3} (G)$,  $F\in \Omega^{2}(G)$ are such
that   $dF =0$ and $dH +F\wedge F =0.$ 
Unless otherwise stated,  in this paper we always assume that $H$ and $F$ are left-invariant.
Recall that  the anchor $\pi : E_{H, F} \rightarrow TG$  is the natural projection,  the  scalar product  is given by
\begin{equation}
\langle X_{x} +\xi_{x} +\lambda_{x} , Y_{x}  +\eta_{x}+ \mu_{x} \rangle := \frac{1}{2} ( \xi_{x} (Y_{x}) +\eta_{x} (X_{x}) ) +\lambda_{x} \mu_{x}
\end{equation}
where   $x\in G$, $X_{x}, Y_{x}\in T_{x} G$, $\xi_{x} , \eta_{x} \in T^{*}_{x}G$,  $\lambda_{x}  , \mu_{x} \in \mathbb{R} = (G\times \mathbb{R})_{x}$
and the Dorfman bracket is given  
\begin{align}
\nonumber& [ X +\xi + \lambda , Y +\eta + \mu ]_{H, F}= {\mathcal L}_{X} ( Y+\eta ) - i_{Y} d\xi + 2\mu d\lambda + i_{X} i_{Y} H \\
\label{def-dorfman}& - 2( \mu i_{X} F - \lambda i_{Y} F ) + X(\mu) - Y(\lambda ) + F(X, Y),
\end{align}
for any $X, Y\in {\mathfrak X}(G)$, $\xi , \eta \in \Omega^{1}(G)$ and $\lambda , \mu \in C^{\infty}(G)$.

There is a natural action of  $G$  on $TG\oplus T^{*}G\oplus \mathbb{R}$ which lifts the action  of $G$ on $G$ by left-translations:
for any $x,y\in G$,  
$$
x: ( E_{H, F})_{y} \rightarrow (E_{H, F})_{xy}
$$ 
is defined by 
\begin{equation}\label{action-G}
x ( X_{y} +\alpha_{y} +\lambda_{y} ) := (d_{y} L_{x})X_{y} + \alpha_{y} \circ  (  d_{y} L_{x})^{-1}  +\lambda_{y} ,
\end{equation}
for any $X_{y}\in T_{y}G$, $\alpha_{y} \in T^{*}_{y}G$ and $\lambda_{y}\in \mathbb{R}$.
It induces naturally an action of $G$ on $\Gamma (E_{H ,F})=\mathfrak{X}(G)\times \Omega^1 (G)\times C^\infty (G)$, defined by 
\[ (x,u=(X,\alpha, \lambda )) \mapsto x\cdot u:=((L_{x})_*X,(L_{x})_*\alpha, (L_{x})_*\lambda),\]
where  $(L_{x})_*$ denotes the push-forward on tensor fields induced by the diffeomorphism $L_{x}$.
(In particular, $(L_{x})_*\lambda = \lambda \circ L_{x}^{-1}$.)
%We  shall denote by $x\cdot u$ the action of $x$ on $u\in \Gamma (E_{H, F})$.
The section $u\in \Gamma (E_{H, F})$ is  called left-invariant if $x\cdot u =u$ for any $x\in G.$
The anchor, scalar product  and Dorfman bracket  of $E_{H, F}$ are compatible with this action, i.e.\ for any $x\in G$, 
\begin{equation}\label{invariance-scalar-anchor}
 \pi ( xu) = (d_{y}L_{x})\pi (u),\ 
\langle x u, x v\rangle = \langle u, v\rangle ,\\
\forall u, v\in (E_{H, F})_{y},\ y\in G,
\end{equation}
and 
\begin{equation}\label{four}
 [ x\cdot u, x\cdot v]_{H, F} = x\cdot [u, v]_{H, F}, \ \forall u, v\in \Gamma (E_{H, F}), 
\end{equation}
where in (\ref{four}) we used that $H$ and $F$ are left-invariant.
In particular, for any   left-invariant sections $u, v\in \Gamma (E_{H, F})$, $\langle u, v\rangle$ is constant and
$[u, v]_{H, F}$ is also left-invariant.

Recall that a {\cmssl generalized metric}  on $E_{H, F}$ is a subbundle $E_{-} \subset TG \oplus T^{*} G\oplus\mathbb{R}$ 
such that $\langle \cdot , \cdot \rangle\vert_{E_{-}}$ is non-degenerate and $\pi\vert_{E_{-}} : E_{-} \rightarrow TG$ is an isomorphism. 
 The generalized metric $E_{-}$ induces
a pseudo-Riemannian metric $g$ on $G$, defined by 
\begin{equation}\label{induced-metric}
g(X, Y) := - \langle s(X), s(Y) \rangle ,\  \forall X, Y\in TG,
\end{equation}
where  $s:TG \rightarrow E_-$ is the  inverse of $\pi\vert_{E_{-}}.$ 
Let $E_{+} := E_{-}^{\perp}$ be the orthogonal complement  of $E_{-}$ with respect to $\langle \cdot , \cdot \rangle$
and $\mathcal G\in\Gamma (\mathrm{Sym}^2E^*)$ defined by
\begin{equation}\label{mathcal-g}
\mathcal G := \langle \cdot , \cdot \rangle\vert_{E_{+}} - \langle \cdot , \cdot \rangle\vert_{E_{-}},
\end{equation}
where $E=E_{H,F}$.
Remark that $\mathcal G \vert_{E_{\pm}}= \pm \langle \cdot , \cdot\rangle\vert_{E_{\pm}}$. 

\begin{defn}
An {\cmssl odd generalized metric} on $G$ is a left-invariant generalized metric on a Courant algebroid $E_{H , F}$  of type $B_{n}$ over $G$
(with $H$ and $F$ left-invariant) 
 i.e.\  a generalized metric $E_{-} $ 
which is preserved by the action  (\ref{action-G}) of $G$ on $TG\oplus T^{*}G\oplus 
\mathbb{R} .$
\end{defn}

\begin{rem}\label{invariance} {\rm If $E_{-}$ is an odd generalized metric and $s: TG \rightarrow E_{-}$ is the inverse of 
$\pi\vert_{E_{-}}$, 
then 
\begin{equation}\label{inv-s}
s (  (d L_{x})(X)) = x  s(X),\ \forall X\in TG,\ x\in G.
\end{equation}
It follows that the induced metric defined by (\ref{induced-metric})  is left-invariant.
}
\end{rem}

\begin{prop}\label{standard-form} Let $E_{-}$ be an odd generalized metric on $G$, defined on a  Courant algebroid $E_{H, F}$ of type $B_{n}$. 
Then  there is an isomorphism from $E_{H, F}$   to another  Courant algebroid $E_{\tilde{H}, \tilde{F}}$ of type $B_{n}$, 
with $\tilde{H}$ and $\tilde{F}$ left-invariant, 
which maps $E_{-}$ to the odd generalized  metric 
\begin{equation}\label{standard-e}
\tilde{E}_{-} =  \{ X - g(X) \mid X\in TG\}  ,
\end{equation}
 where $g$ is the (left-invariant) pseudo-Riemannian metric on $G$ induced by $E_{-}.$ 
 
 \end{prop}

\begin{proof} This is the analogue of  Lemma~7 of \cite{c-d-bn}  in the setting of Lie groups.
As shown in \cite{c-d-bn} (see also \cite{baraglia}),  $E_{-}$ is of the form  
\begin{equation}\label{form}
E_{-} = \{ X - i_{X} (g-b) - A(X)A + A(X)\mid  X\in TG\}
\end{equation}
where $g$ is the pseudo-Riemannian metric 
on $G$ induced by $E_{-}$, $A\in \Omega^{1}(G)$ and $b\in \Omega^{2}(G).$ 
From Remark \ref{invariance}, $g$ is left-invariant. Relations (\ref{inv-s}) and (\ref{form})
imply  that $A$ and $b$ are left-invariant too.  From \cite{c-d-bn}, Section 2.1, $(E_{H, F}, E_{-} )$ is isomorphic to
$(E_{\tilde{H}, \tilde{F}}, \tilde{E}_{-})$, where 
\begin{align*}
\nonumber \tilde{H} & = H - db - ( 2F + dA) \wedge A,\\
\nonumber \tilde{F} & = F + dA
\end{align*}
and $\tilde{E}_{-}$ is given by  (\ref{standard-e}). Since $g, A, b$ are left-invariant, so are $\tilde{H}$ and $\tilde{F}.$ 
\end{proof}

\subsection{Left-invariant Levi-Civita generalized connections}

Following \cite{mario}, we start by recalling well known facts on Levi-Civita generalized connections.
Then we will define and study  left-invariant Levi-Civita generalized connections of  odd generalized metrics.

Let  $E= E_{H, F}$ be a Courant algebroid  of type $B_{n}$ over a Lie group  $G$, with anchor $\pi : E \rightarrow TG$,
scalar product $\langle\cdot , \cdot \rangle$ and Dorfman bracket $[\cdot , \cdot ]_{H, F}.$ We do not assume for the moment that
$H$ and $F$ are left-invariant.  
Recall that a generalized connection on $E$ is a linear map 
$$ 
D: \Gamma (E) \rightarrow \Gamma (E^{*} \otimes E),\ v \mapsto  Dv = ( u\mapsto  D_{u}v)
$$
which satisfies, for any $u, v ,w\in \Gamma (E)$ and $f\in C^{\infty}(G)$, 
\begin{align}
\nonumber& D_{u} ( fv ) = \pi (u) (f) v + f D_{u} v\\
\label{def-gen-conn}& \pi (u) \langle v, w\rangle = \langle D_{u} v, w\rangle + \langle v, D_{u} w\rangle .
\end{align}
The torsion of $D$ is defined by 
\begin{equation}
T^{D}(u, v, w) :=\langle D_{u} v - D_{v} u - [u, v]_{H, F} , w\rangle + \langle v, D_{w} u \rangle,
\end{equation}
for any $u, v, w\in \Gamma (E)$ and  is a section of $\Lambda^{3} E^{*}.$ 
The generalized connection  $D$ preserves a generalized metric $E_{-}$ on $E$ if, by definition, 
$ D_{u} \Gamma ( E_{-}) \subset \Gamma (E_{-})$
for any $u\in \Gamma (E)$, 
or, equivalently, 
\begin{equation}
\pi (u) \mathcal  G (v,w) = \mathcal G ( D_{u} v, w) + \mathcal G( v, D_{u} w),\ \forall  u, v, w\in \Gamma (E ),
\end{equation}
where $\mathcal G$ is the bilinear form  defined by    (\ref{mathcal-g}).  A {\cmssl Levi-Civita generalized connection}
of $E_{-}$ is a  torsion-free generalized connection which preserves $E_{-}.$

A {\cmssl divergence operator} is a map
$\delta : \Gamma (E) \rightarrow C^{\infty} (G)$
which satisfies
\begin{equation}
\delta ( fu ) = \pi (u) (f) + f \delta (u),\ \forall u\in \Gamma (E),\ f\in C^{\infty}(G).
\end{equation}
Given a generalized connection $D$, the map $\delta^{D} : \Gamma (E) \rightarrow C^{\infty} (G)$ defined by
$\delta^{D} (u) := \mathrm{tr}\, Du$ is a divergence operator, called the {\cmssl divergence of $D$}. 
Given a generalized metric $E_{-}$ and a divergence operator $\delta$, there is always a Levi-Civita generalized connection
of $E_{-}$  with divergence $\delta$ (see \cite{mario}).\\

From now on we assume that  the forms  $H$ and $F$ are left-invariant. 

\begin{defn} A generalized connection on $E$ is  {\cmssl left-invariant} if $D_{u} v$ is left-invariant, for any left-invariant sections $u, v\in \Gamma (E).$
\end{defn} 

Let $\mathfrak{g}:= \mathrm{Lie}\, (G)$ and  $\tilde{E}:=\mathfrak{g}\oplus
\mathfrak{g}^{*} \oplus \mathbb{R}$, which will be identified with the vector space of left-invariant sections of $E$. 
Since $\langle u,v\rangle$ is constant  for any left-invariant sections $u, v\in \Gamma (E)$, $\langle\cdot , \cdot \rangle$
induces a scalar product on $\tilde{E}$, also denoted by $\langle \cdot , \cdot \rangle .$  
With these identifications, 
a left-invariant generalized connection $D$  is an element of $\tilde{E}^{*}\otimes \mathfrak{so} (\tilde{E})$ and 
its  torsion  is  an element of $\Lambda^{3} \tilde{E}^{*}.$ 
Let $E_{-}$ be a left-invariant generalized metric on $E$. It  induces a subspace  $\tilde{E}_{-}\subset \tilde{E}$,  and,
 similarly $E_{+} = E_{-}^{\perp}$ induces a subspace
$\tilde{E}_{+}\subset \tilde{E}$. We obtain that $D$  preserves
$E_{-}$ if and only if it is an element of $\tilde{E}^{*}\otimes ( \mathfrak{so}(\tilde{E}_{+}) \oplus \mathfrak{so} (\tilde{E}_{-} )).$

\begin{defn} 
A {\cmssl  left-invariant divergence operator}  is a divergence operator $\delta : \Gamma (E) \rightarrow C^{\infty}(G)$ such that $\delta (u)$ is a left-invariant (hence constant)
function, for any left-invariant section $u\in \Gamma (E).$ 
\end{defn}

A  left-invariant divergence operator  will be  identified with an element of $\tilde{E}^{*}.$ 
Our main result from this section is the following theorem.

\begin{thm}\label{left-inv-div}  Assume that $\dim G\ge 2$. Given a  left-invariant generalized metric $ E_{-}$ on $E$  and a left-invariant divergence operator $\delta \in \tilde{E}^{*}$, there is a left-invariant Levi-Civita generalized connection $D$  of $E_{-}$ such
that $ \delta^{D} = \delta .$ 
\end{thm}

\begin{rem}{\rm Note that when $\dim G=1$, then $\delta^D|_{\tilde{E}_-}=0$ for any left-invariant Levi-Civita generalized connection, 
since $\mathfrak{so} (\tilde{E}_-)=0$. In this case, one can realize as the divergence of a left-invariant Levi-Civita generalized connection 
only the elements $\delta\in \tilde{E}^{*}$ in the annihilator of $\tilde{E}_-$.}
\end{rem}

The remaining part of this section is devoted to the proof of  Theorem~\ref{left-inv-div}.
As the argument is analogous to 
the proof of Proposition 2.16 of \cite{c-k-paper}, we only sketch the main steps.
Consider the setting of  Theorem \ref{left-inv-div}.  From Proposition~\ref{standard-form},  we can assume that
\begin{equation}
E_{-} = \{ X - g(X) \mid X\in TG\} 
\end{equation}
where $g$ is the induced  (left-invariant) pseudo-Riemannian metric on $G$. 
Then 
\begin{equation}
E_{+} = E_{-}^{\perp} =  \{ X + g(X) + a \mid X\in TG,\ a\in \mathbb{R}\}  .
\end{equation}
The subspaces $\tilde{E}_{\pm} \subset \tilde{E}$  defined by $E_{\pm}$ have the same form  as $E_{\pm}$, with $X\in TG$ replaced by $X\in \mathfrak{g}.$ Let $( v_{a})_{1\leq a\leq n}$ be a $g$-orthonormal basis of $\mathfrak{g}$  and $\epsilon_{a}:= g( v_{a}, v_{a})\in 
\{ \pm 1\} .$ Then
\begin{align}
\nonumber& ( e_{0}:= 1,\ e_{a}:= v_{a} + g(v_{a}),\  a\in \{1,\ldots , n\})\\
\nonumber&   ( e_{i} := v_{i-n} - g(v_{i-n}),\  i\in \{n+1,\ldots , 2n\} )
\end{align}
are bases of $\tilde{E}_{+}$ and $\tilde{E}_{-}$ respectively. 
Together, they form a basis 
\begin{equation}\label{frame}
( e_{A} ,\  A\in \{0,\ldots ,  2n\} ) = ( e_{a},\  e_{i},\ 0\leq a\leq n, \ n+1\leq i\leq 2n )
\end{equation}
of $\tilde{E}$.
Unless otherwise stated, capital indices  $A, B, C$, etc.\ will always run from $0$ to $2n$, 
indices $a, b, c, d$   will run from $0$ to $n$ and indices $i, j, k$ will run  from $n+1$ to $2n.$ 
Remark that 
$$
\langle e_{A}, e_{B} \rangle =\epsilon_{A} \delta_{AB},
$$
where $\epsilon_{0} :=1$, $\epsilon_{A} := \epsilon_{a}$ for any $1\leq A =a \leq n$ and $\epsilon_{A} := - \epsilon_{A-n}$ 
for any $A\geq n+1$.
The dual basis of $(e_{A})$ is given by 
$$ 
(e^{A})_{A= 0,\ldots , 2n} =
(e^{a} = \epsilon_{a} \langle e_{a}, \cdot \rangle ,\ e^{i} = -  \epsilon_{i-n} \langle e_{i} ,\cdot \rangle ,\,  
0\leq a\leq n,\ n+1\leq i\leq 2n).
$$
As in \cite{c-k-paper}, 
given a left-invariant generalized connection $D$ 
we define the  coefficients $(\omega_{AB}^{C})$ and $(\omega_{ABC})$  of $D$  in the basis $(e_{A})$  by
\begin{equation}\label{d-c-f}
D_{e_{A}} e_{B} = \sum_C \omega_{AB}^{C}e_{C},\ \omega_{ABC} = \langle D_{e_{A}} e_{B}, e_{C} \rangle = \sum_{D}\omega_{AB}^{D} \eta_{DC}=\omega_{AB}^{C} \epsilon_{C},
\end{equation}
where $\eta_{AB}:= \langle e_{A}, e_{B}\rangle$. We denote by  $(\eta^{AB}) := (\eta_{AB})^{-1}$ the inverse matrix of $(\eta_{AB}).$ 
It coincides with $(\eta_{AB})$,  i.e.\  $\eta^{AB} = \eta_{AB}.$  
We define 
\begin{equation}
B\in \otimes^{3} \tilde{E}^{*},\ 
B(u, v, w) := \langle [u, v]_{H, F}, w\rangle 
\end{equation}
and the {\cmssl Dorfman coefficients  of $E_{H, F}$ in the basis $(v_{A})$}  by
\begin{equation}\label{def-coeff}
B_{ABC}:=  B(e_{A}, e_{B}, e_{C}),
\end{equation}
for any $0\leq A, B, C\leq 2n.$  The same argument as in Lemma  2.9 of \cite{c-k-paper} shows that $B$ is completely skew.
Proposition \ref{zero} below proves  Theorem \ref{left-inv-div} in the case of zero divergence. 
Its proof is completely similar to the proof of   Proposition~2.15  of \cite{c-k-paper} and will be omitted. 
(In relation (\ref{d-0}) below we identify $\tilde{E}$ with $\tilde{E}^{*}$ using $\langle \cdot , \cdot \rangle$).

\begin{prop}\label{zero} The expression
\begin{equation}\label{d-0}
D^{0}:=\frac{1}{3} B\vert_{\Lambda^{3} \tilde{E}_{+}} + \frac{1}{3}  B\vert_{\Lambda^{3} \tilde{E}_{-}} + B\vert_{ \tilde{E}_{+} \otimes \Lambda^{2} \tilde{E}_{-}}  + B\vert_{ \tilde{E}_{-} \otimes \Lambda^{2} \tilde{E}_{+}} 
\end{equation}
defines a left-invariant, divergence-free, Levi-Civita  generalized connection of $E_{-}$ 
with Dorfman coefficients  in the basis $(e_{A})$ given by
\begin{equation}
\omega_{abc} =\frac{1}{3} B_{abc},\ \omega_{ijk} =\frac{1}{3} B_{ijk},\ \omega_{ibc} = B_{ibc},\ \omega_{ajk} = B_{ajk},\
\omega_{Aai} = \omega_{Aia} =0,
\end{equation}
for any $0\leq a, b, c\leq n$,  $i, j, k\geq n+1$ and $0\leq A\leq 2n.$
\end{prop}

To prove the statement of Theorem \ref{left-inv-div} for arbitrary divergence, 
one shows (as in  the proof of Proposition 2.8 of \cite{c-k-paper})   that the  space of left-invariant Levi-Civita generalized connections  of $E_{-}$ is an affine space modeled on  the  generalized first prolongation $\mathfrak{so} (\tilde{E})_{\tilde{E}_-}^{<1>}= \Sigma_{+} \oplus \Sigma_{-}$ of the Lie algebra $\mathfrak{so} (\tilde{E})_{\tilde{E}_-}:  = \{ A \in \mathfrak{so} (\tilde{E}) \mid 
A \tilde{E}_- \subset \tilde{E}_-\} = \mathfrak{so} (\tilde{E}_+)\oplus
\mathfrak{so} (\tilde{E}_-)$, where
\begin{equation}
\Sigma_{\pm}:= \mathrm{Ker}\, \partial_{\pm} : \tilde{E}_{\pm}^{*} \otimes \mathfrak{so} (\tilde{E}_{\pm }) \rightarrow \Lambda^{3} \tilde{E}^{*}, 
\end{equation}
and $\partial_{\pm}$ are defined by 
\begin{equation}
(\partial_{\pm} \alpha )(u, v, w) := \sum_{(u, v, w)} \langle \alpha_{u} (v), w\rangle, 
\end{equation}
where $\sum_{(u, v, w)}$ denotes  sum over cyclic permutations.  Moreover, $\Sigma_{\pm} = \mathrm{Im} ( \mathrm{alt}_{\pm} )$
where 
\begin{equation}
\mathrm{alt}_{\pm} : \mathrm{Sym}^{2} (\tilde{E}_{\pm}^{*} ) \otimes \tilde{E}_{\pm}^*\rightarrow \tilde{E}_{\pm}^{*} \otimes \mathfrak{so} (\tilde{E}_{\pm}) 
\end{equation}
are defined by
\begin{equation}
\langle \mathrm{alt}_{\pm} (\sigma )_{u} v, w\rangle := \sigma (u, v, w) - \sigma (u, w, v),
\end{equation}
for any $\sigma \in \mathrm{Sym}^{2} (\tilde{E}_{\pm}^{*} ) \otimes \tilde{E}_{\pm}^*$ and $u, v, w\in \tilde{E}_{\pm}.$ 
Let $\mathrm{alt}:= \mathrm{alt}_{+} \oplus \mathrm{alt}_{-}.$ 
The next proposition concludes the proof of Theorem \ref{left-inv-div}.

\begin{prop}\label{non-zero} Assume that $n\ge 2$.
Let  $\delta \in \tilde{E}^{*}$ be a left-invariant divergence operator.   Define  $D := D^{0} + S$,  where 
$D^{0}$ is the generalized connection from
Proposition \ref{zero} and
\begin{align}\label{S}
\nonumber S:= &  - \mathrm{alt} \left( \delta_{0} \epsilon_{1} ( e^{1})^{2} \otimes e^{0} +\delta_{n+1} \epsilon_{n+2} (e^{n+2})^{2} \otimes e^{n+1}\right)\\
\nonumber&  - \mathrm{alt} \left( \sum_{A=1}^{n} \delta_{A} (e^{0})^{2} \otimes e^{A} +\sum_{A=n+2}^{2n} \epsilon_{n+1} \delta_{A} (e^{n+1})^{2} \otimes e^{A}\right) ,
\end{align}
where  $\delta_{A}:= \delta (e_{A})$.  Then $D$ 
is a left-invariant Levi-Civita  generalized connection  of $E_{-}$ with divergence $\delta$.
\end{prop}

\begin{proof} 
From Proposition \ref{zero} and the above considerations, it remains to prove that $\delta^{D} = \delta .$ 
For this, it is sufficient to remark that for any $S_{\pm}:=\mathrm{alt}_{\pm}\, (\alpha_{\pm}^{2} \otimes \beta_{\pm})$ 
where $\alpha_{\pm}, \beta_{\pm}\in \tilde{E}_{\pm}^{*}$, the covector $\lambda_{\pm} \in \tilde{E}_{\pm}^{*}$ defined by
$\lambda_{\pm } (v):= \mathrm{tr}\, (S_{\pm} v)$, where
$$
\mathrm{tr}\, (S_{\pm} v):= \mathrm{trace}\, ( \tilde{E}_{\pm} \ni u \mapsto (S_{\pm})_{u} v\in \tilde{E}_{\pm}) , 
$$
is given by 
\begin{equation}\label{contraction}
\lambda_{\pm} = \langle \alpha_{\pm}, \beta_{\pm} \rangle
\alpha_{\pm} - \langle \alpha_{\pm}, \alpha_{\pm}\rangle \beta_{\pm}.
\end{equation} 
In particular, if $\langle \alpha_{\pm}, \beta_{\pm} \rangle =0$ then
\begin{equation}\label{contraction-simpl}
\lambda_{\pm} = -\langle \alpha_{\pm}, \alpha_{\pm} \rangle  \beta_\pm .
\end{equation}
The  relation $\delta^{D} =\delta$ 
follows from $D = D^{0} + S$, 
$\delta^{D^{0}} =0$,  the definition of $S$ and relation
(\ref{contraction-simpl}).  
\end{proof}

\subsection{The Ricci tensors }

Let $E_{-}$ be a generalized metric on a Courant algebroid $E= E_{H, F} $ of type $B_{n}$ over a Lie group $G$, 
with  $H$ and $F$ not necessarily left-invariant. 
Let  $D$  be a Levi-Civita generalized connection 
of $E_{-}.$  Following \cite{mario}, we define 
$$ 
R^{D}(u, v) w:= D_{u} D_{v} w - D_{v} D_{u} w - D_{[u, v]_{H, F}} w,\ \forall u, v, w\in \Gamma (E).
$$
It is well known  (see e.g. \cite{mario}) that  $R^{D}$ restricts to tensor fields
$$
R^{D, \pm} \in \Gamma ( E_{\pm}^{*}\otimes E_{\mp}^{*} \otimes \mathfrak{so}(E_{\pm})).
$$
Let   $\mathrm{Ric}^{\pm }_{D} \in \Gamma ( E_{\mp }^{*}\otimes E_{\pm}^{*})$ defined 
by 
\begin{align}
\nonumber& \mathrm{Ric}^{+}_{D} (u, v) := \mathrm{tr}_{E_{+}} R^{D, +} (\cdot , u)v,\  \forall u\in E_{-},\ v\in E_{+}\\
\label{ricci}& \mathrm{Ric}^{-}_{D} (u, v) := \mathrm{tr}_{E_{-}} R^{D, +} (\cdot , u)v,\  \forall u\in E_{+},\ v\in E_{-},
\end{align}
where $\mathrm{tr}_{E_{\pm}} (A)$ denotes the trace of $A\in \mathrm{End}\, E_{\pm}.$ 
As shown in \cite{mario},  the tensor fields $\mathrm{Ric}^{\pm }_{D} \in \Gamma ( E_{\mp }^{*}\otimes E_{\pm}^{*})$ are independent of the chosen 
Levi-Civita generalized connection $D$,  as long as the divergence $\delta^{D}:= \delta$ is fixed. They are called the 
{\cmssl Ricci tensors of  the pair $(E_{-}, \delta )$} 
and are denoted by $\mathrm{Ric}^{\delta , \pm}.$  The sum 
$$
\mathrm{Ric}^{\delta}:= \mathrm{Ric}^{\delta , +} \oplus \mathrm{Ric}^{\delta , -}\in 
\Gamma (E_{-}^{*}\otimes E_{+}^{*}\oplus E_{+}^{*}\otimes E_{-}^{*}) \subset \Gamma (E^*\otimes E^*)
$$ is called the {\cmssl (total) Ricci tensor}.\

Assume now that  the forms  $H$ and $F$, 
the generalized metric $E_{-} = \{ X - g(X),\ X\in TG\}$ and  the divergence operator 
 $\delta$ are  left-invariant.
Let $(v_{a})_{1\leq a\leq n}$ be a $g$-orthonormal basis of $\mathfrak{g}$, $\epsilon_{a}:= g(v_{a}, v_{a})$, 
and $(e_{A})_{0\leq A\leq 2n}$ the corresponding basis of
$\tilde{E} = \mathfrak{g}\oplus \mathfrak{g}^{*} \oplus \mathbb{R}$, where $\mathfrak{g}:= \mathrm{Lie}\, (G)$,
see  relation (\ref{frame}).   We define the {\cmssl components 
$R_{ia}^{\delta , +}$ and $R_{ai}^{\delta , -}$ 
of the Ricci tensor $\mathrm{Ric}^{\delta}$  of $(E_{-}, \delta )$ in the basis $(e_{A})$}  by 
\begin{align}
\nonumber& R^{\delta , +}_{ia} := \mathrm{Ric}^{+}_{D} ( e_{i}, e_{a}) =  \mathrm{Ric}^{\delta , +} ( e_{i}, e_{a})\\
\label{comp-ricci-gen}& R^{\delta , -} _{ai} := \mathrm{Ric}^{-}_{D} ( e_{a}, e_{i}) =  \mathrm{Ric}^{\delta , -} ( e_{a}, e_{i}),
\end{align}
for any $0\leq a\leq n$ and $n+1\leq i\leq 2n.$\

The next proposition is the analogue of Theorem~2.25 of \cite{c-k-paper} and can be proved by a similar argument. For this reason, 
we only sketch its proof (for more details, see \cite{c-k-paper}).  From now on we will use Einstein's convention, such that summation over repeated lower and upper indices is usually understood. Nonetheless we may sometimes use the summation sign, in particular,  to indicate whether the summation over the indices $a, b, c$, ... and $A, B, C$, ... starts from $0$ or from $1$.

\begin{prop}\label{ricci-expr} 
For  any $0\leq a\leq n$ and $n+1\leq i\leq 2n$, 
\begin{align}
\nonumber& \mathrm{R}^{\delta , +} _{ia} = \sum_{b=0}^{n} ( B_{bi}^{j} B_{aj}^{b} + B_{ia}^{b} \delta_{b}) \\
\label{ricci-delta-gen}&\mathrm{R}^{\delta , -}_{ai} =  \sum_{b=0}^{n} B_{bi}^{j} B_{aj}^{b} +  B_{ai}^{j} \delta_{j},
\end{align}
where  $\delta_{A}:= \delta (e_{A})$ and   in the above sums $j$ ranges from $n+1$ to $2n$ and $B_{AB}^{C}$ are defined by
$$
B_{AB}^{C} := \sum_{D=0}^{2n}B_{ABD} \eta^{DC}= B_{ABC}\epsilon_{C}.
$$
\end{prop}

\begin{proof}  
For a Levi-Civita generalized connection $D$ we define 
$$
(R^{D} )_{ABCD} := \langle R^{D} (e_{A}, e_{B}) e_{C}, e_{D}\rangle .
$$
A computation as in Proposition~2.19 of    \cite{c-k-paper} shows that 
\begin{align}
R^{D^{0}}_{ajcd} = \frac{2}{3} B_{aj}^{l} B_{cld} +\frac{1}{3} B_{jc}^{l} B_{lad}+\frac{1}{3} B_{ca}^{l} B_{ljd}\\
R^{D^{0}}_{ibkl} =\frac{2}{3} B_{ib}^{c} B_{kcl} + \frac{1}{3} B_{bk}^{c} B_{cil} +\frac{1}{3} B_{ki}^{c} B_{cbl} ,
\end{align}
where $D^{0}$ is the Levi-Civita generalized connection from  Proposition \ref{zero}, and, as usual,
$0\leq a,b, c, d\leq n$ and $ n+1\leq i, j, k, l\leq 2n.$
Taking traces, we obtain 
relations (\ref{ricci-delta-gen}) with $\delta =0.$ In order to obtain (\ref{ricci-delta-gen}) for any $\delta$, we 
consider the connection $D= D^{0} + S$ from Proposition \ref{non-zero}.  As in   Lemma~2.22  and Lemma~2.23 of  \cite{c-k-paper}, we obtain 
\begin{equation}
R^{D} (u, v) w = R^{D^{0}} (u, v)w - (D^{0}_{v} S)_{u} w,\ \forall u, w\in \tilde{E}_{\pm},\ v\in \tilde{E}_{\mp}
\end{equation}
and, taking traces,
\begin{align*}
\nonumber& \mathrm{Ric}^{+}_{D} (e_{i}, e_{a}) = \mathrm{Ric}^{+}_{D^{0}} (e_{i}, e_{a}) -\sum_{b=0}^{n} \epsilon_{b} \langle (D^{0}_{e_{i}} S)_{e_{b}} (e_{a}), e_{b}
\rangle\\
\nonumber& \mathrm{Ric}^{-}_{D} (e_{a}, e_{i}) = \mathrm{Ric}^{-}_{D^{0}} (e_{a}, e_{i}) -\sum_{j={n+1}}^{2n}\epsilon_{j} \langle (D^{0}_{e_{a}} S)_{e_{j}} (e_{i}), e_{j}
\rangle .
\end{align*}
Using 
\begin{align}
\nonumber&\sum_{b=0}^{n}  \epsilon_{b} \langle ( D^{0}_{e_{i}} S)_{e_{b}} (e_{a}), e_{b} \rangle = (D^{0}_{e_{i}} \delta )(e_{a})
= - \sum_{b=0}^{n} B_{ia}^{b} \delta_{b} \\
\label{vic}&\sum_{j=n+1}^{2n}  \epsilon_{j}  \langle (D^{0}_{e_{a}}S)_{e_{j}} e_{i}, e_{j}\rangle  = ( D^{0}_{e_{a}}\delta ) (e_{i})=-
\sum_{j=n+1}^{2n} B_{ai}^{j} \delta_{j} 
\end{align}
we obtain (\ref{ricci-delta-gen}) as required.
\end{proof}

Let
$$
F_{ab} := F(v_{a}, v_{b}),\ H_{abc}:= H( v_{a}, v_{b}, v_{c}),
$$
for any $ 1\leq a, b, c\leq n$ and recall that $\delta_{A}= \delta (e_{A})$ for any $0\leq A\leq 2n$. 
Our next aim is to compute 
the components $R_{ia}^{\delta  , +}$ and $R_{ai}^{\delta , -}$  of the Ricci tensor $\mathrm{Ricci}^{\delta}$ in terms of  $H_{abc}$,  $F_{ab}$, $\delta_{A}$,  and the Lie algebra coefficients
$k_{abc}$ defined by
$$
 k_{abc} := g (  {\mathcal L}_{v_{a}} v_{b}, v_{c}),\  1\leq a, b, c\leq n.
 $$
We start with
the Dorfman coefficients  defined in (\ref{def-coeff}). 
\begin{lem}\label{1}
The Dorfman coefficients 
are given by
\begin{align}
\nonumber& B_{abc} = \frac{1}{2} ( \sum_{(a,b,c)} k_{abc} - H_{abc})\\
\nonumber& B_{ijk} = - \frac{1}{2} ( \sum_{ (i^{\prime} j^{\prime} k^{\prime})} k_{i^{\prime} j^{\prime} k^{\prime}} +  H_{i^{\prime} j^{\prime} k^{\prime}})\\
\nonumber& B_{ajk} = \frac{1}{2} ( k_{a k^{\prime} j^{\prime}} + k_{j^{\prime}k^{\prime} a} - k_{aj^{\prime} k^{\prime}} - H_{aj^{\prime} k^{\prime}})\\
\nonumber& B_{ibc} = \frac{1}{2} ( k_{ci^{\prime} b} + k_{cb i^{\prime}} + k_{i^{\prime} bc} - H_{i^{\prime} bc})\\
\label{ec-dorfman}&   B_{0bc} = F_{bc},\ B_{0jk} = F_{j^{\prime}k^{\prime}},\ B_{ib0} = F_{i^{\prime}b} , 
\end{align}
where 
$1\leq a, b, c\leq n$, $n+1 \leq i, j, k\leq 2n$,  $i^{\prime} := i -n$ (similarly for $j^{\prime}$ and $k^{\prime}$).
 \end{lem}

\begin{proof}
The proof is straightforward from the definition  (\ref{def-dorfman}) of the Dorfman bracket. 
For example, to compute  $B_{0bc}$  we  observe that
$$
[ 1, X]_{H, F} = - [X, 1]_{H, F} = 2 i_{X}F,\ [1, \eta ]_{H, F} = [\eta , 1]_{H, F} =0
$$
for any $X\in {\mathfrak X}(G)$ and $\eta \in \Omega^{1}(G)$. Then
\begin{align*}
\nonumber B_{0bc} &= \langle [ e_{0}, e_{b}]_{H, F}, e_{c} \rangle =\langle [1, v_{b} + g( v_{b})]_{H, F} , v_{c} + g( v_{c})\rangle \\
\nonumber&  =  2 \langle i_{v_{b}} F, v_{c} + g(v_{c}) \rangle = F_{bc}. \qedhere
\end{align*}
\end{proof}

\begin{thm}\label{2-prime} 
The  components of the Ricci tensor $\mathrm{Ric}^{\delta }$ of the pair  $(E_{-}, \delta )$ in the basis  $(e_{A})_{0\leq A\leq 2n}$ 
are given by 
\begin{align}
\nonumber& R^{\delta, + }_{ia} = \sum_{b=1}^{n} ( B_{bi}^{j} B_{aj}^{b} + B_{ia}^{b}\delta_{b}) + F_{i^{\prime} a} \delta_{0} -\sum_{b=1}^{n} \epsilon_{b} F_{i^{\prime} b} F_{ab}\\\nonumber& R^{\delta , -}_{ai} = \sum_{b=1}^{n} B_{bi}^{j} B_{aj}^{b} + B_{ai}^{j} \delta_{j} 
-\sum_{b=1}^{n} \epsilon_{b} F_{i^{\prime} b} F_{ab}\\
\nonumber& R^{\delta , +}_{i0}= - \sum_{b=1}^{n}  \sum_{j=n+1}^{2n}\epsilon_{j^{\prime}} \epsilon_{b} F_{j^{\prime} b} B_{bij} - \sum_{b=1}^{n} F_{i^{\prime} b} \epsilon_{b} \delta_{b}\\
\label{expresii-d-t}& R^{\delta , -}_{0i} = -\sum_{b=1}^{n} \sum_{j=n+1}^{2n}\epsilon_{j^{\prime}} \epsilon_{b} F_{j^{\prime} b} B_{b ij} - 
\sum_{j=n+1}^{2n} F_{i^{\prime} j^{\prime}} \epsilon_{j^{\prime}} \delta_{j}
\end{align}
where  
$1\leq a\leq n$, $n+1 \leq i \leq 2n$, and in the above sums 
$j$ ranges from  $n+1$ to $2n$.
\end{thm}

\begin{proof}
From Proposition \ref{ricci-expr}, for any $0\leq a\leq n$ and $n+1\leq i\leq 2n$,
\begin{equation}\label{ricci-new}
 R^{\delta , +} _{ia} = \sum_{b=1}^{n} B_{bi}^{j} B_{aj}^{b} + B_{0i}^{j} B_{aj}^{0} 
 + \sum_{b=1}^{n} B_{ia}^{b} \delta_{b}+ B_{ia}^{0} \delta_{0}. 
\end{equation}
Consider first  $1\leq a\leq  n.$
Using Lemma \ref{1}, $B_{AB}^{C} = B_{ABC} \epsilon_{C}$ and 
$\epsilon_{C}= \epsilon_{c}$ for any $1\leq  C= c\leq n$, $ \epsilon_{C} = - \epsilon_{C-n}$ for any $n+1\leq C\leq 2n$ and $\epsilon_{0}=1$,  we obtain
\begin{equation}\label{obtain}
B_{0i}^{j} = -\epsilon_{j^{\prime}} F_{i^{\prime} j^{\prime}},\ B_{aj}^{0} = - F_{j^{\prime} a},\ B_{ia}^{0} = F_{i^{\prime} a}. 
\end{equation}
The first  relation (\ref{expresii-d-t})  follows from relation (\ref{ricci-new}) and relations (\ref{obtain}).   The second relation (\ref{expresii-d-t}) can be proved in a similar way. 
In order to compute $R_{i0}^{\delta , +}$ we use  again relation (\ref{ricci-new}), together with $B_{0i}^{0}  = B_{i0}^{0}=0$ (because $B_{ABC}$ is completely skew). From  relation (\ref{ricci-new}) with $a=0$ (and $n+1\leq i\leq 2n$), 
$$
R_{i0}^{\delta , +} = \sum_{b=1}^{n} ( B_{bi}^{j} B_{0j}^{b} + B_{i0}^{b} \delta_{b} ) = \sum_{b=1}^{n}  ( -  \epsilon_{j^{\prime}} \epsilon_{b} B_{bij} B_{0jb} + 
\epsilon_{b} B_{i0b}  \delta_{b} ).
$$
Using $B_{0jb} = B_{jb0} = F_{j^{\prime} b}$ we obtain the third relation (\ref{expresii-d-t}). The last relation
(\ref{expresii-d-t}) can be obtained similarly. 
\end{proof}

\begin{rem}{\rm 
The notions of generalized metrics, generalized connections, Levi-Civita generalized connections,  divergence operators  
and  Ricci tensors 
can be defined, in a  completely similar way,   on arbitrary Courant
algebroids (not necessarily of type $B_{n}$)  and some of the results presented above hold in this general setting, see  e.g.\  \cite{mario}.  More precisely, given a generalized metric  $E_{-}$ and a divergence operator  $\delta$ on an arbitrary Courant algebroid  $E=E_+ \oplus E_-$, such that $E_+$ and $E_-$ are at least of rank $2$,
there is always a Levi-Civita generalized connection $D$ of $E_{-}$ with divergence $\delta$. The tensor fields    
$\mathrm{Ric}_{D}^\pm$, defined  by (\ref{ricci}), are  independent of the choice of $D$.  They are called 
the Ricci tensors of the pair $(E_{-}, \delta )$ and are denoted by $\mathrm{Ric}^{\delta ,\pm}$.
Below we apply these considerations on generalized tangent bundles.  Our aim will be to relate the Ricci tensors 
on   Courant algebroids of type $B_{n}$  with closed twisting $3$-form $H$  to the Ricci tensors  
on generalized tangent bundles  $E_{H}$ (see below)}. 
\end{rem}

In the remaining part of this section we only consider Courant algebroids  $E= E_{H, F}$ of type $B_{n}$ for which 
$dH =0$ (and  $F\wedge F =0$,  $d F =0$). 
Since $H$ is closed,  the generalized tangent bundle $E_{H}:= \mathbb{T}G = TG\oplus T^{*}G$ with 
Dorfman bracket  
\begin{equation}\label{dorfman-here}
[ X + \xi  , Y +\eta ]_{H}:= {\mathcal L}_{X} ( Y+\eta ) - i_{Y} d\xi  + i_{X} i_{Y} H,
\end{equation}
for any $X, Y\in {\mathfrak X}(G)$ and $\xi , \eta \in \Omega^{1}(G)$, 
anchor the natural projection $\mathbb{T}G \rightarrow TG$ and  scalar product 
$$
\langle X +\xi , Y+\eta \rangle_{\mathbb{T}G} = \frac{1}{2} ( \xi (Y) + \eta (X))
$$
is a Courant algebroid. A left-invariant  generalized  metric $E_{-}= \{ X - g(X)\mid X\in TG\}$ on $E$ 
can be  also considered as a generalized metric on $E_{H}$.
Let $\delta \in \tilde{E}^{*}$ be a left-invariant divergence operator,  
$\tilde{\delta}:= \delta\vert_{\mathfrak{g}\oplus \mathfrak{g}^{*}}$ 
and $\mathrm{Ric}^{\tilde{\delta}, \pm}$ the Ricci tensors of the pair $(E_{-}, \tilde{\delta} )$   on $E_{H}.$ 
Let
$(v_{a})$ be a $g$-orthonormal basis of $\mathfrak{g}$ and 
$R^{\tilde{\delta}, +}_{ia}$, 
$R_{ai}^{\tilde{\delta}, -}$ 
(with $1\leq a\leq n$ and $n+1\leq i\leq 2n$) 
be the components  of the Ricci tensor $\mathrm{Ric}^{\tilde{\delta }}$ 
in the basis $(e_{A})_{1\leq A\leq 2n}$ of $\mathfrak{g}\oplus \mathfrak{g}^{*}$, defined as in 
(\ref{comp-ricci-gen}).

\begin{lem}\label{ricci-diferit}  Assume that $H$ is  closed. Then for  any $1\leq a\leq n$ and $n+1\leq i\leq 2n$, 
\begin{align}
\nonumber& R_{ia}^{\tilde{\delta}, +}=   \sum_{b=1}^{n}   ( B_{bi}^{j} B_{aj}^{b}  + B_{ia}^{b} \delta_{b} )    \\
\nonumber& R_{ai}^{\tilde{\delta}, - } = \sum_{b=1}^{n}  B_{bi}^{j} B_{aj}^{b}  + B_{ai}^{j} \delta_{j}.  
\end{align}
\end{lem}

\begin{proof}
Let  $\mathcal B_{ABC}:= \langle [ e_{A}, e_{B}]_{H}, e_{C}\rangle_{\mathbb{T}G}$ 
be the Dorfman coefficients of $ E_{H}$ in the basis $(e_{A})_{1\leq A\leq 2n}$ of $\mathfrak{g}\oplus \mathfrak{g}^{*}$   
and $\mathcal B_{AB}^{C} : = \sum_{D=1}^{2n}\mathcal B_{ABD}\tilde{\eta}^{DC}$, where $(\tilde{\eta}^{AB} )$ is the inverse matrix of  $(\tilde{\eta}_{AB})$
and $\tilde{\eta}_{AB}  := \langle e_{A}, e_{B}\rangle_{\mathbb{T}G}$ for any  $1\leq A, B\leq 2n$.  
Let  $ \mathrm{pr}_{\mathfrak{g}\oplus \mathfrak{g}^{*}} : \tilde{E} =\mathfrak{g}\oplus \mathfrak{g}^{*}\oplus \mathbb{R} \rightarrow  \mathfrak{g}\oplus \mathfrak{g}^{*}$  be the natural projection. 
From 
$$
\mathrm{pr}_{\mathfrak{g}\oplus \mathfrak{g}^{*}}  [ u, v]_{H, F}= [u, v]_{H},\ 
\langle u, v\rangle_{\mathbb{T}G}  = \langle u, v\rangle ,\  \langle u , 1\rangle =0, 
$$
for all $u, v\in \mathfrak{g}\oplus \mathfrak{g}^{*}$, 
we obtain that 
\begin{equation}\label{ABC-tang}
B_{ABC} = {\mathcal B}_{ABC},\  B_{AB}^{C} =\mathcal B_{AB}^{C},\ 
1\leq A,  B,  C\leq 2n.
\end{equation}
The claim follows from   Theorem 2.25 of \cite{c-k-paper} combined with
(\ref{ABC-tang}).
\end{proof}

\begin{rem}\label{compare}{\rm 
There is a  conventional difference of sign  (in the  $3$-form $H$),  between   the Dorfman bracket (\ref{dorfman-here}) and the Dorfman bracket   considered in 
\cite{c-k-paper}. 
Owing to this,  in the next sections, where we will use 
various expressions (for $\mathcal B_{ABC}$, $R_{ia}^{\tilde{\delta}, +}$ or $R_{ai}^{\tilde{\delta}, -}$)   computed in   \cite{c-k-paper}, 
we will have to replace $H$ with $-H$.}
\end{rem}

\begin{cor}\label{2-prime-cor} Let $E_{H, F}$ be a Courant algebroid of type $B_{n}$ over a Lie group $G$
such that $dH =0$. Let 
$E_{-} = \{ X - g(X) \mid X\in TG\}$ be a  left-invariant generalized metric  on $E_{H, F}$  and 
$(e_{A})_{0\leq A\leq n}$ the  basis of  $\tilde{E} =\mathfrak{g}\oplus \mathfrak{g}^{*}\oplus \mathbb{R}$ induced  by a $g$-orthonormal basis
$(v_{a})$ of $\mathfrak{g}$, see  (\ref{frame}). 
Let 
$\delta \in \tilde{E}^{*}$ be a left-invariant divergence operator and  $\tilde{\delta}:= \delta\vert_{\mathfrak{g}\oplus \mathfrak{g}^{*}}$.
Then
\begin{align}
\nonumber& R^{\delta, + }_{ia} = R^{\tilde{\delta}, +}_{ia} + F_{i^{\prime} a} \delta_{0} -\sum_{b=1}^{n} \epsilon_{b} F_{i^{\prime} b} F_{ab}\\
\nonumber& R^{\delta , -}_{ai} = R^{\tilde{\delta}, -}_{ai} 
-\sum_{b=1}^{n}  \epsilon_{b} F_{i^{\prime} b} F_{ab}\\
\nonumber& R^{\delta , +}_{i0}= - \sum_{b=1}^{n}   \sum_{j=n+1}^{2n}\epsilon_{j^{\prime}} \epsilon_{b} F_{j^{\prime} b} B_{bij} - \sum_{b=1}^{n} F_{i^{\prime} b} \epsilon_{b} \delta_{b}\\
& R^{\delta , -}_{0i} = -\sum_{b=1}^{n}  \sum_{j=n+1}^{2n}\epsilon_{j^{\prime}} \epsilon_{b} F_{j^{\prime} b} B_{b ij} -  \sum_{j=n+1}^{2n}F_{i^{\prime} j^{\prime}} \epsilon_{j^{\prime}} \delta_{j},
\end{align}
where   $R^{\tilde{\delta}, +}_{ia} $ and $R^{\tilde{\delta}, -}_{ai} $ are the components of the Ricci tensor of   the pair $(E_{-}, \tilde{\delta })$ defined on $E_{H}$, 
$1\leq a\leq n$, $n+1 \leq i \leq 2n$, and 
the sums are over $1\leq b\leq n$ and $n+1\leq j\leq 2n$.
\end{cor}

\begin{proof}
The claim follows from Theorem \ref{2-prime} combined with Lemma  \ref{ricci-diferit}.
\end{proof}

\subsection{The generalized Einstein equation}

Recall that a  generalized metric $E_{-}$  on a Courant  algebroid  is  called  {\cmssl generalized Einstein  with divergence $\delta$} if the total Ricci tensor  of the pair
$(E_{-}, \delta )$ vanishes, i.e. 
$\mathrm{Ric}^{\delta , \pm } =0$ (see e.g. \cite{mario}).

\begin{defn} An {\cmssl odd generalized  Einstein metric} on a Lie group $G$ is a generalized Einstein metric $E_{-}$ 
on a Courant algebroid  $E_{H, F}$ of type $B_{n}$ over $G$, such that $H$, $F$, $E_{-}$  and the divergence  $\delta$ of $E_{-}$ are all 
left-invariant.
\end{defn}

\begin{cor}\label{3} A left-invariant generalized  metric $E_{-} = \{ X - g(X) \mid X\in TG\}$
defined on a Courant algebroid $E_{H, F}$ of type $B_{n}$ over a Lie group $G$  
 is  odd generalized Einstein with divergence $\delta$ if and only if
the following relations hold: for any $1\leq a \leq n$ and $n+1\leq i\leq 2n$, 
\begin{align}
\nonumber  & \sum_{b=1}^{n} (B_{bi}^{j} B_{aj}^{b} + B_{ia}^{b} \delta_{b})  
+ F_{i^{\prime} a} \delta_{0} - \sum_{b=1}^{n}  \epsilon_{b} F_{i^{\prime}b} F_{ab} =0\\
\nonumber& \sum_{b=1}^{n}  \epsilon_{b} ( \epsilon_{j^{\prime} }F_{j^{\prime} b} B_{bij} + F_{i^{\prime}b} \delta_{b}) =0\\
\nonumber&  \sum_{b=1}^{n}  F_{ab} \epsilon_{b} (\delta_{b} - \delta_{b+n} ) =0\\
\label{rel-ricci}& \sum_{b=1}^{n}\epsilon_{b} ( B_{iab}\delta_{b} - B_{ia\, b+n}\delta_{b+n} ) +\delta_{0} F_{i^{\prime}a} =0,
\end{align}
where in the above sums $n+1 \leq j \leq 2n.$ 
\end{cor}

\begin{proof}
The claim follows from 
Theorem \ref{2-prime}. 
The first  two relations (\ref{rel-ricci}) coincide with  $R_{ia}^{\delta , +} =0$ and
$R_{i0}^{\delta , +}=0$ for any $1\leq a\leq n$ and $n+1\leq i\leq 2n.$  
The last two relations (\ref{rel-ricci}) follow by noticing that 
\begin{align}
\nonumber& R_{i0}^{\delta , +} - R_{0i}^{\delta , -} =   \sum_{b=1}^{n}F_{i^{\prime} b} \epsilon_{b} ( \delta_{b+n} - \delta_{b})\\
\nonumber&  R_{ia}^{\delta , + } - R_{ai}^{\delta , -} = \sum_{b=1}^{n} \epsilon_{b} ( B_{iab} \delta_{b} - B_{ia\, b+n} \delta_{b+n}  ) +  \delta_{0} F_{i^{\prime} a}.\qedhere
\end{align}
\end{proof}

\begin{rem}\label{F}{\rm   i)  When  the $3$-form $H$ is closed  we can replace the term $ \sum_{b=1}^{n} (B_{bi}^{j} B_{aj}^{b} + B_{ia}^{b} \delta_{b})$
from the first relation (\ref{rel-ricci}) with $R^{\tilde{\delta}, +}_{ia}$ (see Lemma \ref{ricci-diferit}).\

ii) When  $F=0$,   the $3$-form $H$ is closed  and we obtain that
$E_{-}$ is an odd  generalized Einstein metric with divergence $\delta$ if and only if, 
 considered as  a generalized metric on  the Courant algebroid  $E_{H}$,  it
 is ordinary generalized Einstein with divergence $\tilde{\delta }=\delta\vert_{\mathfrak{g}\oplus \mathfrak{g}^{*}}.$ 
 This is straightforward  from  Corollary~\ref{2-prime-cor}.  Ordinary generalized Einstein metrics on $3$-dimensional Lie groups 
 were classified in \cite{c-k-paper}.}
\end{rem}

\section{The $3$-dimensional case }\label{three-dim}

Assume now that $G$ is a $3$-dimensional Lie group and let 
$E_{H, F}$ be  a  Courant algebroid of type $B_{3}$ over $G$. Then 
$dH =0$,  $F\wedge F = 0$ and the constraints on the pair $(H, F)$ reduce to $d F=0.$ 
The next corollary will be used systematically in the next sections. 

\begin{cor}\label{forapplic}  Let $E_{-} = \{ X - g(X)\mid X\in TG \}$ be a left-invariant generalized metric defined on a  
Courant algebroid $E_{H, F}$ of type $B_{3}$ over a $3$-dimensional Lie group 
$G$ with Lie algebra $\mathfrak{g}.$  Let $(v_{a})$ be a $g$-orthonormal basis of $\mathfrak{g}$, $\epsilon_{a} := g(v_{a}, v_{a} )\in \{ \pm 1\}$ 
and $(e_{A})$ the corresponding basis  of
$\tilde{E} = \mathfrak{g}\oplus \mathfrak{g}^{*} \oplus \mathbb{R}$, see (\ref{frame}).  Let $\delta\in \tilde{E}^{*}$ and $\tilde{\delta} = \delta\vert_{ \mathfrak{g} \oplus \mathfrak{g}^{*}}$.  Then  $E_{-}$ is odd generalized Einstein
with divergence $\delta$ 
 if and only if the components of the Ricci tensor $\mathrm{Ric}^{\tilde{\delta }, +}$ of the  pair  $(E_{-}, \tilde{\delta }) $ defined on $E_{H}$, 
 in the basis $(e_{A})_{1\leq A\leq 6}$ of $\mathfrak{g}\oplus \mathfrak{g}^{*}$, 
 satisfy
\begin{align}
\nonumber& R^{\tilde{\delta}, +}_{41} = \epsilon_{2} (F_{12})^{2} +\epsilon_{3} (F_{13})^{2},\ R^{\tilde{\delta}, +}_{51} =\delta_{0} F_{12} + \epsilon_{3} F_{23} F_{13},\
R^{\tilde{\delta}, +}_{61} = \delta_{0} F_{13} +\epsilon_{2} F_{32} F_{12}\\
\nonumber& R^{\tilde{\delta}, +}_{42}  = - \delta_{0} F_{12} +\epsilon_{3} F_{13} F_{23},\ R^{\tilde{\delta}, +}_{52}  = \epsilon_{1} (F_{12})^{2}  +\epsilon_{3} (F_{23})^{2},\ 
R^{\tilde{\delta}, +}_{62} = \delta_{0} F_{23}  +\epsilon_{1}  F_{12} F_{13}\\
\label{rel-1}  &  R^{\tilde{\delta}, +}_{43}= - \delta_{0} F_{13} + \epsilon_{2} F_{12} F_{32},\  R^{\tilde{\delta}, +}_{53}= - \delta_{0} F_{23} + \epsilon_{1} F_{21} F_{31},\ 
 R^{\tilde{\delta}, +}_{63}= \epsilon_{1}  (F_{13} )^{2} + \epsilon_{2}  (F_{23})^{2},
  \end{align}
  the components of the $2$-form $F$ satisfy 
  \begin{align}
  \nonumber& F_{12} \epsilon_{2} (\delta_{2} - \epsilon_{1} B_{145} ) + F_{13} \epsilon_{3} (\delta_{3} - \epsilon_{1} B_{146}) + F_{23} \epsilon_{2} \epsilon_{3}
  ( B_{345} - B_{246}) =0\\
  \nonumber& F_{12} \epsilon_{1} (\epsilon_{2} B_{254} - \delta_{1}) +\epsilon_{1}\epsilon_{3} F_{13} ( B_{354} - B_{156}) + F_{23}\epsilon_{3} (\delta_{3} - \epsilon_{2} B_{256}) =0\\
  \nonumber& F_{12} \epsilon_{1} \epsilon_{2} (B_{264} - B_{165}) +F_{13} \epsilon_{1} (\epsilon_{3} B_{364} -\delta_{1}) +F_{23} \epsilon_{2} (\epsilon_{3} B_{365} -\delta_{2} ) =0\\
  \nonumber& F_{12} \epsilon_{2} (\delta_{2} -\delta_{5}) + F_{13} \epsilon_{3} (\delta_{3} - \delta_{6})=0\\
  \nonumber& F_{21} \epsilon_{1} (\delta_{1} -\delta_{4}) +F_{23} \epsilon_{3} (\delta_{3} -\delta_{6}) =0\\
  \label{rel-2}& F_{31} \epsilon_{1} (\delta_{1} -\delta_{4}) +F_{32} \epsilon_{2} ( \delta_{2} - \delta_{5}) =0,
  \end{align}
  the products $\delta_{0} F_{ab}$  are given by 
   \begin{align}
  \nonumber& \delta_{0} F_{12}=  \epsilon_{2} B_{512} \delta_{2} +\epsilon_{3} B_{513} \delta_{3} - \epsilon_{1} B_{514} \delta_{4} - \epsilon_{3} B_{516} \delta_{6} \\
  \nonumber&  \delta_{0} F_{13}= \epsilon_{2} B_{612} \delta_{2} +\epsilon_{3} B_{613} \delta_{3} -\epsilon_{1} B_{614} \delta_{4} -\epsilon_{2} B_{615}\delta_{5} \\
  \label{rel-3}& \delta_{0} F_{23}= \epsilon_{1} B_{621} \delta_{1} +\epsilon_{3} B_{623} \delta_{3} - \epsilon_{1} B_{624} \delta_{4} - \epsilon_{2} B_{625}\delta_{5} 
  \end{align}
 and the following additional relations hold: 
  \begin{align}
 \nonumber&  \epsilon_{2} B_{412} \delta_{2} +\epsilon_{3} B_{413} \delta_{3} - \epsilon_{2} B_{415} \delta_{5} -\epsilon_{3} B_{416} \delta_{6} =0\\
  \nonumber& \epsilon_{1} B_{521} \delta_{1} +\epsilon_{3} B_{523}\delta_{3} - \epsilon_{1} B_{524}\delta_{4} -\epsilon_{3} B_{526}\delta_{6} =0\\
  \nonumber& \epsilon_{1} B_{631}\delta_{1} +\epsilon_{2} B_{632}\delta_{2} - \epsilon_{1} B_{634} \delta_{4} - \epsilon_{2} B_{635} \delta_{5} =0\\
  \nonumber& \epsilon_{1} B_{421}\delta_{1} +\epsilon_{2} B_{512} \delta_{2} +\epsilon_{3} ( B_{513} +  B_{423}) \delta_{3} - \epsilon_{1} B_{514} \delta_{4}
  -\epsilon_{2} B_{425}\delta_{5}\\
  \nonumber&  - \epsilon_{3} (B_{516} + B_{426})\delta_{6} =0\\
 \nonumber&  \epsilon_{1}  B_{431} \delta_{1} +\epsilon_{2} (B_{612} + B_{432}) \delta_{2} +\epsilon_{3} B_{613}\delta_{3} - \epsilon_{1} B_{614}\delta_{4}
 - \epsilon_{2} (B_{615} + B_{435}) \delta_{5}\\
 \nonumber&  - \epsilon_{3} B_{436} \delta_{6} =0\\
 \nonumber& \epsilon_{1} (B_{621} + B_{531} )\delta_{1} +\epsilon_{2} B_{532}\delta_{2} +\epsilon_{3} B_{623}\delta_{3} - \epsilon_{1} ( B_{624} + B_{534}) \delta_{4}
 - \epsilon_{2} B_{625} \delta_{5}\\
 \label{rel-4}&  - \epsilon_{3} B_{536}\delta_{6} =0.
  \end{align}
\end{cor}

\begin{proof}  The claim follows from Corollary \ref{3}.  
The first  relation (\ref{rel-ricci}) coincides with
relations (\ref{rel-1}), the second relation (\ref{rel-ricci}) coincides with the first three relations 
 (\ref{rel-2}),
 the third relation (\ref{rel-ricci}) coincides with the last three relations
 (\ref{rel-2}). The last relation (\ref{rel-ricci}) coincides with 
 relations (\ref{rel-3}) and (\ref{rel-4}).
 \end{proof}

 Our goal in the next sections is to describe all 
 odd generalized Einstein metrics on $3$-dimensional Lie groups.

\begin{ass}\label{ass-unimod}{\rm  As  ordinary generalized Einstein metrics over $3$-dimensional Lie groups were completely 
classified in \cite{c-k-paper},  in view of  Remark \ref{F} ii) 
we will assume from now   on (without repeating it each time)  that 
$$
F\neq 0.
$$}\end{ass}

\section{The classification when $\mathfrak{g}$ is unimodular}\label{unimod-section}

\subsection{Basic facts on $3$-dimensional   unimodular metric Lie algebras}
\label{basic_facts_unimod:sec}

Let $\mathfrak{g}$ be a $3$-dimensional  Lie algebra with a non-degenerate scalar product  $g$. 
We start by recalling several facts on the pair $(\mathfrak{g}, g).$ 
For details, see \cite{c-k-paper}. Let   $\underline{\mathrm{vol}}=\mathbb{R}_{>0}\cdot \mathrm{vol}\subset \Lambda^{3} \mathfrak{g}^{*}$, $\mathrm{vol}\in  (\Lambda^{3} \mathfrak{g}^{*})\setminus \{ 0\}$,  be an orientation 
of $\mathfrak{g}$ 
and   $L\in \mathrm{End}\, \mathfrak{g}$  the operator  which relates the cross product defined by $(g, \underline{\mathrm{vol}})$  with 
the Lie bracket of $\mathfrak{g}$, 
namely 
$$
{\mathcal L}_{u} v= L( u\times v),\  \forall u, v\in \mathfrak{g}.
$$
The operator $L$ was first considered in \cite{milnor} when $g$ is positive definite  and  later in \cite{c-k-paper} for arbitrary signature.

\begin{rem}{\rm  Changing   the orientation  the operator $L$ multiplies  by a  minus sign. By abuse of language we will refer to $L$ 
as {\it the}  operator associated $(\mathfrak{g}, g)$ rather than $(\mathfrak{g}, g, \underline{\mathrm{vol}}).$
}
\end{rem}

As proved in  Proposition 3.2 of \cite{c-k-paper}, 
$\mathfrak{g}$ is unimodular (i.e.\  $\mathrm{tr}\, \mathrm{ad}_{u} =0$ for any $u\in \mathfrak{g}$) 
if and only if the operator $L$ is $g$-symmetric. In this case,
there is a $g$-orthonormal basis $(v_{a})$ of $\mathfrak{g}$ such that  $L$ takes one the following    canonical forms 
in this basis: 
$$
L_{1}(\alpha , \beta , \gamma ) = \left(\begin{tabular}{ccc}
$\alpha $ &  $0$ & $0$\\
$0$ & $\beta$ & $0$\\
$0$ & $ 0$ & $\gamma$
\end{tabular}\right),\  L_{2} (\alpha , \beta , \gamma )= \left(\begin{tabular}{ccc}
$\gamma $ & $0$ & $ 0$\\
$0$ & $ \alpha $ & $ -\beta$\\
$0$ & $ \beta $ & $\alpha $
\end{tabular}\right) 
$$
and 
$$ 
L_{3, \eta } (\alpha , \beta   )= \left( \begin{tabular}{ccc}
$ \beta$ & $0$ & $0$\\
$0$ & $\frac{\eta}{2} +\alpha $ & $ \frac{\eta}{2}$\\
$ 0$ & $ -\frac{\eta}{2} $ & $ -\frac{\eta}{2} +\alpha$
\end{tabular}\right),\ L_{5}(\alpha )=\left( \begin{tabular}{ccc}
$ \alpha$ & $\frac{1}{\sqrt{2}}$  & $0$\\
$\frac{1}{\sqrt{2}}$ & $\alpha$ & $\frac{1}{\sqrt{2}}$\\
$0$ & $-\frac{1}{\sqrt{2}}$ & $\alpha$
\end{tabular}\right),
$$
where $\alpha , \beta , \gamma \in \mathbb{R}$, $\beta \neq 0$ in  the form $L_{2} (\alpha , \beta , \gamma )$
and $\eta = \pm 1$.
Moreover,  $\epsilon_{1} = \epsilon_{2}$  and, when $L$ is non-diagonal, 
 also  $\epsilon_{3} = -\epsilon_{2}$
(where $\epsilon_{a}:= g (v_{a}, v_{a}) \in \{ \pm 1\}$).   When $g$ is  definite,  $L$ is always diagonal.

\begin{rem}{\rm The canonical forms  $L_{3, \eta } $ 
were denoted in \cite{c-k-paper} by  $L_{3}$ (when $\eta =1$) and by $L_{4}$ (when $\eta =-1$). 
}
\end{rem}

As remarked in the proof of Proposition 3.2 of  \cite{c-k-paper}, the above  canonical  forms  can be written in a unified way as
\begin{equation}\label{unified}
L = \left( \begin{tabular}{ccc}
$\alpha$  & $\lambda$ & $0$\\
$\lambda$ & $\beta$ & $\mu$\\
$0$ & $\epsilon_{2}\epsilon_{3} \mu $ & $ \gamma$
\end{tabular}\right) ,
\end{equation}
where $\alpha , \beta , \lambda , \mu \in \mathbb{R}$. 
The Lie bracket of $\mathfrak{g}$  is given by 
\begin{align}
\nonumber& \mathcal L_{ v_{1}} v_{2} = \epsilon_{3} L(v_{3}) = \epsilon_{3} ( \mu v_{2} + \gamma v_{3})\\
\nonumber& \mathcal L_{v_{2} } v_{3} = \epsilon_{1} L(v_{1}) = \epsilon_{1} ( \alpha v_{1} + \lambda v_{2})\\
\label{brackets}& \mathcal L_{v_{3}} v_{1}  = \epsilon_{2} L(v_{2}) = \epsilon_{2} ( \lambda v_{1} +\beta v_{2} +\epsilon_{2}\epsilon_3\mu v_{3}),
\end{align}
where we use the convention $L(v_{a}) = \sum_{b=1}^{3} L_{ba} v_{b}$ for any $1\leq a \leq 3.$\

In the next  lemma we  relate the classification 
of $3$-dimensional  unimodular Lie algebras 
mentioned in the introduction  with the classification (in terms of the canonical forms for the operator $L$)
of such Lie algebras endowed  with  a non-degenerate scalar product. The case when $L$ is diagonal can be deduced from \cite[Table on page 307]{milnor} by replacing $(\lambda_1,\lambda_2,\lambda_3)$ with $(\epsilon_1\alpha,\epsilon_2\beta,\epsilon_3\gamma)$. For this reason we assume that $L$ belongs to one of the classes
$L_{2}$, $L_{3, \eta}$  or $L_{5}$ and, in particular, that $\epsilon_{2}\epsilon_3=-1$.

\begin{lem} \label{identif-lie-alg} Let $(\mathfrak{g}, g)$ be a  $3$-dimensional non-abelian unimodular Lie algebra 
with a non-degenerate scalar product  $g$  and  associated operator $L$.\

i) Assume that $L = L_{2} (\alpha , \beta , \gamma )$ (with $\beta\neq 0$). If $\gamma \neq 0$ then $\mathfrak{g}$ is isomorphic to
$\mathfrak{so}(2, 1).$ If $\gamma =0$ then $\mathfrak{g}$ is isomorphic to $\mathfrak{e}(1,1)$.\

ii) Assume that $L = L_{3, \eta} (\alpha , \beta )$. If $\alpha =\beta =0$ then $\mathfrak{g}$ is isomorphic to 
$\mathfrak{heis}$. If $\alpha \neq 0$ and $\beta =0$, then $\mathfrak{g}$ is isomorphic to $\mathfrak{e}(1,1)$.  If $\alpha =0$ and
$\beta \neq 0$, then $\mathfrak{g}$ is isomorphic to $\mathfrak{e} (2)$ when $\beta \eta >0$ and to
$\mathfrak{e}(1,1)$ when $\beta \eta < 0$.  If $\alpha \beta \neq 0$,  then $\mathfrak{g}$ is isomorphic to $\mathfrak{so}(2,1).$\

iii) Assume that $L = L_{5} (\alpha ).$ If $ \alpha =0$, then $\mathfrak{g}$ is isomorphic to $\mathfrak{e}(1,1).$ If $\alpha \neq 0$, then
$\mathfrak{g}$ is isomorphic to $\mathfrak{so}(2,1).$
\end{lem}

\begin{proof}  Note  that  for any $3$-dimensional non-abelian unimodular Lie algebra
$\mathfrak{g}$ with Killing form $K$ one of the following situations hold:

a) $[\mathfrak{g}, \mathfrak{g} ] = \mathfrak{g}$. Then either $K$ is negative definite 
or is of signature $(2,1).$ In the first case,  $\mathfrak{g}$ is isomorphic to $\mathfrak{so}(3)$. In the second case, 
$\mathfrak{g}$ is isomorphic to $\mathfrak{so}(2,1)$;\

b) $\mathrm{dim}\, [ \mathfrak{g}, \mathfrak{g} ] =2.$ Then $[\mathfrak{g}, \mathfrak{g} ] =\mathrm{Ker}\, K$ and 
the induced form $\widehat{K}$ on  $\mathfrak{g}/ [ \mathfrak{g}, \mathfrak{g} ]$ 
is non-trivial. When $\widehat{K}$  is negative definite,  $\mathfrak{g}$ is isomorphic to
$\mathfrak{e}(2)$. When $\widehat{K}$ is  positive definite, $\mathfrak{g}$ is isomorphic to
$\mathfrak{e}(1,1)$;\

c) $\mathrm{dim}\, [ \mathfrak{g}, \mathfrak{g} ] =1.$ Then $\mathfrak{g}$ is isomorphic to $\mathfrak{heis}$.

To prove the above  statements it is sufficient to consider   a basis $(w_{a})$ in which the Lie bracket  of $\mathfrak{g}$ takes the form
$[w_{a}, w_{b}]  = \beta_{c} w_{c}$ for any cyclic permutation $(a, b, c)$, and to use the identification of the Lie algebra $\mathfrak{g}$ in terms of the signs of $\beta_{c}$, see  \cite{milnor} page 307.  (Such a basis can be obtained by choosing a  positive  definite scalar product on $\mathfrak{g}$ and using  that the associated operator $L$ is diagonalizable). \

The claims of the lemma are now  obtained by writing the Lie brackets 
(\ref{brackets}) for each normal form $L = L_{2}$, $L_{3, \eta}$, $L_{5}$,  computing  the commutator $[\mathfrak{g}, \mathfrak{g} ]$ 
and the Killing form 
for  each of the corresponding 
Lie algebras $\mathfrak{g}$, and applying the above statements.
\end{proof}

The following remark will be useful in the  computations from the next sections.

\begin{rem}\label{rem-dorfman}{\rm i) 
One can show  (e.g.\ by using the basis $(w_{a})$ from the previous proof), 
that any left-invariant $2$-form on a  $3$-dimensional unimodular Lie group $G$ is closed.
Therefore, any pair $(H, F)$, where $H\in \Omega^{3}(G)$ and $F\in \Omega^{2}(G)$ are left-invariant, defines a  Courant algebroid $E_{H, F}$ of type $B_{3}$ over $G$.\

ii) Let $G$ be a  $3$-dimensional  unimodular Lie group 
with  Lie algebra $\mathfrak{g}$,  $g$ a  left-invariant metric on $G$
and $(v_{a})$ a $g$-orthonormal basis of $\mathfrak{g}.$ 
Let  
$ H\in \Omega^{3} (G)$ be a left-invariant $3$-form.
The 
Dorfman coefficients of  the generalized tangent bundle $E_{-H}$ 
with respect to the induced basis $ (e_{A}):= (e_{a} := v_{a} + g(v_{a}),\ e_{i} := v_{i-3} - g(v_{i-3}),\
1\leq a\leq 3,\ 4\leq i\leq 6)$ of
$\mathfrak{g}\oplus \mathfrak{g}^{*}$ were computed in the proof of Proposition 3.3 of \cite{c-k-paper}  (when $L$ is  of the general form 
(\ref{unified}) and 
non-diagonal) and  in the proof of Theorem 3.4 
of \cite{c-k-paper} 
(when $L$ is diagonal).
Replacing in  those expressions $H$ with $ - H$  we obtain 
the Dorfman coefficients 
$B_{ABC} = {\mathcal B}_{ABC}$  (where $1\le A, B, C\leq 2n$) 
of $E_{H, F}$ in the basis
$( e_{A})_{0\leq A\leq 6}$ of $\mathfrak{g}\oplus \mathfrak{g}^{*}\oplus \mathbb{R}$
(see the proof of Lemma \ref{ricci-diferit} and Remark  \ref{compare}).\

iii)  In the above setting, let 
$E_{-} = \{ X - g(X)\mid X\in TG\}$, $\delta \in (\mathfrak{g}\oplus \mathfrak{g}^{*} \oplus \mathbb{R})^{*}$ and 
$\tilde{\delta}:= \delta\vert_{\mathfrak{g}\oplus \mathfrak{g}^{*}} $.  The components of the Ricci tensor of 
the pair $(E_{-}, \tilde{\delta })$ defined on $E_{-H}$ 
in the basis $(e_{A})_{1\leq A\leq 6}$ of $\mathfrak{g}\oplus \mathfrak{g}^{*}$ 
were  computed in the proof of Proposition 3.11 of \cite{c-k-paper}
(when $L$ is of the general form (\ref{unified}) and 
non-diagonal) and in the proof of Theorem 3.12  of \cite{c-k-paper} (when $L$ is diagonal).
Replacing in those  expressions $H$ with $ - H$ 
and using Remark \ref{compare} again 
we obtain the components $R^{\tilde{\delta}, +}_{ia}$  ($4\leq i\leq 6$, $1\leq a\leq 3$) of the Ricci tensor
$\mathrm{Ric}^{\tilde{\delta}  , + }$  
from  Corollary \ref{forapplic}.}
\end{rem}

\subsection{The case $L$ diagonal}\label{diag-section}

Let $\mathfrak{g}$ be a $3$-dimensional  Lie algebra with a non-degenerate  scalar product  $g$ and 
diagonal operator $L$ in a $g$-orthonormal basis $(v_{a})$ of $\mathfrak{g}.$ 
Thus 
$L = \mathrm{diag} (\alpha_{1}, \alpha_{2}, \alpha_{3})$, where $\alpha_{i} \in \mathbb{R}$, 
and 
the Lie bracket  of $\mathfrak{g}$ is given by
\begin{equation}\label{def-bracket-diag}
\mathcal L_{v_{1}} v_{2} = \epsilon_{3} \alpha_{3} v_{3},\ \mathcal L_{ v_{3}} v_{1}  = \epsilon_{2} \alpha_{2} v_{2},\  \mathcal L_{v_{2}} v_{3}  = \epsilon_{1} \alpha_{1} v_{1},
\end{equation}
where $\epsilon_{a} = g( v_{a}, v_{a}) \in \{ \pm 1\}$.    
Let $G$ be a Lie group with Lie algebra $\mathfrak{g}$ and 
$E_{H, F}$ (with $F\neq 0$)  a Courant algebroid of type $B_{3}$ over $G$, with $H\in \Omega^{3}(G)$ and $F\in \Omega^{2}(G)$ left-invariant.
We write 
\begin{equation}\label{H-F-coord} 
H= h v_{1}^{*}\wedge v_{2}^{*}\wedge v_{3}^{*},\ F= \frac{1}{2}F_{ab} v_{a}^{*}\wedge v_{b}^{*},
\end{equation} 
where
$( v_{a}^{*})$ is the dual basis of $(v_{a})$ and $h, F_{ab}\in \mathbb{R}.$  
Let $\delta\in
(\mathfrak{g}\oplus\mathfrak{g}^{*}\oplus \mathbb{R})^{*}$ be a left-invariant divergence operator and $\delta_{A}:= \delta (e_{A})$, 
where $(e_{A})_{0\leq A\leq 6}$ is the basis of $\mathfrak{g}\oplus \mathfrak{g}^{*}\oplus \mathbb{R}$ defined by $(v_{a}).$

\begin{thm}\label{classif-diag}  i) The metric  $E_{-} =\{ X - g(X)\mid  X\in TG \}$
defined on $E_{H, F}$ 
is generalized Einstein with divergence $\delta$ if and only if (up to a permutation of  the basis $(v_{a}$)), 
\begin{align}
\nonumber&  \alpha_{1} =\alpha_{3},\ \epsilon_{2} \alpha_{2} (\alpha_{1} -\alpha_{2} ) >0,\  h = - \alpha_{2},\ F_{12} = F_{23} =0,\\
\nonumber& \delta_{A} =0,\ \forall A\in \{ 1, 3, 4, 6\},\ \delta_{5}\in \mathbb{R}
\end{align}
and either $\delta_{0}\delta_{2}\neq  0$ in which case 
$$
(\frac{\delta_{2}}{\delta_{0}})^{2} =\frac{ \epsilon_{2} (\alpha_{1} -\alpha_{2})}{\alpha_{2}},\ 
F_{13} = \epsilon_{2} \alpha_{2} \frac{\delta_{2}}{\delta_{0}},\ 
$$
or  $ \delta_{0}=\delta_{2} =0$ in which case  $(F_{13})^{2} =\epsilon_{2} \alpha_{2} (\alpha_{1} -\alpha_{2}).$\

 ii) If  $\alpha_{1} =\alpha_{3} =0$ then $\mathfrak{g}$ is isomorphic to $\mathfrak{heis}.$ If 
$\alpha_{1} = \alpha_{3} \neq 0$ then $\mathfrak{g}$ is isomorphic to $\mathfrak{so}(3)$
when $( \epsilon_{1} =\epsilon_{3}, \mathrm{sign} (\epsilon_{2} \alpha_{2} ) =  \mathrm{sign} (\epsilon_{1} \alpha_{1} ) )$
and to $\mathfrak{so}(2,1)$ when  $( \epsilon_{1} =\epsilon_{3}, \mathrm{sign} (\epsilon_{2} \alpha_{2} ) =  
-\mathrm{sign} (\epsilon_{1} \alpha_{1} ) )$ or when $\epsilon_{3} =-\epsilon_{1}.$ 
\end{thm}

\begin{proof}
i) From   Remark \ref{rem-dorfman} ii), we obtain  the Dorfman coefficients
$B_{ABC}$  ($1\leq A, B, C\leq 6$) 
 in the basis $(e_{A})_{0\leq A\leq 6}$ 
of $\tilde{E}= \mathfrak{g}\oplus \mathfrak{g}^{*}\oplus \mathbb{R}$ by replacing $h\mapsto - h$  in the Dorfman coefficients 
computed in the proof of Theorem 3.4 of \cite{c-k-paper}: 
\begin{align}
\nonumber& B_{156} =\frac{1}{2} ( - h +\alpha_{1} - \alpha_{2} - \alpha_{3} ) :=\frac{1}{2} X_{1}\\
\nonumber& B_{264} = \frac{1}{2} ( - h -\alpha_{1} +\alpha_{2} -\alpha_{3} ) :=\frac{1}{2} X_{2}\\
\nonumber& B_{345} =\frac{1}{2} ( - h -\alpha_{1} -\alpha_{2} +\alpha_{3}) :=\frac{1}{2} X_{3}\\
\nonumber& B_{423} =\frac{1}{2} ( - h -\alpha_{1} +\alpha_{2} +\alpha_{3}) :=\frac{1}{2} Y_{1}\\
\nonumber& B_{531} =\frac{1}{2} ( - h +\alpha_{1} -\alpha_{2} +\alpha_{3} ) :=\frac{1}{2} Y_{2}\\
\label{dorf-diag} & B_{612}=\frac{1}{2} ( - h +\alpha_{1} +\alpha_{2} -\alpha_{3} ) :=\frac{1}{2} Y_{3},
\end{align}
while the remaining Dorfman coefficients  of the type $B_{aij} $ and $B_{ibc} $ vanish, 
for any other  $1\leq a, b\leq 3$ and $4\leq i, j\leq 6.$  
From   Remark \ref{rem-dorfman} iii)   and the proof of  Theorem 3.12 of \cite{c-k-paper}, 
we obtain  the  coefficients 
$R^{\tilde{\delta}, +}_{ia}$ 
 of the Ricci tensor
$\mathrm{Ric}^{\tilde{\delta} , +}$ 
from Corollary \ref{forapplic} 
in the basis $( e_{A})_{1\leq A\leq 6}$ of $\mathfrak{g}\oplus \mathfrak{g}^{*}$:
\begin{align}
\nonumber& R^{\tilde{\delta}, +}_{41} =-\frac{\epsilon_{2} \epsilon_{3} }{4} (X_{2} Y_{3} + X_{3} Y_{2} ),\ R^{\tilde{\delta}, +}_{51} =-\frac{1}{2} \epsilon_{3} \delta_{3} Y_{2},\ 
R^{\tilde{\delta}, +}_{61} = \frac{1}{2} \epsilon_{2} \delta_{2} Y_{3}\\
\nonumber& R^{\tilde{\delta}, +}_{42}=\frac{1}{2} \epsilon_{3} \delta_{3} Y_{1}, R^{\tilde{\delta}, +}_{52}= -\frac{\epsilon_{1} \epsilon_{3}}{4} (X_{1} Y_{3} + X_{3} Y_{1}),\
R^{\tilde{\delta}, +}_{62} = -\frac{1}{2} \epsilon_{1} \delta_{1} Y_{3}\\
\label{ricci-diag} & R^{\tilde{\delta}, +}_{43}  = -\frac{1}{2} \epsilon_{2} \delta_{2} Y_{1},\ 
R^{\tilde{\delta}, +}_{53} = \frac{1}{2} \epsilon_{1} \delta_{1} Y_{2},\ 
R^{\tilde{\delta}, +}_{63}= -\frac{\epsilon_{1} \epsilon_{2}}{4} ( X_{1} Y_{2} + X_{2} Y_{1}).
\end{align}
Using  relations (\ref{dorf-diag}) and (\ref{ricci-diag}), we obtain from Corollary \ref{forapplic} that  the generalized metric $E_{-}$ 
defined on $E_{H, F}$ 
is generalized Einstein with divergence $\delta$  if and only if 
relations
\begin{align}
\nonumber& -\frac{\epsilon_{2} \epsilon_{3}}{4} (X_{2} Y_{3} + X_{3} Y_{2}) =\epsilon_{2} (F_{12})^{2} +\epsilon_{3} (F_{13})^{2}\\
\nonumber& -\frac{1}{2} \epsilon_{3} \delta_{3} Y_{2} = \delta_{0} F_{12} +\epsilon_{3} F_{23} F_{13}\\
\nonumber& \frac{1}{2} \epsilon_{2} \delta_{2} Y_{3} = \delta_{0} F_{13} -\epsilon_{2} F_{23} F_{12}\\
\nonumber& \frac{1}{2} \epsilon_{3} \delta_{3} Y_{1} = -\delta_{0} F_{12} +\epsilon_{3} F_{23} F_{13}\\
\nonumber& -\frac{\epsilon_{1} \epsilon_{3}}{4} ( X_{1} Y_{3} + X_{3} Y_{1}) = \epsilon_{1} (F_{12})^{2} + \epsilon_{3} (F_{23})^{2}\\
\nonumber& -\frac{1}{2} \epsilon_{1} \delta_{1} Y_{3} = \delta_{0} F_{23} +\epsilon_{1} F_{12} F_{13}\\
\nonumber& -\frac{1}{2} \epsilon_{2} \delta_{2} Y_{1} = -\delta_{0} F_{13} -\epsilon_{2} F_{12} F_{23}\\
\nonumber& \frac{1}{2}  \epsilon_{1}\delta_{1} Y_{2} = -\delta_{0} F_{23} +\epsilon_{1} F_{12} F_{13}\\
\label{set-1-diag} & -\frac{\epsilon_{1} \epsilon_{2}}{4} (X_{1} Y_{2} + X_{2} Y_{1}) = \epsilon_{1} (F_{13})^{2} +\epsilon_{2} (F_{23})^{2}
\end{align}
(which are equivalent to  relations (\ref{rel-1})), 
together with 
\begin{align}
\nonumber& F_{12} \epsilon_{2} \delta_{2} + F_{13} \epsilon_{3} \delta_{3} +F_{23} \frac{\epsilon_{2} \epsilon_{3}}{2} (X_{2}+ X_{3}) =0\\
\nonumber& F_{12} \epsilon_{1} \delta_{1} +\frac{\epsilon_{1}\epsilon_{3}}{2}  F_{13} (X_{1}+ X_{3}) - \epsilon_{3} F_{23} \delta_{3} =0\\
\nonumber& F_{12} \frac{\epsilon_{1}\epsilon_{2}}{2} (X_{1} + X_{2}) - F_{13}\epsilon_{1} \delta_{1} - F_{23} \epsilon_{2}\delta_{2} =0\\
\nonumber& F_{12} \epsilon_{2} (\delta_{2} - \delta_{5}) +F_{13} \epsilon_{3} (\delta_{3} -\delta_{6}) =0\\
\nonumber& F_{12} \epsilon_{1} (\delta_{1} -\delta_{4}) - F_{23}\epsilon_{3} (\delta_{3} - \delta_{6}) =0\\
\label{set-2-diag}& F_{13} \epsilon_{1} (\delta_{1} - \delta_{4}) + F_{23} \epsilon_{2} (\delta_{2} -\delta_{5}) =0
\end{align}
(which are equivalent to  relations  (\ref{rel-2})), 
together with
\begin{align}
\nonumber& \delta_{0} F_{12} = \frac{ \epsilon_{3}}{2}  ( X_{1} \delta_{6} - Y_{2} \delta_{3})\\
\nonumber& \delta_{0} F_{13} = \frac{\epsilon_{2}}{2} ( Y_{3} \delta_{2} - X_{1} \delta_{5} )\\
\nonumber& \delta_{0} F_{23} = \frac{\epsilon_{1}}{2}    (X_{2} \delta_{4} - Y_{3} \delta_{1} )\\
\nonumber& (Y_{1} - Y_{2}) \delta_{3} + (X_{1} - X_{2}) \delta_{6} =0\\
\nonumber& (Y_{3} - Y_{1}) \delta_{2} + (X_{3} - X_{1}) \delta_{5} =0\\
\label{set-3-diag}& (Y_{2} - Y_{3}) \delta_{1} + (X_{2} - X_{3}) \delta_{4}=0
\end{align}
are satisfied. (Relations (\ref{set-3-diag}) are equivalent to relations (\ref{rel-3}) and (\ref{rel-4}). Here we used that   
the first three relations  (\ref{rel-4}) are satisfied, as $B_{41A} = B_{52A} = B_{63A} =0$, which follow from 
$[ v_{a} - g(v_{a}), v_{a} + g(v_{a}) ]_{H, F} =0$).  
Combining  the 2nd relation (\ref{set-1-diag}) with the 4th relation (\ref{set-1-diag}), the 3rd relation (\ref{set-1-diag}) with the 
7th relation
(\ref{set-1-diag}), and the 6th relation (\ref{set-1-diag}) with the 8th relation (\ref{set-1-diag}),  we obtain
$$
\delta_{0} F_{12} = -\frac{1}{4} \epsilon_{3}\delta_{3} (Y_{1}+ Y_{2}),\ \delta_{0} F_{13} = \frac{1}{4} \epsilon_{2}\delta_{2} (Y_{1}+ Y_{3}),\
\delta_{0} F_{23} = - \frac{1}{4} \epsilon_{1}\delta_{1} (Y_{2}+ Y_{3}).
$$
Using the above relations, the first three relations (\ref{set-3-diag}) become
\begin{equation}\label{3-diag}
(Y_{1} - Y_{2}) \delta_{3} + 2 X_{1} \delta_{6}=0,\ (Y_{3} - Y_{1}) \delta_{2} - 2 X_{1} \delta_{5} =0,\  (Y_{2} - Y_{3}) \delta_{1} + 2 X_{2} \delta_{4} =0.
\end{equation}
From (\ref{3-diag}), we deduce that the last three relations (\ref{set-3-diag}) take the form
\begin{equation}\label{4-diag}
(X_{1}+ X_{2}) \delta_{6} = (X_{3}+ X_{1}) \delta_{5} = (X_{2}+ X_{3}) \delta_{4} =0. 
\end{equation}
To summarize our argument so far:  the generalized metric  $E_{-}$  is generalized Einstein with divergence
$\delta$ if and only if 
relations (\ref{set-1-diag}),   (\ref{set-2-diag}),  (\ref{3-diag}) and  (\ref{4-diag}) hold. 
In view of relations (\ref{4-diag}) we distinguish several cases:\

$\bullet$ case I):  none of the vectors $X_{a}+ X_{b}$ (with $a\neq b$) vanishes. 
Then $\delta_{4} = \delta_{5} = \delta_{6} =0$ and relations
(\ref{3-diag}) become
$$
( Y_{1} - Y_{2}) \delta_{3} = (Y_{3} - Y_{1}) \delta_{2} = (Y_{2} - Y_{3}) \delta_{1} =0.
$$
Distinguishing various subcases (according to how many differences $Y_{a} - Y_{b}$  with $a\neq b$ vanish)  and using that a least one from $F_{12}$, $F_{23}$, 
$F_{13}$ is non-zero (because $F\neq 0$), we conclude that case I) does not lead to any solution.\

$\bullet$ case II):  exactly one of the vectors  $X_{a}+ X_{b}$  (where $a\neq b$) vanishes. Without loss of generality, we assume that 
$$
X_{1}+ X_{3} =0,\  (X_{1}+ X_{2}) ( X_{2}+ X_{3}) \neq 0.
$$  
From (\ref{4-diag}) we deduce that $\delta_{4} = \delta_{6} =0$ and relations (\ref{3-diag}) imply 
\begin{equation}\label{sub-2-diag}
(Y_{1} - Y_{2}) \delta_{3} = (Y_{2} - Y_{3} ) \delta_{1} =0.
\end{equation}
In view of (\ref{sub-2-diag}), we consider  various subcases: $(Y_{1} - Y_{2}) (Y_{2} - Y_{3}) \neq 0$;  or $Y_{1} = Y_{2}$ and $Y_{2}\neq Y_{3}$;  or
$Y_{2} = Y_{3}$ and $Y_{1} \neq Y_{2}$;  or $Y_{1} = Y_{2} = Y_{3}.$ It turns out that the first subcase leads to the solutions from the statement of the theorem, while the other subcases do not lead to solutions. Let us give more details on the first subcase. Since $(Y_{1} - Y_{2}) (Y_{2} - Y_{3}) \neq 0$ we obtain from  (\ref{sub-2-diag})
that $\delta_{3} = \delta_{1} =0$. Hence all $\delta_{i} =0$ except, perhaps, $\delta_{0}$, $\delta_{2}$ and $\delta_{5}.$ 
From  the 2nd  and 4th relation (\ref{set-1-diag}), and the 6th and 8th relation (\ref{set-1-diag}), we obtain
\begin{equation}
\delta_{0} F_{12} = \delta_{0} F_{23} = F_{23} F_{13} = F_{12} F_{13} =0.
\end{equation}
If $\delta_{0} \neq 0$ then $F_{12} = F_{23} =0$ (and $F_{13}\neq 0$ because $F\neq 0$). The generalized Einstein condition on $E_{-}$  reduces  to
\begin{align}
\nonumber& \delta_{A} =0,\ \forall A\in \{ 1, 3, 4, 6\},\  \delta_{0} \neq 0,\  F_{12} = F_{23} =0, F_{13}\neq 0,\  X_{3} = - X_{1},\\
\label{details-diag-0} &  (X_{1}+ X_{2}) (X_{2} + X_{3})\neq 0,\ (Y_{1} - Y_{2})( Y_{2} - Y_{3})\neq 0
\end{align}
together with
\begin{align}
\nonumber& (F_{13})^{2}= \frac{\epsilon_{2}}{4} ( X_{1} Y_{2} - X_{2} Y_{3}  ) =-\frac{\epsilon_{2}}{4} ( X_{1} Y_{2} + X_{2} Y_{1}),\\
\nonumber&  \delta_{0} F_{13}=  \frac{1}{2} \epsilon_{2} \delta_{2} Y_{3} = \frac{1}{2} \epsilon_{2} \delta_{2} Y_{1}\\
\label{details-diag}& X_{1} ( Y_{3} - Y_{1}) =0,\  (Y_{3} - Y_{1}) \delta_{2} = 2 X_{1} \delta_{5}.
\end{align}
From the second line of  (\ref{details-diag})  together with $\delta_{0} F_{13} \neq 0$ 
we obtain $Y_{1} = Y_{3}$, which,  together with $X_{3} = -X_{1}$, imply  $\alpha_{1} = \alpha_{3}$ and $h = -\alpha_{2}.$  
The last relation  (\ref{details-diag-0}) is equivalent to $\alpha_{1} \neq \alpha_{2}$. The first and second line of  (\ref{details-diag}) reduce to 
\begin{equation}\label{details-diag-1}
(F_{13})^{2} = -\frac{\epsilon_{2} X_{2} Y_{3}}{ 4},\ F_{13} = \frac{ \epsilon_{2} \delta_{2} Y_{3}}{ 2\delta_{0}}.
\end{equation}
Since $X_{2} = 2 (\alpha_{2} -\alpha_{1})$ and $Y_{3} = 2\alpha_{2}$, we deduce from (\ref{details-diag-1}) that  
$F_{13} = \epsilon_{2}\alpha_{2}\frac{\delta_{2}}{\delta_{0}}$. 
Since $F_{13}\neq 0$,  we also obtain that
$\alpha_{2}\delta_{2} \neq 0$ and then
$(\frac{\delta_{2}}{\delta_{0}})^{2} = \frac{\epsilon_{2} (\alpha_{1} -\alpha_{2})}{\alpha_{2}}$.
We arrive at the class of solutions from Theorem \ref{classif-diag} with $\delta_{0}\delta_{2}\neq 0.$
A similar argument shows that the assumption $\delta_{0}=0$ leads to the other class  of solutions from Theorem \ref{classif-diag}.\

$\bullet$ case III) and IV):  exactly two or all  $(X_{a} + X_{b})$ (with $a\neq b$) vanish.
These cases do not lead to any solution.\

 ii)  Claim ii) follows from \cite{milnor}, p.\ 307.
\end{proof}

\subsection{The case $L=L_{2}$}

Let $G$ be a $3$-dimensional unimodular Lie group with Lie algebra $\mathfrak{g}.$

\begin{thm}\label{classif-l2}  
There are no odd generalized Einstein metrics  $E_{-}$  on 
$G$, with the property that the operator $L$ associated to  $(\mathfrak{g}, g)$ belongs to the second class of canonical forms.
Here $g$ is the (left-invariant) pseudo-Riemannian metric on $G$ induced by $E_{-}.$
\end{thm}

\begin{proof}   Assume, by contradiction, that
$E_{-} = \{ X - g(X)\mid  X\in TG \}$ is   a left-invariant generalized Einstein metric 
on a Courant algebroid 
$E_{H, F}$  of type $B_{3}$ 
over $G$, 
with  left-invariant divergence $\delta$ and with the property from the statement of the theorem. 
Let  $(v_{a})_{1\leq a\leq 3}$ be a $g$-orthonormal basis of $\mathfrak{g}$, in which the operator $L$ takes the form 
$L = L_{2} (\alpha , \beta , \gamma )$ where $\alpha , \beta  ,\gamma \in \mathbb{R}$ and $\beta \neq 0 .$ 
Since $L$ is not diagonal, $\epsilon_{1} = \epsilon_{2} = -\epsilon_{3}.$ 
We will arrive at a contradiction. We write $H$ and $F$ 
as in (\ref{H-F-coord}) and, as before, $\delta_{A}:= \delta (e_{A})$ where  
$(e_{A})_{0\leq A\leq 6}$ is the basis of $\mathfrak{g}\oplus \mathfrak{g}^{*}\oplus \mathbb{R}$ determined by $(v_{a}) .$

From Remark \ref{rem-dorfman}  ii) we obtain the Dorfman coefficients  $B_{ABC}$ ($1\leq A, B, C\leq 6$)  in the basis $(e_{A})_{0\leq A\leq 6}$,  by replacing in the Dorfman coefficients
computed in  the  proof of Proposition 
3.3 of \cite{c-k-paper}  $h\mapsto - h$,   $\alpha \mapsto\gamma$, $\lambda \mapsto 0$, $\beta \mapsto\alpha$, 
$\gamma\mapsto \alpha$ and $\mu\mapsto - \beta$. We obtain
\begin{align}
\nonumber& B_{145} = B_{146} = B_{256} = B_{356} = B_{412} = B_{413}=
B_{523} = B_{623} =0\\
\nonumber& B_{156} = \frac{1}{2} ( - h -2\alpha +\gamma ),\  B_{245} =\beta ,\ B_{246} =\frac{1}{2} (h+\gamma )\\
\nonumber& B_{345} = -\frac{1}{2} ( h+\gamma ),\ B_{346}=\beta ,\ B_{423} =\frac{1}{2} ( -h +2\alpha -\gamma )\\
\label{dorf-l2}& B_{512} = -\beta ,\ B_{513} = \frac{1}{2} (h-\gamma ),\ B_{612} = \frac{1}{2} ( \gamma - h),\ B_{613} = -\beta .
\end{align}
From  Remark \ref{rem-dorfman} iii)   we obtain 
 the coefficients 
$R_{ia}^{\tilde{\delta}, +}$ 
of the Ricci tensor $\mathrm{Ric}^{\tilde{\delta}, +}$ from Corollary \ref{forapplic} by replacing 
$h\mapsto - h$ in the Ricci coefficients computed 
in the proof of Proposition 3.11 of \cite{c-k-paper}:
\begin{align}
\nonumber& R^{\tilde{\delta}, +}_{41} = - 2\beta^{2} -\frac{\gamma^{2}}{2} +\frac{h^{2}}{2}\\
\nonumber& R^{\tilde{\delta}, +}_{42} = \frac{\epsilon_{3}}{2} ( - h + 2\alpha -\gamma )\delta_{3}\\
\nonumber& R^{\tilde{\delta}, +}_{43}= \frac{\epsilon_{2}}{2} ( h - 2\alpha + \gamma ) \delta_{2}\\
\nonumber& R^{\tilde{\delta}, +}_{51}= -\epsilon_{2}\beta \delta_{2} +\frac{\epsilon_{3}}{2} (h -\gamma )\delta_{3}\\
\nonumber& R^{\tilde{\delta}, +}_{52}= -\frac{\alpha^{2}}{2} +\frac{h^{2}}{2} +\frac{1}{2} ( \alpha -\gamma )^{2} +\epsilon_{1} \beta\delta_{1}\\
\nonumber& R^{\tilde{\delta}, +}_{53}= \beta (2\alpha -\gamma ) + \frac{\epsilon_{1}}{2} ( \gamma -h )\delta_{1}\\
\nonumber& R^{\tilde{\delta}, +}_{61}=\frac{\epsilon_{2}}{2} ( \gamma - h) \delta_{2} -\epsilon_{3} \beta  \delta_{3}\\
\nonumber& R^{\tilde{\delta}, +}_{62}= \beta (2\alpha - \gamma  ) + \frac{\epsilon_{1}}{2} ( h-\gamma )\delta_{1}\\
\label{ricci-l2}& R^{\tilde{\delta}, +}_{63}= \frac{\alpha^{2}}{2} -\frac{h^{2}}{2} -\frac{1}{2}  ( \alpha - \gamma )^{2} +\epsilon_{1} \beta \delta_{1}.
\end{align}
Using relations (\ref{dorf-l2}) and (\ref{ricci-l2})  we obtain from Corollary \ref{forapplic} that  
\begin{align}
\nonumber& - 2\beta^{2} -\frac{\gamma^{2}}{2} + \frac{h^{2}}{2} = \epsilon_{1}  (( F_{12})^{2} - (F_{13})^{2})\\
\nonumber& - \epsilon_{1}\beta \delta_{2} + \frac{\epsilon_{1}}{2} ( \gamma -h) \delta_{3} = \delta_{0} F_{12} -\epsilon_{1} F_{23} F_{13}\\
\nonumber& \frac{\epsilon_{1}}{2} ( \gamma -h ) \delta_{2} +\epsilon_{1} \beta \delta_{3} =\delta_{0} F_{13} +\epsilon_{1} F_{32} F_{12}\\
\nonumber& \frac{\epsilon_{1}}{2} ( h - 2\alpha +\gamma ) \delta_{3} = -\delta_{0} F_{12} -\epsilon_{1} F_{13} F_{23}\\
\nonumber& -\frac{\alpha^{2}}{2} +\frac{h^{2}}{2} +\frac{ (\alpha -\gamma )^{2}}{2} +\epsilon_{1} \beta \delta_{1} = \epsilon_{1} ( (F_{12})^{2} - (F_{23})^{2})\\
\nonumber& \beta ( 2\alpha -\gamma ) +\frac{\epsilon_{1}}{2} ( h-\gamma )\delta_{1} = \delta_{0} F_{23} +\epsilon_{1} F_{12} F_{13}\\
\nonumber& \frac{\epsilon_{1}}{2} ( h -2\alpha +\gamma ) \delta_{2} = - \delta_{0} F_{13} +\epsilon_{2} F_{12} F_{32}\\
\nonumber& \beta ( 2\alpha -\gamma ) + \frac{\epsilon_{1}}{2} ( \gamma -h ) \delta_{1} = -\delta_{0}  F_{23}  +\epsilon_{1} F_{12} F_{13}\\
\label{set1-l2}& \frac{\alpha^{2}}{2} -\frac{h^{2}}{2} - \frac{ (\alpha - \gamma )^{2}}{2} +\epsilon_{1} \beta \delta_{1} = \epsilon_{1} ( (F_{13})^{2} + (F_{23})^{2})
\end{align}
(which are equivalent to relations (\ref{rel-1})), 
together with 
\begin{align}
\nonumber& F_{12}\delta_{2} - F_{13}\delta_{3} +\epsilon_{1} F_{23} ( h+\gamma ) =0\\
\nonumber& F_{12} ( \beta +\epsilon_{1} \delta_{1}) + F_{13} ( h+\alpha ) + F_{23} \epsilon_{1} \delta_{3} =0\\
\label{set2-l2} & F_{12}(h +\alpha )  + F_{13} ( \epsilon_{1} \delta_{1}- \beta ) +\epsilon_{1} F_{23} \delta_{2} =0
 \end{align}
(which are equivalent to the first three relations (\ref{rel-2})), 
together with
\begin{align}
\nonumber& \epsilon_{1} \delta_{0} F_{12} = -\beta \delta_{2} + \frac{1}{2} (\gamma - h) \delta_{3} +\frac{1}{2} ( h+ 2\alpha -\gamma ) \delta_{6}\\
\nonumber& \epsilon_{1} \delta_{0} F_{13} = \frac{1}{2} ( - h +\gamma ) \delta_{2} +\beta \delta_{3} +\frac{1}{2} ( h -\gamma +2\alpha ) \delta_{5} \\
\label{set3-l2} & \epsilon_{1} \delta_{0} F_{23} = \frac{1}{2} ( h-\gamma ) \delta_{1} -\frac{1}{2} ( h+\gamma ) \delta_{4} 
\end{align}
(which are equivalent to relations (\ref{rel-3})), together with
\begin{align}
\nonumber& \beta ( \delta_{1} -\delta_{4}) =0\\
\nonumber&  \beta  ( \delta_{2} -\delta_{5}) + (\alpha -\gamma ) (\delta_{3} - \delta_{6}) =0\\
\label{set4-l2} &  \beta ( \delta_{3}-\delta_{6}) - (\alpha  -\gamma ) ( \delta_{2} - \delta_{5} ) =0
\end{align}
(which are equivalent to relations  (\ref{rel-4})) hold.  (The last three relations (\ref{rel-2}) will not be used in this proof.) 
Since $\beta \neq 0$, the 1st relation (\ref{set4-l2}) implies that $\delta_{1} =\delta_{4}$, while 
the  2nd and 3rd relation (\ref{set4-l2})  imply that $\delta_{2}= \delta_{5} $ and $\delta_{3} =\delta_{6}$.
Relations  (\ref{set3-l2}) become
\begin{equation}\label{delta0-l2}
\epsilon_{1} \delta_{0} F_{12} = -\beta \delta_{2} +\alpha \delta_{3},\ \epsilon_{1} \delta_{0} F_{13} = \alpha \delta_{2} +\beta \delta_{3},\ \epsilon_{1} \delta_{0} F_{23}=
-\gamma \delta_{1}.
\end{equation}
We claim  that either $F_{12} \neq 0$ or
$F_{13} \neq 0$. Suppose,  by contradiction, that $F_{12} = F_{13} =0$. Then $F_{23} \neq 0$ (because $F\neq 0$). 
The  1st relation (\ref{set2-l2}) implies $h =-\gamma $. But then the 1st relation (\ref{set1-l2}) implies $\beta =0$, which is a contradiction.\

We claim  that $F_{23} =  0.$ Assume, by contradiction, that $F_{23} \neq 0.$ We distinguish two  cases, namely
$\delta_{0}=0$ and $\delta_{0}\neq 0$,  and we show that both lead to a contradiction. If $\delta_{0} =0$ then
the first two relations (\ref{delta0-l2})  and $\beta \neq 0$ 
imply that $\delta_{2} =\delta_{3} =0.$
But the 2nd and 3rd  relation (\ref{set1-l2}) with $\delta_{0} = \delta_{2}=\delta_{3}=0$ and  $F_{23} \neq 0$ imply that $F_{12} = F_{13} =0$ which is impossible by our previous argument.    If $\delta_{0} \neq 0$ then relations (\ref{delta0-l2}) become
\begin{equation}\label{delta0-l2-1}
F_{12} = \frac{\epsilon_{1} }{\delta_{0}} ( -\beta \delta_{2} +\alpha \delta_{3}),\ F_{13} = \frac{\epsilon_{1}}{\delta_{0}} ( \alpha\delta_{2} +\beta \delta_{3}),\
F_{23}  = -\frac{\epsilon_{1}\gamma }{\delta_{0}} \delta_{1}.
\end{equation}
In particular,  $\gamma \delta_{1}\neq 0$ because $F_{23} \neq 0$. Using the 3rd relation (\ref{delta0-l2-1}) we obtain from the 6th and 8th relation
(\ref{set1-l2}) that $(h+\gamma ) \delta_{1} =0$, i.e.\ $ h=-\gamma$  (because $\delta_{1}\neq 0$). Multiplying the 1st 
relation (\ref{set2-l2}) with $\delta_{0}$ and
using (\ref{delta0-l2-1}) together with $ h = -\gamma$ and  $\beta\neq 0$  we obtain
$\delta_{2} =\delta_{3} =0. $
But $\delta_{2} =\delta_{3}  =0$ imply again  
from (\ref{delta0-l2-1}) 
that $F_{12} = F_{13} =0$ which is a contradiction. We proved that $F_{23} =0.$ Moreover, since 
$\delta_{1} =\delta_{4}$ and $F_{23} =0$, the last relation (\ref{set3-l2}) implies that $\gamma \delta_{1} =0.$\

To summarize our argument so far:  we proved that 
$$
\delta_{1} = \delta_{4},\ \delta_{2} =\delta_{5},\ \delta_{3} =\delta_{6},\  F_{23} =\gamma \delta_{1} =0
$$
and 
either $F_{12}$ or  $F_{13}$ is different from zero. We  will distinguish  two cases, namely $\gamma =0$ respectively $\gamma \neq  0$, and we will prove that both lead to a contradiction.\ 

$\bullet$ case I): $\gamma =0.$ The 6th and 8th relations  (\ref{set1-l2}) imply that  $h\delta_{1} =0.$ 
Using (\ref{delta0-l2}), the 4th and 7th relation (\ref{set1-l2})  become
$ \frac{h}{2} \delta_{3} =\beta \delta_{2}$, $ \frac{h}{2} \delta_{2} =- \beta \delta_{3}$, which imply that $\delta_{2} = \delta_{3} =0.$ Using that either $F_{12}$ of
$F_{13}$ is non-zero, the first two relations (\ref{delta0-l2}) together with $\delta_{2} =\delta_{3} =0$ imply that $\delta_{0}=0.$ 
So:  
\begin{equation}
h\delta_{1} =0,\   \delta_{0} = \delta_{2} = \delta_{3} = \delta_{5} =\delta_{6} =0,\   \delta_{1} =\delta_{4},\ F_{23} =0,\ \gamma =0.
\end{equation}
If $h=0$, the 5th and 9th  relation (\ref{set1-l2}) imply that 
$(F_{12})^{2} = (F_{13})^{2}$  and then, from the 1st relation (\ref{set1-l2}), $\beta =0$, which is  contradiction.
If $h\neq 0$, $\delta_{1} =0$ and the 2nd  and 3rd relation (\ref{set2-l2}) give
\begin{equation}\label{repeta}
F_{12} \beta + F_{13} ( h+\alpha ) =0,\  F_{12} ( h+\alpha )-F_{13} \beta =0,
\end{equation}
which imply $F_{12} = F_{13} =0$.  We arrived again at a contradiction.\ 

$\bullet$ case II): $\gamma \neq 0.$ From $\gamma \delta_{1} =0$ we obtain $\delta_{1} =0.$  
The 2nd and 3rd relation (\ref{set2-l2}) with $\delta_{1} = F_{23}=0$ 
reduce  again to (\ref{repeta}) (and to a contradiction).  
\end{proof}

\subsection{The case $L = L_{3, \eta}$}

We consider  the setting from  the beginning of Section \ref{diag-section}, but  we assume that 
$L = L_{3, \eta } (\alpha , \beta  )$ in the $g$-orthonormal basis $(v_{a})$,  
where $\alpha , \beta \in \mathbb{R}$ and $\eta \in \{ \pm 1\} .$ 
Since $L$ is non-diagonal, $\epsilon_{1} =\epsilon_{2} = - \epsilon_{3}.$

\begin{thm}\label{classif-l3}
i) The   metric $E_{-}= \{ X - g(X)\mid  X\in TG\}$  defined on $E_{H, F}$ is generalized Einstein with divergence
$\delta$  if and only if one of the following situations holds:\

1) $\beta = 2\alpha\neq 0$,   $\epsilon_{1} 
= \epsilon_{2} = -1$,  $\epsilon_{3} =1$, $ h = -2\alpha$,    $F_{12} = F_{13} =0$, $F_{23}= \epsilon \sqrt{2} \alpha$
(where $\epsilon \in \{ \pm 1\}$), 
  $\delta =0$;\

2)  $\beta \neq 2\alpha $, $\epsilon_{1} \beta (\alpha -\beta )>0$,    $ h = -\beta$,  $F_{12} = F_{13}=0$, $F_{23} =\frac{
\beta (2\alpha -\beta )}{\delta_{0}}$,  $(\delta_{0})^{2} = \frac{\epsilon_{1}\beta ( 2\alpha - \beta )^{2}}{ \alpha -\beta }$, 
$\delta_{1} = \delta_{4} = \epsilon_{1} ( \beta - 2\alpha )$ and  $\delta_{A} =0$
for any $A\notin  \{ 0, 1, 4\}$;\

3)  $\beta=0$,   $\epsilon_{1}\alpha \eta < 0$,  $h=0$, 
$F_{12} = F_{13}$, $(F_{12})^{2} = -\frac{\epsilon_{1} \alpha \eta}{ 2}$, 
$F_{23} =0$,  
$\delta_{A} =0$ for any $A\notin \{ 1, 4\}$ and  $\delta_{1} =\delta_{4} = - \epsilon_{1} \alpha$;\

4) $\beta = \alpha$,  $\epsilon_{1}\alpha \eta <0$,  
$h = -\alpha$,  $F_{12} = F_{13}$,  
$F_{23} =0$, $\delta_{1} =\delta_{4} =0$, $\delta_{2} =\delta_{3}$, $\delta_{5} =\delta_{6}$ and
$$
\delta_{0} F_{12} = \epsilon_{1}\alpha \delta_{2},\ (F_{12})^{2} = -\frac{\epsilon_{1} \alpha \eta}{2}. 
$$

ii)  The Lie algebra  $\mathfrak{g}$ is isomorphic to $\mathfrak{so}(2,1)$, except for the class of solutions 3), when it is isomorphic to
$\mathfrak{e}(1,1)$.
\end{thm}

\begin{proof} 
i) From Remark \ref{rem-dorfman} ii), we 
obtain the Dorfman coefficients  $B_{ABC}$ ($1\leq A, B, C\leq 6$) 
in the basis $(e_{A})_{0\leq A\leq 6}$ by 
replacing  in the Dorfman coefficients computed in the proof of Proposition 
3.3 of \cite{c-k-paper} $h\mapsto - h$,   $\alpha\mapsto \beta$, $\lambda\mapsto 0$, $\beta\mapsto \frac{\eta}{2} +\alpha$, $\mu\mapsto \frac{\eta}{2}$
and $\gamma \mapsto -\frac{\eta}{2} +\alpha$:
\begin{align}
\nonumber& B_{145} =B_{146} = B_{256} = B_{356} = B_{412} = B_{413} = B_{523} = B_{623} =0,\\
\nonumber& B_{156} =\frac{1}{2} ( - h -2\alpha +\beta ),\ B_{245} =-\frac{1}{2}\eta ,\ B_{246} =\frac{1}{2} ( h -\eta +\beta ),\\
\nonumber& B_{345} =-\frac{1}{2} (h+\eta +\beta ),\ B_{346}  = -\frac{1}{2}\eta ,\ B_{423} =\frac{1}{2} ( - h +2\alpha - \beta ),\\
\label{dorfman-l34} & B_{512} =\frac{1}{2}\eta ,\ B_{513} =\frac{1}{2} ( h+\eta -\beta ),\ B_{612} =\frac{1}{2} ( - h +\eta +\beta ),\
B_{613} =\frac{1}{2}\eta .
\end{align} 
From Remark \ref{rem-dorfman} iii)  
we  obtain  the  coefficients  
$R_{ia}^{\tilde{\delta} , +}$ 
of the Ricci tensor $\mathrm{Ricci}^{\tilde{\delta}, + }$ from Corollary \ref{forapplic} 
by making the same replacements as above in the Ricci coefficients computed in
the proof of  Proposition 3.11 from \cite{c-k-paper}.
Using that  $ \epsilon_{1} =\epsilon_{2} = -\epsilon_{3}$ we  obtain
\begin{align}
\nonumber& R_{41}^{\tilde{\delta}, +} =\frac{1}{2} ( h^{2} - \beta^{2}),\ R_{42}^{\tilde{\delta}, +} = \frac{\epsilon_{1}}{2} ( h -2\alpha +\beta )\delta_{3},\
R_{43}^{\tilde{\delta}, +} = \frac{\epsilon_{1}}{2} ( h - 2\alpha +\beta ) \delta_{2},\\
\nonumber& R_{51}^{\tilde{\delta}, +} = \frac{\epsilon_{1} \eta}{2} \delta_{2}-\frac{\epsilon_{1}}{2} ( h +\eta -\beta ) \delta_{3},\ 
R_{52}^{\tilde{\delta}, +}  = \frac{1}{2} (\beta - 2\alpha  ) (\beta   +\eta ) +\frac{h^{2}}{2} -\frac{\epsilon_{1} \eta}{2}  \delta_{1},\\
\nonumber& R_{53}^{\tilde{\delta}, +} =\frac{\eta }{2} (\beta - 2\alpha ) -\frac{\epsilon_{1}}{2} ( h+\eta -\beta )\delta_{1},\ R_{61}^{\tilde{\delta}, +} =\frac{\epsilon_{1}}{2} ( - h+\eta +\beta ) \delta_{2} -\frac{\epsilon_{1}\eta }{2}  \delta_{3},\\
\label{ricci-l34}& R_{62}^{\tilde{\delta}, +} = \frac{\eta }{2} ( \beta - 2\alpha ) +\frac{\epsilon_{1}}{2} ( h -\eta -\beta ) \delta_{1},\
R_{63}^{\tilde{\delta}, +} =\frac{1}{2} ( 2\alpha - \beta )( \beta - \eta ) -\frac{h^{2}}{2} -\frac{\epsilon_{1} \eta }{2} \delta_{1}.
\end{align} 
The metric  $E_{-} $ is generalized Einstein  with divergence
$\delta$ 
if and only if the relations from Corollary \ref{forapplic}, with Dorfman coefficients given by  (\ref{dorfman-l34}), the  components of the Ricci tensor 
$\mathrm{Ric}^{\tilde{\delta}, +}$
given by  (\ref{ricci-l34})
and $\epsilon_{1}= \epsilon_{2} = -\epsilon_{3}$, hold.  
Relations (\ref{rel-4})  from Corollary \ref{forapplic} are equivalent to 
\begin{align}
\nonumber& \delta_{1}  =\delta_{4},\\
\nonumber&  \eta ( \delta_{2} -\delta_{5}) = 2( \beta -\alpha -\frac{\eta}{2}) (\delta_{6} -\delta_{3}),\\
\label{primele} & \eta ( \delta_{3} -\delta_{6}) = 2( \beta -\alpha +\frac{\eta}{2} ) (\delta_{2} -\delta_{5}),
\end{align}
or to
\begin{equation}\label{beta-alpha-delta}
\delta_{1} =\delta_{4},\ (\beta -\alpha ) (\delta_{2} - \delta_{5})=0,\ \delta_{2} - \delta_{5}=\delta_{3} - \delta_{6}.
\end{equation}
Using  $\delta_{1} = \delta_{4}$, the 3rd relation (\ref{rel-3})  becomes 
\begin{equation}\label{primele-1} 
- \epsilon_{1} \beta \delta_{1} = \delta_{0} F_{23}
\end{equation}
and the remaining relations from Corollary  \ref{forapplic}
reduce to 
\begin{align}
\nonumber& \frac{1}{2} ( h^{2} - \beta^{2}) = \epsilon_{1} (  (F_{12})^{2} - (F_{13})^{2}),\\
\nonumber& \frac{\epsilon_{1} \eta}{2} \delta_{2} -\frac{\epsilon_{1}}{2} ( h +\eta -\beta ) \delta_{3} = \delta_{0} F_{12} - \epsilon_{1} F_{23} F_{13},\\
\nonumber& \frac{\epsilon_{1}}{2} ( - h +\eta +\beta ) \delta_{2} -\frac{\epsilon_{1}\eta}{2} \delta_{3} = \delta_{0} F_{13} - \epsilon_{1} F_{23} F_{12},\\
\nonumber& \frac{\epsilon_{1}}{2} ( h -2\alpha +\beta ) \delta_{3} = - \delta_{0} F_{12} - \epsilon_{1} F_{13} F_{23},\\
\nonumber& \frac{1}{2} (\beta - 2\alpha  ) (\beta +\eta ) +\frac{h^{2}}{2} -\frac{\epsilon_{1} \eta}{2}  \delta_{1} = \epsilon_{1} ( (F_{12})^{2} - (F_{23})^{2}),\\
\nonumber& \frac{\eta}{2} ( \beta - 2\alpha ) +\frac{\epsilon_{1}}{2} ( h -\eta -\beta )\delta_{1} = \delta_{0} F_{23} +\epsilon_{1} F_{12} F_{13},\\
 \nonumber& \frac{\epsilon_{1}}{2} ( h -2\alpha +\beta ) \delta_{2} = - \delta_{0} F_{13} +\epsilon_{1} F_{12} F_{32},\\
 \nonumber& \frac{\eta}{2} ( 2\alpha -\beta ) +\frac{\epsilon_{1}}{2} ( h +\eta -\beta ) \delta_{1} = \delta_{0} F_{23} - \epsilon_{1} F_{12} F_{13},\\
 \label{l-45-1}& \frac{1}{2} ( 2\alpha -\beta ) (\beta -\eta ) -\frac{h^{2}}{2} -\frac{ \epsilon_{1} \eta}{2}  \delta_{1}  = \epsilon_{1} 
 ( (F_{13})^{2} + (F_{23})^{2})
 \end{align}
  and
 \begin{align}
 \nonumber& F_{12} \epsilon_{1}\delta_{2} - F_{13} \epsilon_{1} \delta_{3}   + ( h+\beta ) F_{23} =0,\\
 \nonumber& F_{12} ( \frac{\eta}{2} - \epsilon_{1} \delta_{1}) - F_{13} ( h +\alpha +\frac{\eta}{2}) - \epsilon_{1} F_{23} \delta_{3} =0,\\
 \label{l-45-2}& F_{12} (\frac{\eta}{2} - ( h+\alpha )) - F_{13} (\frac{\eta}{2} +\epsilon_{1} \delta_{1}) -  \epsilon_{1} F_{23} \delta_{2} =0 
 \end{align}
  and
\begin{equation}\label{l-45-3}
(F_{12}  - F_{13}) ( \delta_{2} -\delta_{5}) = 0,\  F_{23} ( \delta_{2} - \delta_{5}) =0
\end{equation}
and 
\begin{align}
\nonumber& \frac{\epsilon_{1} \eta}{2}\delta_{2}  -\frac{\epsilon_{1} }{2} ( h +\eta -\beta ) \delta_{3} +\frac{\epsilon_{1} }{2} ( h+ 2\alpha -\beta ) \delta_{6} =   \delta_{0} F_{12},\\
\label{l-45-4}& \frac{\epsilon_{1} }{2} ( - h +\eta +\beta ) \delta_{2} -\frac{\epsilon_{1} \eta}{2} \delta_{3} +\frac{\epsilon_{1} }{2} ( h +2\alpha -\beta ) \delta_{5} 
= \delta_{0}
F_{13}.
\end{align}
(Relations (\ref{l-45-3}) coincide with the last three relations (\ref{rel-2}),  by using  $\delta_{2} - \delta_{5} = \delta_{3} -\delta_{6}$ 
and $\delta_{1} =\delta_{4}$,  see 
relations  (\ref{beta-alpha-delta})). 
In view of 
the second relation   (\ref{l-45-3}), we consider two cases, namely  $F_{23} \neq 0$ or  $F_{23} =0.$\

$\bullet$ case I):   $F_{23} \neq 0.$     
From  the 3rd relation (\ref{beta-alpha-delta}) 
and  the 2nd relation (\ref{l-45-3}),  we get
$\delta_{2}= \delta_{5}$ 
and $\delta_{3} =\delta_{6}$. 
Relations  (\ref{beta-alpha-delta}),  (\ref{primele-1}), 
 (\ref{l-45-3}) and (\ref{l-45-4}) become 
\begin{align}
\nonumber& \delta_{2} = \delta_{5},\ \delta_{3} = \delta_{6},\ \delta_{1} = \delta_{4},\\
 \nonumber& \frac{\eta}{2} \delta_{2} +( \alpha -\frac{\eta}{2}) \delta_{3} = \epsilon_{1} \delta_{0} F_{12},\\
 \nonumber& (\alpha +\frac{\eta}{2}) \delta_{2} -\frac{\eta}{2}\delta_{3} =\epsilon_{1}\delta_{0} F_{13},\\
 \label{sec-l-45} & - \beta \delta_{1} = \epsilon_{1} \delta_{0} F_{23}.
 \end{align}
The generalized metric $E_{-}$ is generalized Einstein with divergence $\delta$  if and only if  relations (\ref{sec-l-45}) together with
(\ref{l-45-1}) 
and (\ref{l-45-2}) hold.\

$\bullet$ case I-1):  $\delta_{0}=0$.  The 2nd, 3rd and 4th lines from   (\ref{sec-l-45}) give
\begin{equation}
\delta_{2} = \delta_{3},\ \alpha \delta_{2} =0,\ \beta \delta_{1}=0
\end{equation}
and the remaining relations (\ref{l-45-1}) and (\ref{l-45-2}) are solved  precisely by
\begin{equation}
\beta =-h =  2\alpha \neq 0,\ F_{12} = F_{13} =0,\ \delta =0,\ \epsilon_{1} =-1,\  (F_{23})^{2} =2\alpha^{2}
\end{equation}
(where we used $F\neq 0$).  We obtain the 1st  class of solutions  listed in  the theorem.

$\bullet$ case I-2):  $\delta_{0}\neq 0$.  From the  2nd, 3rd and 4th  lines from  (\ref{sec-l-45}) we obtain
\begin{align}
\nonumber& F_{12} = \frac{\epsilon_{1}}{\delta_{0}} (\frac{\eta}{2} \delta_{2} + (\alpha -\frac{\eta}{2} )\delta_{3}),\\
\nonumber& F_{13} = \frac{\epsilon_{1}}{\delta_{0}} ( (\alpha + \frac{\eta}{2} ) \delta_{2}  -\frac{\eta}{2} \delta_{3}),\\
\label{f-23-new} & F_{23} = -\frac{\epsilon_{1}\beta \delta_{1}}{ \delta_{0}}.
\end{align}
As $F_{23}\neq 0$, we obtain that $\beta \delta_{1}\neq 0.$ 
Adding  the 6th and 8th  relation (\ref{l-45-1}),  using the expression of $F_{23}$ given by the 3rd
relation (\ref{f-23-new}) and $\delta_{1}\neq 0$  we obtain that 
$$
h +\beta =0.
$$ 
Since $h +\beta =0$ the 1st relation 
(\ref{l-45-2}) becomes $F_{12} \delta_{2} = F_{13}\delta_{3}$. Replacing in this relation   the expressions of $F_{12}$ and $F_{13}$ from 
(\ref{f-23-new}) we deduce that
$\delta_{2} =\delta_{3}.$ 
From   (\ref{sec-l-45}),  
we  obtain that
$$
\delta_{2} = \delta_{3} = \delta_{5}  = \delta_{6}.
$$
Relations (\ref{f-23-new}) become
\begin{equation}\label{expr-F-l-45}
F_{12} = F_{13} = \frac{\epsilon_{1} \alpha \delta_{2}}{\delta_{0}},\  F_{23} = -\frac{\epsilon_{1} \beta \delta_{1}}{ \delta_{0}}.
\end{equation}
Combining the first two relations (\ref{expr-F-l-45})  with the 7th  relation (\ref{l-45-1}) and using that 
$h+\beta =0$ and 
$F_{23} \neq 0$
we obtain that $F_{12} =0$. Hence $F_{12} =  F_{13} =0$  
and from  the 2nd relation (\ref{l-45-2}) 
we deduce that $\delta_{3}=0.$ 
Thus,
$$
\delta_{A} =0,\ \forall A\notin \{ 0, 1, 4\} .
$$
Combining the 6th  and 8th relation 
(\ref{l-45-1}) 
and using  (\ref{sec-l-45}) again, 
we obtain 
\begin{equation}\label{delta-1-1}
\delta_{4} = \delta_{1} = \epsilon_{1} ( \beta - 2\alpha ).
\end{equation}
In fact, the 5th, 6th, 8th and 9th relation (\ref{l-45-1}) 
reduce  to (\ref{delta-1-1}) 
and 
\begin{equation}\label{f-23-square-l} 
(F_{23})^{2} = \epsilon_{1} \beta (\alpha -\beta ).
\end{equation}
Replacing in (\ref{f-23-square-l})  the expression of $F_{23}$ given by
(\ref{expr-F-l-45}) and using (\ref{delta-1-1}) we obtain 
\begin{equation}
(\delta_{0})^{2} (\alpha -\beta ) = \epsilon_{1}\beta (\beta -2\alpha  )^{2}.
\end{equation}
(Recall that, after equation \eqref{f-23-new}, we showed that $\beta \neq 0$.)
Combining (\ref{delta-1-1}) with the 3rd equality 
(\ref{expr-F-l-45}) we obtain 
\begin{equation}
F_{23} = \frac{\beta ( 2\alpha -\beta )}{ \delta_{0}}.
\end{equation}
We arrive at the 2nd class of solutions listed in the theorem. This finishes case I).\

$\bullet$ case II):  $F_{23} =0.$  Since $F\neq 0$, either $F_{12}$ or $F_{13}$ is non-zero.
From relation (\ref{primele-1}) we deduce that 
\begin{equation}\label{beta-delta}
\beta \delta_{1} =0.
\end{equation} 
We distinguish two subcases, namely $\beta=0$ or $\beta \neq 0$.\

$\bullet$ case II-1):  $\beta =0.$ 
Recall, from relation (\ref{beta-alpha-delta}), that 
$ \delta_{2} -\delta_{5} = \delta_{3} -\delta_{6}$.
We claim that
\begin{equation}\label{delta-2-5-3}
\delta_{2} = \delta_{5},\ \delta_{3} = \delta_{6}.
\end{equation}
To prove the claim we assume, by contradiction,  that 
\begin{equation}\label{abs}
\delta_{3} - \delta_{6} = \delta_{2} - \delta_{5}\neq 0.
\end{equation}
From  the 2nd relation (\ref{beta-alpha-delta}) we obtain that $\alpha =0.$
The 1st relation (\ref{l-45-3}) implies  that  $F_{12} = F_{13}$. 
The 1st  relation (\ref{l-45-2}) with $F_{23} =0$ and $F_{12} = F_{13}\neq 0$ implies that $\delta_{2} =\delta_{3}.$ So 
$\delta_{2} =\delta_{3}$, $\delta_{5} =\delta_{6}$ and $\delta_{2} \neq \delta_{5}$. 
The 1st  relation (\ref{l-45-1}) with $F_{12} = F_{13}$ 
and $\beta =0$ 
implies that 
$h =0.$  Then the  2nd  relation (\ref{l-45-2}) implies that $\delta_{1} =0$  and  the  5th relation  (\ref{l-45-1}) implies that $F_{12} =0$.
We obtain  a contradiction.  Relation (\ref{delta-2-5-3}) follows.

Combining  the 2nd relation (\ref{l-45-1}) with the 1st relation
(\ref{l-45-4}), and, respectively, the 3rd  relation (\ref{l-45-1}) with the 2nd relation (\ref{l-45-4}), we obtain
\begin{equation}\label{has}
( h+2\alpha ) \delta_{6} = ( h+2\alpha )\delta_{5} =0.
\end{equation}
In view of relation (\ref{has}) we distinguish two subcases, namely  $h + 2\alpha \neq 0$ or   $h+2\alpha =0.$ If $h+ 2\alpha \neq 0$ then, 
 from relations (\ref{delta-2-5-3}) and  (\ref{has}), 
$$
\delta_{2} = \delta_{3} = \delta_{5} = \delta_{6} =0.
$$
The 2nd  and 3rd relation (\ref{l-45-1})
(with $\delta_{2} =\delta_{3} = F_{23} =  0$  and $F\neq 0$)  imply that $\delta_{0} =0.$ 
Thus
\begin{equation}
F_{23} = \beta = 0,\ h +2\alpha \neq 0,\  \delta_{A} =0,\ \forall A\notin\{ 1, 4\},\  \delta_{1} =\delta_{4}.
\end{equation}
If $h =0$, the  1st, 5th and  6th relation (\ref{l-45-1}) imply that $F_{12} = F_{13}$ and 
the 3rd  relation (\ref{l-45-2}) implies that $\delta_{1}  = - \epsilon_{1}\alpha$.
The 5th  relation (\ref{l-45-1}) implies that
$(F_{12})^{2} =  - \frac{\epsilon_{1}\alpha  \eta}{2}.$ This leads to the 3rd class of solutions  listed in the theorem. 
If $h\neq 0$ then, from 6th  and 8th relations  (\ref{l-45-1}) we obtain that $\delta_{1} =0.$ 
The 2nd and 3rd relation (\ref{l-45-2})  with $\delta_{1} =0$ imply that $F_{12} = F_{13}$ (and $h = -\alpha$). 
From the  1st  relation (\ref{l-45-1})  we obtain that $h=0$, which contradicts our assumption.
Similar  computations show that  $h +2\alpha =0$ leads  to  a contradiction ($F=0$).
We concluded the case II-1).\

$\bullet$ case II-2):  $\beta \neq 0.$ From  relation (\ref{beta-delta})  we obtain  $\delta_{1} =0.$   
The 2nd  and 3rd relations (\ref{l-45-2})  imply, as before, that  $h=-\alpha$ and $F_{12} = F_{13}.$ 
The  5th  and 9th relations  (\ref{l-45-1})  imply that 
$\alpha = \beta .$ Similar computations as above lead to the 4th  class of solutions  listed in  the theorem.\

ii) Claim ii) follows from Lemma \ref{identif-lie-alg} ii).
\end{proof}

\subsection{The case $L = L_{5}$}

We consider  the setting from  the beginning of Section \ref{basic_facts_unimod:sec}, but  we assume that 
$L = L_{5} (\alpha )$ in the $g$-orthonormal basis $(v_{a})$,  
where $\alpha \in  \mathbb{R}$.  Since $L$ is non-diagonal, $\epsilon_{1} = \epsilon_{2} = -\epsilon_{3}.$

\begin{thm}\label{classif-l5}  The metric $E_{- } =\{ X - g(X)\mid  X\in TG\}$ defined on $E_{H, F}$ is generalized Einstein with divergence 
$\delta$ if and only if 
\begin{align}
\nonumber& \alpha =h= F_{13} = 0,\  F_{23} = - F_{12},\ (F_{12})^{2} = -\epsilon_{1} -\frac{\delta_{1}}{ \sqrt{2}}\\
\nonumber& \delta_{0} = \delta_{2} =\delta_{5} =0,\  \delta_{1} =\delta_{3} =\delta_{4} = \delta_{6},\  \epsilon_{1} +\frac{\delta_{1}}{ \sqrt{2}} < 0.
\end{align}
The Lie algebra $\mathfrak{g}$ is isomorphic to $\mathfrak{e}(1,1).$
\end{thm}

\begin{proof}
From Remark \ref{rem-dorfman} ii), 
obtain the Dorfman coefficients  $B_{ABC}$
($1\leq A, B, C\leq 6$) 
in the basis $(e_{A})_{0\leq A\leq 6}$ 
by replacing  in the Dorfman coefficients computed in the proof of Proposition 
3.3 of \cite{c-k-paper}
$h\mapsto - h$ and 
 $ \beta  , \gamma \mapsto \alpha$ and $\lambda , \mu \mapsto \frac{1}{\sqrt{2}}$ (and  leaving 
$\alpha$ unchanged): 
\begin{align}
\nonumber& B_{145} =0,\ B_{146} =-\frac{1}{\sqrt{2}},\ 
B_{156} = -\frac{1}{2} ( h+\alpha ),\ 
B_{245} =-\frac{1}{\sqrt{2}},\ B_{246} = \frac{1}{2} ( h+\alpha ),\\
\nonumber &  B_{256}
=\frac{1}{\sqrt{2}},\ B_{345} =-\frac{1}{2} ( h+\alpha ),\  B_{346}  = -\frac{1}{ \sqrt{2}},\ B_{356}  = B_{412} =0,\  B_{413} =\frac{1}{\sqrt{2}},\\
\nonumber&  B_{423} = \frac{1}{2} ( \alpha - h),\ B_{512} = \frac{1}{\sqrt{2}},\ B_{513} =\frac{1}{2} ( h-\alpha ),\ 
B_{523} =-\frac{1}{\sqrt{2}},\\
\label{dorfman-l5}&  B_{612} = \frac{1}{2} ( \alpha - h),\ 
B_{613} =\frac{1}{ \sqrt{2}},\ B_{623} =0. 
\end{align}
From Remark \ref{rem-dorfman} iii)  we obtain the coefficients  
$R_{ia}^{\tilde{\delta} , +}$ 
of the Ricci tensor 
$\mathrm{Ric}^{\tilde{\delta}, + }$ from Corollary \ref{forapplic}
by replacing  $h\mapsto - h$  (and $ \beta  , \gamma \mapsto \alpha$, $\lambda , \mu \mapsto \frac{1}{\sqrt{2}}$) in the Ricci coefficients computed 
in   the proof of Proposition 3.11  of  \cite{c-k-paper}: 
\begin{align}
\nonumber& R_{41}^{\tilde{\delta}, +} = - 1-\frac{\alpha^{2}} {2} +\frac{h^{2}}{2} - \epsilon_{1} \frac{\delta_{3}}{\sqrt{2}},\
R_{42}^{\tilde{\delta}, +} = -\frac{\alpha}{ \sqrt{2}} +\frac{\epsilon_{1}}{2} ( h-\alpha ) \delta_{3}\\
\nonumber& R_{43}^{\tilde{\delta}, +} = - 1-\frac{\epsilon_{1}}{\sqrt{2}} \delta_{1} +\frac{\epsilon_{1}}{2} ( h-\alpha  )\delta_{2},\ 
R_{51}^{\tilde{\delta}, +}  = -\frac{\alpha}{\sqrt{2}} +\frac{\epsilon_{1}}{\sqrt{2}} \delta_{2} + \frac{\epsilon_{1}}{2} ( \alpha  - h) \delta_{3}\\
\nonumber& R_{52}^{\tilde{\delta}, +} = -\frac{\alpha^{2}}{ 2} +\frac{ h^{2}}{2} -\frac{\epsilon_{1}}{ \sqrt{2}} \delta_{1} +
\frac{\epsilon_{1}}{\sqrt{2}} \delta_{3},\ R_{53}^{\tilde{\delta}, +} = -\frac{\alpha}{\sqrt{2}} + \frac{\epsilon_{1}}{2} ( \alpha  - h) \delta_{1} +\frac{\epsilon_{1}}{ \sqrt{2}} \delta_{2}\\
\nonumber& R_{61}^{\tilde{\delta}, +} = -1 +\frac{\epsilon_{1}}{2} ( \alpha - h) \delta_{2} -\frac{\epsilon_{1}}{\sqrt{2}}\delta_{3},\ 
R_{62}^{\tilde{\delta}, +} = -\frac{\alpha}{ \sqrt{2}} + \frac{\epsilon_{1}}{ 2} ( h-\alpha ) \delta_{1}\\
\label{ricci-l5}& R_{63}^{\tilde{\delta}, +} = -1 +\frac{\alpha^{2}}{2} -\frac{h^{2}}{ 2} -\frac{\epsilon_{1}}{\sqrt{2}} \delta_{1} .
 \end{align}
Using  relations (\ref{dorfman-l5}) and (\ref{ricci-l5}) and  $\epsilon_{1} =\epsilon_{2} = -\epsilon_{3}$, 
we obtain from Corollary~\ref{forapplic} that $E_{-}$ is generalized Einstein with divergence $\delta$  if and only if
relations
\begin{align}
\nonumber& -1 -\frac{\alpha^{2}}{ 2} +\frac{h^{2}}{2} -\frac{\epsilon_{1}}{\sqrt{2}}\delta_{3} = \epsilon_{1} ( (F_{12})^{2} - (F_{13})^{2})\\
\nonumber& -\frac{\alpha}{ \sqrt{2}} +\frac{\epsilon_{1}}{ \sqrt{2}} \delta_{2} +\frac{\epsilon_{1}}{2} ( \alpha - h) \delta_{3}
=\delta_{0} F_{12} - \epsilon_{1} F_{23} F_{13}\\
\nonumber& - 1 +\frac{\epsilon_{1}}{2} (\alpha - h) \delta_{2} -\frac{\epsilon_{1}}{\sqrt{2}}\delta_{3} = \delta_{0} F_{13}- \epsilon_{1}
F_{23} F_{12}\\
 \nonumber& -\frac{\alpha}{\sqrt{2}} +\frac{\epsilon_{1}}{2} ( h-\alpha ) \delta_{3 } = -\delta_{0} F_{12} -\epsilon_{1} F_{13} F_{23}\\
 \nonumber& -\frac{\alpha^{2}}{ 2} +\frac{h^{2}}{2} -\frac{\epsilon_{1}}{\sqrt{2}} \delta_{1} +\frac{\epsilon_{1}}{\sqrt{2}} \delta_{3}
 = \epsilon_{1} ( (F_{12})^{2} - (F_{23})^{2})\\
 \nonumber& -\frac{\alpha}{ \sqrt{2}} +\frac{\epsilon_{1}}{2} ( h-\alpha ) \delta_{1} = \delta_{0} F_{23} +\epsilon_{1} F_{12} F_{13}\\
 \nonumber& - 1 -\frac{\epsilon_{1}}{\sqrt{2}} \delta_{1} +\frac{\epsilon_{1}}{2} ( h-\alpha ) \delta_{2} = -\delta_{0} F_{13} -\epsilon_{1} F_{12} F_{23}\\
 \nonumber& -\frac{\alpha}{\sqrt{2}} +\frac{\epsilon_{1}}{2} (\alpha - h) \delta_{1}+ \frac{\epsilon_{1}}{\sqrt{2}} \delta_{2} = -\delta_{0} F_{23}
 +\epsilon_{1} F_{12} F_{13}\\
 \label{rel-l1} & -1 +\frac{\alpha^{2}}{ 2} -\frac{h^{2}}{2} -\frac{\epsilon_{1}}{\sqrt{2}} \delta_{1} = \epsilon_{1} ( (F_{13})^{2} 
 + (F_{23})^{2})
 \end{align}
 (which are equivalent to  relations (\ref{rel-1})),  together with 
 \begin{align}
 \nonumber& F_{12}\epsilon_{1} \delta_{2} - F_{13} (\epsilon_{1} \delta_{3} +\frac{1}{\sqrt{2}}) + F_{23} ( h+\alpha ) =0\\
 \nonumber& F_{12} (\epsilon_{1} \delta_{1}- \frac{1}{\sqrt{2}} ) + F_{13} ( h+\alpha )  + F_{23} ( \epsilon_{1} \delta_{3} -\frac{1}{\sqrt{2}} ) =0\\
  \nonumber& F_{12} ( h+\alpha ) + F_{13} ( \epsilon_{1} \delta_{1} +\frac{1}{\sqrt{2}} ) + F_{23} \epsilon_{1} \delta_{2} =0\\
  \nonumber& F_{12} (\delta_{2} - \delta_{5}) = F_{13} (\delta_{3} -\delta_{6})\\
  \nonumber& F_{12} (\delta_{1} - \delta_{4}) + F_{23} (\delta_{3} - \delta_{6}) =0\\
  \label{rel-l2}& F_{13} (\delta_{1} -\delta_{4}) + F_{23} (\delta_{2} -\delta_{5}) =0
  \end{align}
   (which are equivalent to (\ref{rel-2})), together with
   \begin{align}
 \nonumber&   \frac{1}{\sqrt{2}} \delta_{2} +\frac{1}{2} (\alpha - h) \delta_{3} +\frac{1}{2} ( h+\alpha ) \delta_{6} = \epsilon_{1} \delta_{0} F_{12}\\
 \nonumber& \frac{1}{2} (\alpha - h)\delta_{2} -\frac{1}{\sqrt{2}} \delta_{3} +\frac{1}{\sqrt{2}} \delta_{4} +\frac{1}{2} ( h+\alpha ) \delta_{5}=
 \epsilon_{1} \delta_{0} F_{13}\\
 \label{rel-l3}& \frac{1}{2} ( h-\alpha ) \delta_{1} -\frac{1}{2} ( h+\alpha ) \delta_{4} -\frac{1}{\sqrt{2}} \delta_{5} =\epsilon_{1}  \delta_{0} F_{23}
 \end{align}
 (which are equivalent to  relations (\ref{rel-3})),  together with 
 \begin{equation}\label{delta}
 \delta_{1} = \delta_{4},\ \delta_{2} =\delta_{5},\ \delta_{3} =\delta_{6}
 \end{equation}
 (which are equivalent to  relations (\ref{rel-4})) hold. Using (\ref{delta}), we obtain that the 4th, 5th  and 6th relation (\ref{rel-l2})  are  satisfied
 and relations (\ref{rel-l3}) become
 \begin{align}
\nonumber&  \epsilon_{1} \delta_{0} F_{12} = \frac{1}{\sqrt{2}} \delta_{2} +\alpha \delta_{3},\\
 \nonumber& \epsilon_{1} \delta_{0} F_{13} = \alpha \delta_{2} + \frac{1}{\sqrt{2}} (\delta_{1} -  \delta_{3}),\\
\label{delta-1-new}&  \epsilon_{1} \delta_{0} F_{23} = - ( \alpha \delta_{1} + \frac{1}{\sqrt{2}} \delta_{2}).
 \end{align}
 We substract  the  4th from the 2nd relation (\ref{rel-l1}) and  we combine it with the 1st  relation 
 (\ref{delta-1-new}). Similarly, we substract the 7th from the 3rd relation (\ref{rel-l1}) and we combine it with the
 2nd relation (\ref{delta-1-new}). We obtain
 \begin{equation}\label{delta-i}
 \delta_{2} = - \sqrt{2} (h+\alpha )\delta_{3},\ \delta_{1} = \left( 1 + 2( h+\alpha )^{2} \right) \delta_{3}. 
 \end{equation}
 We substract the 8th from the 6th relation (\ref{rel-l1}) and we obtain
 \begin{equation}
 \delta_{0} F_{23} =\frac{\epsilon_{1}}{2} ( h-\alpha ) \delta_{1} -\frac{\epsilon_{1}}{ 2\sqrt{2}} \delta_{2}.
 \end{equation}
 Combined with the 3rd relation (\ref{delta-1-new}), it gives
 \begin{equation}\label{delta-ii} 
 \delta_{2} = - \sqrt{2} ( h+\alpha ) \delta_{1}.
 \end{equation}
To summarize our argument so far:  we proved that 
$E_{-}$ is generalized Einstein with divergence $\delta$ if and only  if 
relations
(\ref{rel-l1}), the first three relations (\ref{rel-l2}), together with relations 
(\ref{delta}), (\ref{delta-i}) and (\ref{delta-ii}) hold.  Remark that the 1st relation (\ref{delta-i}) and relation (\ref{delta-ii}) imply that 
\begin{equation}\label{h+a}
(h+\alpha )(\delta_{1} -\delta_{3}) =0. 
\end{equation}
In view of relation (\ref{h+a}) we consider two cases,  namely $h+\alpha =0$ or $h+\alpha\neq 0.$\

$\bullet$ case I): $ h+\alpha =0.$ Relations (\ref{delta-i}) imply that $\delta_{2} =0$ and $\delta_{1} =\delta_{3}.$ From 
relation (\ref{delta}), we obtain
\begin{equation}\label{delta-multe} 
\delta_{2} = \delta_{5}=0,\ \delta_{1} =\delta_{3} =\delta_{4} =\delta_{6}.
\end{equation} 
The 1st, 5th and 9th relation (\ref{rel-l1})  
become
\begin{equation}\label{modif-l5}
F_{13} =0,\ (F_{12})^{2} = (F_{23})^{2} = -\epsilon_{1} -\frac{\delta_{1}}{\sqrt{2}}.
\end{equation}
The 6th and 8th relation (\ref{rel-l1})  together with $h+\alpha =0$ and $F_{23} \neq 0$
(because $F\neq 0$) 
 imply 
$$
h =\alpha = \delta_{0}=0.
$$ 
The 7th relation (\ref{rel-l1}) implies that  $F_{12} F_{23} = \epsilon_{1} +\frac{1}{\sqrt{2}} \delta_{1}$. Combined 
with (\ref{modif-l5}) it implies $F_{12} = - F_{23}$. 
We obtain the class of solutions from the statement of the theorem.\

$\bullet$ case II): $h+\alpha \neq 0.$ 
From relations  (\ref{delta}),  (\ref{delta-i}) and (\ref{h+a})  
we deduce that
 $\delta_{A} =0$ for any $A\neq 0.$  From relations (\ref{rel-l3}) and $F\neq 0$ we obtain  that 
 $\delta_{0} =0.$ From the 1st and 3rd  relation (\ref{rel-l2}) we obtain that $F_{13} \neq 0$. 
From the 2nd and 6th relation (\ref{rel-l1}) we obtain that $F_{23} = - F_{12}.$  But then  the 2nd relation
(\ref{rel-l2}) implies that $(h+\alpha ) F_{13} =0$, which is a contradiction.\

The last statement of the theorem follows from Lemma \ref{identif-lie-alg} iii).
 \end{proof}

\section{The classification when $\mathfrak{g}$ is non-unimodular}\label{non-unimod-section}

Let  $\mathfrak{g}$ be  a  $3$-dimensional non-unimodular  Lie algebra 
with unimodular kernel $\mathfrak{u}:= \{ x\in \mathfrak{g} \mid \mathrm{tr}\, \mathrm{ad}_{x} =0 \}$.
Since $\mathfrak{u}$ is a $2$-dimensional abelian ideal of $\mathfrak{g}$ (see e.g.\  \cite{milnor}),
 $\mathfrak{g} $ is a semi-direct product 
 $\mathbb{R}\ltimes_{\mathbf A} \mathbb{R}^{2}$ 
 of $\mathbb{R}$ and $\mathbb{R}^{2}$, 
where $\mathbb{R}$ acts on $\mathbb{R}^{2}$ by an endomorphism ${\mathbf A}\in \mathrm{End}\, \mathbb{R}^{2}$ of non-zero trace.
The endomorphism $\mathbf{A}$ is uniquely determined by $\mathfrak{g}$ up to conjugation and multiplication by a non-zero scalar.
It follows that there are four classes of non-unimodular $3$-dimensional Lie algebras (see \cite{vinb}, page 212),
according to the following four classes of canonical  forms for   the operator $A$:
\begin{equation}\label{canonical-forms-A}
\mathbf{A} = \left( \begin{tabular}{cc}
$1$ & $0$\\
$0$ & $0$
\end{tabular}
\right) ,\
\mathbf{A} = \left( \begin{tabular}{cc}
$1$ & $1$\\
$0$ & $1$
\end{tabular}
\right) ,\
\mathbf{A} = \left( \begin{tabular}{cc}
$1$ & $0$\\
$0$ & $\lambda$
\end{tabular}
\right) ,\
\mathbf{A} = \left( \begin{tabular}{cc}
$\lambda$ & $-1$\\
$1$ & $\lambda$
\end{tabular}
\right) ,
\end{equation}
where in the third matrix $\lambda \in \mathbb{R}\setminus \{ -1\}$, $ 0 < | \lambda | \leq 1$ and in the fourth matrix 
$\lambda \in \mathbb{R}\setminus \{ 0\}$. Following \cite{vinb} we shall denote the corresponding Lie algebras by
$\tau_{2} (\mathbb{R}) \oplus \mathbb{R}$, $ \tau_{3} (\mathbb{R})$, $\tau_{3, \lambda }(\mathbb{R})$ and 
$\tau^{\prime}_{3, \lambda }(\mathbb{R})$ respectively.\

Let $G$ be a Lie group with Lie algebra $\mathfrak{g}$
and $E_{H, F}$ a  Courant algebroid of type $B_{3}$ over $G$. 
When a basis $(v_{a})_{1\leq a\leq 3}$ of $\mathfrak{g}$ is fixed, we write as in the previous sections
$$
H= h v_{1}^{*}\wedge v_{2}^{*}\wedge v_{3}^{*},\ F= \frac12 F_{ab} v_{a}^{*}\wedge v_{b}^{*},
$$ where
$( v_{a}^{*})$ is the dual basis of $(v_{a})$ and $h, F_{ab}\in \mathbb{R}.$

Let 
$g$ be a non-degenerate scalar product on $\mathfrak{g}$, 
$\delta\in
(\mathfrak{g}\oplus\mathfrak{g}^{*}\oplus \mathbb{R})^{*}$  a left-invariant divergence operator and $\delta_{A}:= \delta (e_{A})$, 
where $(e_{A})$ is the basis of $\mathfrak{g}\oplus \mathfrak{g}^{*}\oplus \mathbb{R}$ defined by $(v_{a}).$ 
In the next section we assume that 
$g\vert_{\mathfrak{u}\times \mathfrak{u}}$ is degenerate. The case when $g\vert_{\mathfrak{u}\times\mathfrak{u}}$ 
is non-degenerate 
will be treated in Section~\ref{nondeg-sect}.

\subsection{The case $g\vert_{\mathfrak{u}\times \mathfrak{u}}$- degenerate}\label{degenerate-section}

As proved in  Proposition 3.7 of \cite{c-k-paper},
when $g\vert_{\mathfrak{u}\times \mathfrak{u}}$ is degenerate 
there is a
$g$-orthonormal basis 
$(v_{a} )$ of $\mathfrak{g}$ with $\epsilon_{1} = \epsilon_{2} = -\epsilon_{3}$
and 
\begin{align}
\nonumber& {\mathcal L}_{v_{1}} v_{2}  =
\epsilon_{1} ( \lambda v_{1} +\mu v_{2} +\mu v_{3})\\
\nonumber& \mathcal L_{v_{2}} v_{3}  
= - \epsilon_{1} ( \nu v_{1} +\rho v_{2} +\rho v_{3})\\
\label{bracket-deg}& \mathcal L_{v_{3}} v_{1} 
= \epsilon_{1} ( \lambda v_{1} +\mu  v_{2} +\mu v_{3}),
\end{align}
where $\lambda , \mu , \rho \in \mathbb{R}$,  $\lambda +\rho \neq 0$ and $\epsilon_{i}:= g(v_{i}, v_{i}).$  
The Lie algebra $\mathfrak{g}$ is the semi-direct product $\mathbb{R}\ltimes_{\mathbf{A}} \mathbb{R}^{2}$  where 
\begin{equation}\label{form-A}
\mathbf{A} =  -\epsilon_{1} \left( \begin{tabular}{cc}
$\lambda$ & $\frac{ \nu}{2}$\\
$2\mu$ & $\rho$
\end{tabular}\right) .
\end{equation}
\begin{rem}\label{corrected}{\rm 
In \cite{c-k-paper}, the   minus sign in  the expression  of  $ \mathcal L_{v_{2}} v_{3} $
from (\ref{bracket-deg}) 
 was forgotten.  As a consequence,  the Lie algebra  coefficients $k_{23a}$  ($1\leq a\leq 3$) from  the proof of Proposition  3.7 of \cite{c-k-paper} need to be 
replaced  with  $-k_{23a}$.  The correct expressions of the  Dorfman  coefficients and curvature components  from the proof of Propositions  3.7 and 3.16 of
\cite{c-k-paper}
are given in 
the proof of Theorem \ref{deg-thm} below. }
\end{rem}
From (\ref{bracket-deg}) we obtain that a  left-invariant $2$-form $F\in \Omega^{2}(G)$  is closed  if and only if 
\begin{equation}\label{F-closed-deg}
F_{12} + F_{13}=0.
\end{equation}
In particular, the twisting $2$-form $F$ of $E_{H, F}$ satisfies (\ref{F-closed-deg}).

\begin{thm}\label{deg-thm}  There are two families  1) and 2) (see below) of generalized Einstein metrics  $E_{-} = \{ X - g(X)\mid  X\in TG\}$  
with divergence $\delta$ defined on $E_{H, F}$.
In the basis $(v_{a})$ above, they are described as follows:\

1) $\lambda +\rho \neq 0$, $\nu = F_{23} = h =0$,  $F_{13} = - F_{12}$, 
\begin{equation}
(F_{12})^{2} = -\epsilon_{1}  (\lambda^{2}+\rho^{2}) +\rho \delta_{2} > 0
\end{equation}
and 
\begin{equation}
\rho (\delta_{2} -\delta_{5}) =0, \ \delta_{3} = -\delta_{2},\  \delta_{6} = -\delta_{5},\ \delta_{0} = \delta_{1}=\delta_{4} =0. 
\end{equation}

2) $\nu (\lambda +\rho ) \neq 0$,  $F_{23} = \epsilon\nu$ (where $\epsilon = \pm 1$),  $F_{12} = - F_{13} = -\epsilon \lambda$, $\epsilon_{1} =-1$, 
$ h =\nu$,  $\delta_{0} = -\epsilon \delta_{1}$ and
\begin{align}
\nonumber&(\lambda \rho - \mu \nu ) (\delta_{1} +\nu ) = (\lambda \rho - \mu \nu ) (\delta_{1} -\delta_{4}) =0\\
\label{cond-ii}& \delta_{3} = -\delta_{2} = \lambda +\rho +\frac{\lambda}{\nu} \delta_{1},\ \delta_{6} = -\delta_{5} = \lambda +\rho +\frac{\lambda}{\nu} \delta_{4}.
\end{align}
 \end{thm}

\begin{proof} From Remarks \ref{compare}
and \ref{corrected}, 
we obtain the Dorfman coefficients  $B_{ABC}$ ($1\leq A, B, C\leq 6$)  in the basis $(e_{A})_{0\leq A\leq 6}$ of
$\tilde{E} = \mathfrak{g}\oplus \mathfrak{g}^{*} \oplus \mathbb{R}$ 
determined by the basis $(v_{a})$ 
by replacing $h\mapsto - h$,
$\nu\mapsto - \nu$, $\rho \mapsto -\rho$ and leaving $\lambda$, $\mu$ unchanged 
in the formulae from  the proof of Proposition 3.7 of \cite{c-k-paper}: 
\begin{align}
\nonumber& B_{145} =\lambda ,\ B_{146} = -\lambda ,\ B_{156} = - \frac{1}{2} ( \nu +  h),\  B_{245} = \mu ,\  B_{613} = -\mu,\  B_{623} = - \rho ,\\
\nonumber& B_{246} = \frac{1}{2} ( h -  2\mu - \nu ),\ B_{256} = - \rho ,\ B_{345} = \frac{1}{2} ( - h-2\mu +\nu ),\  B_{346} = \mu\\
\nonumber&  B_{356} = \rho ,\ B_{412}= -\lambda ,\ B_{413}=\lambda ,\ B_{423} = \frac{1}{2} (\nu -h),\  B_{512} = -\mu ,\\
\label{dorf-deg}&  B_{513} =\frac{1}{2}( h+2\mu + \nu ),\ B_{523}= \rho,\ B_{612}= - \frac{1}{2} ( h -2\mu +\nu ).
\end{align}
The   components  $R_{ia}^{\tilde{\delta}, +}$ 
(with $1\leq a\leq 3$ and $4\leq j\leq 6$) 
of the Ricci tensor $\mathrm{Ric}^{\tilde{\delta}, +}$  
from Corollary \ref{forapplic}  are computed from Lemma \ref{ricci-diferit}, 
by using  relations (\ref{dorf-deg}) and 
\begin{equation}\label{sus-jos}
B_{bi}^{j} = -\epsilon_{j-3} B_{bij},\ B_{aj}^{b} = \epsilon_{b} B_{ajb},\ B_{ia}^{c} =\epsilon_{c} B_{iac}.
\end{equation}
Since   $\epsilon_{1} = \epsilon_{2} = -\epsilon_{3}$ we obtain
\begin{align}
\nonumber& R_{41}^{\tilde{\delta}, +} =  \frac{1}{2} ( h^{2} -\nu^{2}) -\epsilon_{1}\lambda ( \delta_{2}+\delta_{3}) \\
\nonumber&R_{42}^{\tilde{\delta}, +} =
\frac{\nu}{2} (\lambda -\rho ) +\frac{h}{2} (\lambda +\rho ) +\epsilon_{1} \lambda \delta_{1} 
+\frac{\epsilon_{1}}{2} (h- \nu )\delta_{3}\\
\nonumber& R_{43}^{\tilde{\delta}, +} =
-\frac{h}{2} ( \lambda +\rho ) -\frac{\nu}{2}  (\lambda -\rho )
-\epsilon_{1} \lambda\delta_{1} +\frac{\epsilon_{1}}{2} (h- \nu ) \delta_{2}\\
\nonumber& R_{51}^{\tilde{\delta}, +} = \frac{\nu}{2} (\lambda -\rho  )  -\frac{h}{2} (\lambda +\rho ) 
-\epsilon_{1} \mu \delta_{2} -  \frac{\epsilon_{1}}{2} ( h+ 2\mu + \nu ) \delta_{3}\\
\nonumber& R_{52}^{\tilde{\delta}, +} = -\lambda^{2} - \rho^{2} +
\frac{1}{2} ( h^{2} +\nu^{2}) -\mu \nu +
\epsilon_{1}  ( \mu \delta_{1} - \rho \delta_{3})\\
\nonumber& R_{53}^{\tilde{\delta}, +} = \lambda^{2} +  \rho^{2} +  \mu \nu -  \frac{\epsilon_{1}}{2} ( h +  2\mu + \nu )\delta_{1} -
\epsilon_{1} \rho \delta_{2}\\
\nonumber& R_{61}^{\tilde{\delta}, +} =\frac{h}{2} (\lambda +\rho )  -\frac{\nu}{2}  (\lambda -\rho )
- \frac{\epsilon_{1}}{2} ( h -  2\mu + \nu ) \delta_{2} +\epsilon_{1} \mu \delta_{3}\\
\nonumber& R_{62}^{\tilde{\delta}, +} =\lambda^{2} +\rho^{2} +\mu \nu +\frac{\epsilon_{1}}{2} ( h -  2\mu + \nu ) \delta_{1} + \epsilon_{1} \rho \delta_{3}\\
\label{ricci-deg}& R_{63}^{\tilde{\delta}, +} = -\lambda^{2}  -\rho^{2} -\frac{1}{2} ( h^{2} +  \nu^{2}) 
- \mu \nu 
 +\epsilon_{1} (  \mu \delta_{1} +  \rho \delta_{2}).
\end{align}
Using  now $F_{13} = - F_{12}$,  $\epsilon_{1} =\epsilon_{2} =-\epsilon_{3}$ and  relations 
 (\ref{ricci-deg}),  
 it is easy to see that relations (\ref{rel-1}) from Corollary 
 \ref{forapplic} are equivalent to 
\begin{align}
\nonumber& h^{2} -\nu^{2}  =  2 \epsilon_{1}\lambda ( \delta_{2}+\delta_{3})\\
\nonumber&    h (\lambda +\rho ) 
+ \epsilon_{1}\mu (\delta_{2}+\delta_{3}) + \epsilon_{1} (  \lambda \delta_{1}  +  h \delta_{3} )
= -2 \delta_{0}  F_{12}\\
\nonumber& (h+\nu )(\delta_{2} +\delta_{3}) =0\\
\nonumber& (\lambda -\rho )\nu  -\epsilon_{1} \mu (\delta_{2}+\delta_{3}) +\epsilon_{1} ( \lambda \delta_{1} - \nu \delta_{3})  
= 2\epsilon_{1} F_{12} F_{23}\\
 \nonumber&    \lambda^{2} + \rho^{2}  +\mu \nu 
- \epsilon_{1}  \mu \delta_{1}  - \frac{\epsilon_{1} \rho}{2} (\delta_{2} -\delta_{3}) = - \epsilon_{1} (F_{12})^{2}\\
\nonumber& ( h -\nu ) (\delta_{2} +\delta_{3}) =0\\
\nonumber&  ( h+\nu ) \delta_{1} + \rho (\delta_{2}+ \delta_{3}) = 2\epsilon_{1} \delta_{0} F_{23}\\
\label{1-deg} & ( h^{2} +\nu^{2}) - \epsilon_{1} \rho (\delta_{2} +\delta_{3}) =- 2 \epsilon_{1} (F_{23})^{2}.
\end{align}
We now write the remaining relations 
(\ref{rel-2}), 
(\ref{rel-3}) and (\ref{rel-4}) from Corollary  \ref{forapplic}. For this, we  first notice that 
they imply 
\begin{equation}\label{indep-F}
\delta_{3} - \delta_{6} = - (\delta_{2} -\delta_{5}).
\end{equation}
In order to prove relation (\ref{indep-F})  we consider the first three relations (\ref{rel-4}). Using  the Dorfman coefficients given in (\ref{dorf-deg}), 
they become: 
\begin{align}
\nonumber& \lambda ( \delta_{2} + \delta_{3} -\delta_{5} -\delta_{6}) =0\\
\nonumber& \mu (\delta_{1} -\delta_{4}) -\rho ( \delta_{3} -\delta_{6}) =0\\
\label{three-F}& \mu ( \delta_{1} -\delta_{4}) +\rho (\delta_{2} -\delta_{5}) =0. 
\end{align}
The last two relations (\ref{three-F}) imply that
$$
\rho ( \delta_{2} + \delta_{3} -\delta_{5} -\delta_{6})  =0,
$$
and, since $\lambda +\rho \neq 0$, we obtain relation  (\ref{indep-F}), as required. 
Using 
$F_{13} = - F_{12}$,  $\epsilon_{1} =\epsilon_{2} =-\epsilon_{3}$ and  relations 
 (\ref{dorf-deg}),  (\ref{ricci-deg}) and (\ref{indep-F}), we obtain that   
relations 
(\ref{rel-2}), 
(\ref{rel-3}) and (\ref{rel-4}) from Corollary  \ref{forapplic} are equivalent to 
\begin{align}
\nonumber& \epsilon_{1} F_{12} (\delta_{2} +\delta_{3}) = F_{23} (\nu - h)\\
\nonumber& F_{12} ( h  -\epsilon_{1} \delta_{1}) = F_{23} (\rho + \epsilon_{1} \delta_{3} )\\
\nonumber& F_{23} (\delta_{2}+ \delta_{3}) =0\\ 
\nonumber& \delta_{3} -\delta_{6} = - ( \delta_{2} -\delta_{5})\\
\nonumber& F_{12} ( \delta_{1} -\delta_{4}) = F_{23} ( \delta_{2} -\delta_{5})\\
\nonumber&   (h+\nu) (\delta_{2} -\delta_{5}) - 2\mu (\delta_{2} + \delta_{3}) -2\lambda\delta_{4} =2\epsilon_{1}\delta_{0} F_{12}\\
\nonumber&  ( h- 2\mu  ) ( \delta_{1}-\delta_{4})   +\nu (\delta_{1} +\delta_{4}) + 2\rho ( \delta_{3} + \delta_{5}) 
= 2\epsilon_{1} \delta_{0} F_{23}\\
\nonumber& \mu ( \delta_{1} - \delta_{4}) + \rho ( \delta_{2} -\delta_{5}) =0\\
\label{2-deg}& \lambda (\delta_{1} -\delta_{4}) +\nu (\delta_{2} -\delta_{5}) =0.
\end{align}
We proved that $E_{-}$ is generalized Einstein with divergence $\delta$ if and only if relations
(\ref{1-deg}) and (\ref{2-deg}) hold.  (In the above argument we used nowhere our assumption that $F\neq 0$.
The generalized Einstein condition on $E_{-}$ is equivalent to relations (\ref{1-deg}) and (\ref{2-deg}) also when $F =0$, a fact which will be used in the proof of Proposition \ref{add-correct} below). 
Remark  that 
\begin{equation}\label{delta-23}
\delta_{2} +\delta_{3} =0,
\end{equation}
which follows from the 1st and 3rd relations (\ref{2-deg}), combined with $F\neq 0$ and $F_{13} = - F_{12}.$  
Using relation  (\ref{delta-23}), the 2nd relation (\ref{1-deg})  and the last relation (\ref{2-deg}), we can replace the 6th relation (\ref{2-deg}) with 
\begin{equation}
h (\lambda +\rho )  = \epsilon_{1} ( \lambda \delta_{4} + h \delta_{5}). 
\end{equation}
Similarly, comparing  the 7th relation (\ref{1-deg}) with  the 7th relation (\ref{2-deg}) and using 
the 8th relation (\ref{2-deg})  together with relation (\ref{delta-23}),  we can replace 
the 7th relation (\ref{2-deg}) with
\begin{equation}
( h -\nu )\delta_{4} =0.
\end{equation}
We obtain that relations (\ref{1-deg}) and (\ref{2-deg}) are equivalent to relation (\ref{delta-23}) together with 
\begin{align}
\nonumber& ( h-\nu )(h +\nu ) =0\\
\nonumber& h ( \lambda +\rho  -\epsilon_{1} \delta_{2})  + \epsilon_{1}\lambda \delta_{1}  = -2\delta_{0} F_{12}\\
\nonumber& \epsilon_{1} \nu  (\lambda -\rho ) + \lambda \delta_{1} + \nu \delta_{2} = 2F_{12} F_{23}\\
\nonumber&  \epsilon_{1} (\lambda^{2} +\rho^{2} +\mu \nu ) - \mu \delta_{1} - \rho \delta_{2} =- (F_{12})^{2}\\
\nonumber& \epsilon_{1} ( h+\nu) \delta_{1}  = 2\delta_{0} F_{23}\\
\nonumber&   h^{2} +\nu^{2}  = - 2\epsilon_{1} (F_{23})^{2}\\
\nonumber& F_{23} ( h -\nu ) =0\\
\nonumber& F_{12} ( h -\epsilon_{1} \delta_{1}) = F_{23} (\rho - \epsilon_{1} \delta_{2})\\
\nonumber&  \delta_{6} = -\delta_{5}\\
\nonumber& F_{12} ( \delta_{1} -\delta_{4}) = F_{23} ( \delta_{2} -\delta_{5})\\
\nonumber& \lambda \delta_{4} + h \delta_{5} = \epsilon_{1} h (\lambda +\rho )\\
\nonumber&  ( h-\nu )\delta_{4} =0\\
\nonumber&  \mu (\delta_{1} -\delta_{4}) +\rho (\delta_{2} -\delta_{5}) =0\\
\label{delta-23-last}& \lambda  (\delta_{1} -\delta_{4}) + \nu (\delta_{2} -\delta_{5}) =0.
\end{align}
We note for later purpose that the derivation of \eqref{delta-23-last} did not use that $F\neq 0$ but only $\delta_2+\delta_3=0$.
 (In the proof of Proposition 
\ref{add-correct} we will  actually show that $\delta_{2} +\delta_{3} =0$ also when $F=0$.)  
The 6th  and 7th relations (\ref{delta-23-last}) imply that 
$h = \nu.$ Straightforward computations lead to the first class of solutions listed in the theorem when $F_{23} =0$ and to the second class of solutions when $F_{23} \neq 0$. 
\end{proof}

The next proposition follows   from the classification of $3$-dimensional non-unimodular Lie algebras recalled  at the beginning 
of Section \ref{non-unimod-section}. The parameter $\lambda$ in Proposition \ref{identif-lie} 
refers to the notation explained after equation \eqref{canonical-forms-A}  and is not meant to be 
the parameter $\lambda$ in Theorem \ref{deg-thm}.

\begin{prop}\label{identif-lie}   i) There is an odd   generalized Einstein metric 
from the first family stated in Theorem \ref{deg-thm}  if and only if $\mathfrak{g}$ is isomorphic to 
$\tau_{2}(\mathbb{R}) \oplus \mathbb{R}$, $\tau_{3} (\mathbb{R})$ or $\tau_{3, \lambda }(\mathbb{R}).$ It can be chosen to be divergence-free or with non-zero divergence.\ 

ii)  There is an odd   generalized Einstein metric 
from the second  family stated in Theorem \ref{deg-thm}  if and only if $\mathfrak{g}$ is isomorphic to $\tau_{2}(\mathbb{R}) \oplus \mathbb{R}$,  
$\tau_{3} (\mathbb{R})$,  $\tau_{3, \lambda }(\mathbb{R})$  with $\lambda\neq 1$,   or $\tau_{3, \lambda}^{\prime} (\mathbb{R}).$ All such 
generalized  Einstein metrics have non-zero divergence. 
\end{prop}

\begin{proof} Let us explain the argument for the second family.  As $\lambda +\rho \neq 0$ we immediately obtain
from (\ref{cond-ii})  that  any generalized metric $E_{-}$ from the second family has non-zero divergence. Let $\mathbf{A}$ be the matrix defined by
(\ref{form-A}),  $\mathbf{A}^{\prime}:= - \epsilon_{1} \mathbf{A}$ and 
$$
P_{\mathbf{A}^{\prime}} (x) = x^{2} - (\lambda +\rho ) x +\lambda \rho - \mu \nu 
$$
the characteristic polynomial of $\mathbf{A}^{\prime}.$ 
When  $\lambda \rho =\mu \nu$,  $P_{\mathbf{A}}^{\prime}$ has two real roots, one zero and the other non-zero. We obtain that
$\mathfrak{g}$ is isomorphic to $\tau_{2} (\mathbb{R}) \oplus \mathbb{R}.$  When $\lambda \rho \neq \mu \nu$,  the obstructions on the constants 
$\lambda$, $\nu$, $\mu$, $\rho$ are 
$$
\nu \neq 0,\ \lambda +\rho \neq 0,\ \lambda\rho \neq \mu \nu .
$$ Letting 
$\mu := \frac{1}{4\nu} ( \Delta -   (\lambda -\rho )^{2})$, we see that any $\Delta \in \mathbb{R}\setminus \{ (\lambda +\rho )^{2} \}$ 
can be realized as the discriminant of $P_{\mathbf{A}^{\prime}}.$ When $\Delta\in \mathbb{R}^{+} \setminus \{ 
(\lambda +\rho )^{2} \}$, the roots $x_{\pm} = \frac{1}{2} ( \lambda +\rho
\pm \sqrt{\Delta})$ 
of $P_{\mathbf{A}}^{\prime}$ 
are distinct, real, and non-zero. The corresponding Lie algebras $\mathfrak{g}$ cover  the  entire family   $\tau_{3, \lambda}(\mathbb{R})$ 
except $\tau_{3, 1}(\mathbb{R})$  (as $x_{+}\neq x_{-}$).  When $\Delta <0$, $x_{\pm}$ are non-real and  the corresponding Lie algebras $\mathfrak{g}$ cover 
 the family 
 $\tau_{3, \lambda}^{\prime} (\mathbb{R}).$ 
Finally, when $\Delta =0$, $x_{+} = x_{-} \in \mathbb{R}\setminus  \{  0\}$. As  $\mathbf{A}^{\prime}$ is non-diagonal (since  $\nu \neq 0$), we  obtain that
$\mathfrak{g} $ is isomorphic to $\tau_{3} (\mathbb{R}).$ 
\end{proof}

We end this section by describing the generalized Einstein metrics $ E_{-} = \{ X - g(X)\mid X\in TG\}$  on $E_{H}:=
E_{H, 0}$  with 
$g\vert_{\mathfrak{u}\times \mathfrak{u}}$-degenerate. This is done in the next proposition, which   corrects Proposition 3.16  of \cite{c-k-paper}.

\begin{prop}\label{add-correct} Consider the setting from the beginning of this section and  assume that $F=0.$ 

i) The generalized metric $E_{-} =\{ X - g(X)\mid X\in TG\}$
is generalized Einstein with divergence $\delta$ if and only if 
$h = \nu =0$,   $\lambda +\rho \neq 0$, $\delta_{3} = -\delta_{2}$, $\delta_{6} = -\delta_{5}$,  $\lambda \delta_{1} =\lambda \delta_{4} =0$ and
\begin{equation}\label{mu-delta-i}
\mu \delta_{1} +\rho \delta_{2} = \mu \delta_{4} +\rho \delta_{5} = \epsilon_{1} ( \lambda^{2} +\rho^{2}).
\end{equation}
ii) A generalized metric like in i) exists  if and only if $\mathfrak{g}$ is isomorphic to $\tau_{2}(\mathbb{R}) \oplus \mathbb{R}$, 
$\tau_{3} (\mathbb{R})$ or $\tau_{3, \lambda}(\mathbb{R}).$ Moreover, it has always  non-zero divergence.
\end{prop}

\begin{proof}   i) As stated in the proof of   Theorem \ref{deg-thm},   $E_{-}$ is generalized Einstein with divergence  $\delta$ if and only if 
 relations (\ref{1-deg}) and (\ref{2-deg}) hold with $F =0.$  We now prove that
 \begin{equation}\label{delta-deg-0}
 \delta_{2} +\delta_{3} =0.
 \end{equation}
If, by contradiction, $\delta_{2} +\delta_{3} \neq 0$, then the 3rd and 6th relations (\ref{1-deg}) imply that $h =\nu =0$. 
But then the   1st and 7th  relations 
(\ref{1-deg})  imply  that $\lambda = \rho =0$, which contradicts $\lambda +\rho \neq 0.$  Relation (\ref{delta-deg-0}) follows. We obtain that  $E_{-}$ is generalized Einstein with divergence $\delta$ if and only if  relations (\ref{delta-deg-0}) and 
(\ref{delta-23-last})  hold (the latter with $F =0$). We arrive at the family of solutions listed in the proposition. 

ii) Relations (\ref{mu-delta-i}) imply that $\delta \neq 0$ (otherwise $\lambda +\rho =0$). 
The identification of the Lie algebra $\mathfrak{g}$ follows from the classification of
$3$-dimensional non-unimodular Lie algebras recalled at the beginning of 
Section \ref{non-unimod-section}.
\end{proof}

\subsection{The case $g\vert_{\mathfrak{u}\times \mathfrak{u}}$ non-degenerate}\label{nondeg-sect}

Consider the setting from the beginning of Section \ref{non-unimod-section}
and assume that  $g\vert_{\mathfrak{u}\times \mathfrak{u}}$ is non-degenerate. Let 
$(v_{a})$ be a $g$-orthonormal  basis of $\mathfrak{g}$ such that   $v_{1}, v_{3} \in \mathfrak{u}$. 
The Lie bracket of $\mathfrak{g}$  is given by
\begin{align}
\nonumber& \mathcal L_{v_{1}} v_{3}  =0\\
\nonumber& \mathcal L_{v_{2}} v_{1}  = \epsilon_{1} \lambda v_{1} + \epsilon_{3} \mu v_{3}\\
\label{lie-non-deg}& \mathcal L_{v_{2}} v_{3}   = \epsilon_{1} \nu v_{1} + \epsilon_{3} \rho v_{3},
\end{align}
where $\epsilon_{i} := g (v_{i}, v_{i})$,  $\lambda , \mu , \nu , \rho \in \mathbb{R}$ and  
$\mathrm{tr}\, \mathrm{ad}_{v_{2}} = 
\epsilon_{1} \lambda + \epsilon_{3}
\rho \neq 0.$ 
The Lie algebra $\mathfrak{g}$ is the semi-direct product $\mathbb{R}\ltimes_{\mathbf{A}} \mathbb{R}^{2}$ where 
\begin{equation}\label{form-A-deg}
\mathbf{A} = \left( \begin{tabular}{cc}
$\epsilon_{1} \lambda$ & $ \epsilon_{1} \nu$\\
$\epsilon_{3}\mu$ & $\epsilon_{3} \rho$
\end{tabular}\right) .
\end{equation}
From  (\ref{lie-non-deg}) 
we obtain that a  left-invariant  $2$-form $F\in \Omega^{2}(G)$  is closed  if and only if 
\begin{equation}\label{closed-non-deg}
F_{13}=0.
\end{equation}
In particular, the twisting $2$-form $F$ of  $E_{H, F}$ satisfies (\ref{closed-non-deg}).

Case 7) of  Theorem \ref{non-deg-thm} below  will be treated in the next section.

\begin{thm}\label{non-deg-thm}  
There are seven classes of odd  generalized Einstein metrics 
$E_{-} = \{ X - g(X)\mid  X\in TG\}$  
with divergence $\delta$ defined on $E_{H, F}$. In 
the basis $(v_{a})$ above (with $v_{1}$ and $v_{3}$ interchanged, if necessary), 
they are described as follows  (below $\epsilon , \tilde{\epsilon} \in \{ \pm 1\}$):

1) $\lambda\neq 0$,  $ h = \mu$,
$F_{23} =F_{13} =0$,  $\delta_{0} F_{12} = \epsilon_{3} \mu\delta_{3} $,
$\delta_{1} = \delta_{4} =0$, $\delta_{2} =\delta_{5} = -\epsilon_{1} \lambda$
and one of the following conditions hold:

a) $\rho = \nu =0$ and $(F_{12})^{2} = - (\epsilon_{1} \lambda^{2} +\epsilon_{3} \mu^{2}) >0$;\

b) $\mu =0$, $\rho = - 2\lambda$, $\nu = 2\epsilon \lambda$, $\epsilon_{1} =-1$, $\epsilon_{3} =1$, 
$F_{12}= \tilde{\epsilon}\lambda\sqrt{3}$,  
$\delta_0=\delta_{3} =\delta_{6} =0$;\

2) $\lambda =0$,  $\mu \neq 0$, $\nu = - 2\mu$, $\rho =\tilde{\epsilon}\mu $, 
$\epsilon_{1} =-1$, 
$\epsilon_{3} =1$,  $h = \mu$,
 $F_{13} =0$, $F_{23} =\epsilon\rho$, $F_{12} = \epsilon\mu$,
$\delta_{0} =\delta_{1} = \delta_{3} =\delta_{4} =\delta_{6} =0$, $\delta_{2} = \delta_{5} = -\rho$;\

3) $\lambda \nu (\epsilon_{1}\lambda +\epsilon_{3}\rho )  \neq 0$, 
$\epsilon_{1} \lambda^{2} +\epsilon_{3} \mu^{2}<0$, $\lambda \rho = \mu \nu$, 
$ h  = \mu - \nu$, 
$F_{23} =- \frac{\nu}{\lambda} F_{12}$, $(F_{12})^{2} = - (\epsilon_{1} \lambda^{2} +\epsilon_{3} \mu^{2})$, $F_{13} =0$, $\delta_{0} F_{12} = \epsilon_{1} \frac{\lambda}{\nu} \delta_{1} ( \nu - \mu )$, $\delta_{2} = \delta_{5} = - ( \epsilon_{1} \lambda + \epsilon_{3} \rho )$, 
$\delta_{3} = -\epsilon_{1} \epsilon_{3} \frac{\lambda}{\nu } \delta_{1}$,
$\delta_{6} = -\epsilon_{1}\epsilon_{3} \frac{\lambda}{\nu} \delta_{4}$, 
and one of the following conditions holds:\

a) $\delta_{1} =\delta_{4}\neq 0$;\

b) $\delta_{0} = \delta_{1} =\delta_{4}=0$ and either   $(\rho = \mu =0, \epsilon_{1} =-1)$
or $(\mu - \nu ) \rho \neq 0$;\

c) $\delta_{1} \neq \delta_{4}$ and  either  $\delta_{1} \neq 0$ or  $\delta_{0} =\delta_{1} =0$;\

4) $\lambda = \nu =0$, $\rho \mu \neq 0$, 
$\epsilon_{3} =-1$,
$ h = \mu$,  $F_{12} = \epsilon \mu$, 
$F_{23} = -\epsilon \rho$, 
$F_{13} =0$, $\delta_{2} = 
\delta_{5} = \rho$,
$\delta_{3} = -\epsilon \delta_{0}$,  
$\delta_{3} =  \epsilon_{1}\frac{ \mu }{\rho  }  \delta_{1}$, $\delta_{6} =  \epsilon_{1} \frac{\mu}{ \rho } \delta_{4}$ 
and one of the following conditions hold:\

a)  either 
$\delta_{0} = 0$ or $\delta_{0} (\delta_{1} -\delta_{4}) \neq 0$;\

b) $\delta_{4} = \delta_{1}$ and $\delta_{1} \delta_{3}\neq 0$;\

5) $\lambda = \frac{\epsilon}{3} (2\mu + \nu )$, $\rho = \frac{\epsilon}{3} ( 2\nu +\mu )$, 
$\epsilon_{1} ( \mu^{2}-  \nu^{2})>0$, $\mu + 2\nu \neq 0$, $ \nu + 2\mu \neq 0$, $\epsilon_{3} = -\epsilon_{1}$, 
$ h = \frac{1}{3} ( \mu - \nu )$, 
$F_{13} =0$, 
$F_{23} = -\epsilon F_{12}$, $(F_{12})^{2} = \frac{\epsilon_{1} }{3} ( \mu^{2} - \nu^{2})$, $\delta_{0} = \delta_{1}
=\delta_{3} =\delta_{4} =\delta_{6} =0$, $\delta_{2} =\delta_{5} =\frac{\epsilon \epsilon_{1}}{3} (\nu - \mu )$;\

6) $ \rho = -2\lambda \neq 0$, $\mu = 2\epsilon \lambda$,  $\nu =0$,  
$\epsilon_{1} =1$, $\epsilon_{3} =-1$, $h=0$, $F_{12} = \tilde{\epsilon} \lambda$, 
$F_{23} = 2 \epsilon \tilde{\epsilon} \lambda$, $F_{13}=0$, $\delta_{0} = \delta_{1} =\delta_{3}=\delta_{4} =\delta_{6}=0$,
$\delta_{2} =\delta_{5} = -\lambda$;\

7)  all constants 
 $\mu$, $\nu$, 
$F_{12}$, $F_{23}$ and $\epsilon_{1}\lambda +\epsilon_{3}\rho$
are non-zero, 
\begin{equation}
F_{13} = h =\delta_{0}  = \delta_{1} =\delta_{3}= \delta_{4} =\delta_{6}=0,\quad  \delta_2=\delta_5
\end{equation}
and relations (\ref{syst1})  with $h=0$ hold. 
\end{thm}

We  divide the proof of the above theorem  into several lemmas. 
To keep the text reasonably short, we only present the main steps of the computations. The remaining parts of the argument follow  easily.

\begin{lem} The metric $E_{-}$ is generalized Einstein with divergence $\delta$ if and only if
relations
\begin{align}
\nonumber&   \frac{ \epsilon_{3}}{2} ( \nu^{2} - h^{2} - \mu^{2}) + \delta_{2} \lambda  =  (F_{12})^{2}\\
\nonumber& \frac{\epsilon_{3} \delta_{3}}{2} ( h +\mu -\nu ) =\delta_{0} F_{12} \\
\nonumber&  \frac{\epsilon_{1}}{2} \lambda ( h - \mu  + \nu )+ \frac{\epsilon_{3}}{2} \rho ( h +\mu - \nu )  
+ \frac{\delta_{2}}{2} ( h - \mu - \nu ) =  F_{12}  F_{23}\\
\nonumber& \epsilon_{1} \delta_{1} \lambda +\frac{\epsilon_{3}\delta_{3}}{2} ( h +\mu +\nu ) = \delta_{0} F_{12} \\
\nonumber& -\lambda^{2} -\frac{\epsilon_{1} \epsilon_{3} }{2 } h^{2} -\frac{\epsilon_{1} \epsilon_{3}}{2} ( \mu +\nu )^{2} -\rho^{2} =
\epsilon_{1} ( F_{12})^{2} +\epsilon_{3} ( F_{23})^{2}\\
\nonumber& \frac{\epsilon_{1} \delta_{1}}{2} ( h -\mu - \nu ) -\epsilon_{3} \delta_{3} \rho = \delta_{0} F_{23} \\
\nonumber& \frac{\epsilon_{1} }{2} \lambda ( h +\mu - \nu ) +\frac{\epsilon_{3}}{2} \rho ( h -\mu +\nu ) 
+\frac{ \delta_{2}}{2} ( h+\mu +\nu ) = -  F_{12} F_{23}\\
\nonumber& \frac{\epsilon_{1}\delta_{1}}{2} ( h +\mu - \nu) = \delta_{0} F_{23} \\
\label{3-nondeg}& \frac{\epsilon_{1}}{2} ( \mu^{2} - h^{2} - \nu^{2}) + \delta_{2} \rho = (F_{23})^{2}
\end{align}
together with 
\begin{equation}\label{4-nondeg}
\delta_{2} = \delta_{5}
\end{equation}
and
\begin{align}
\nonumber& F_{12} \epsilon_{3} (\delta_{2} +\epsilon_{1} \lambda )  = F_{23} ( h+\nu ) \\
\nonumber& F_{12} \epsilon_{3} \delta_{1} =  F_{23} \epsilon_{1} \delta_{3} \\
\nonumber& F_{12} ( \mu - h)  =  F_{23} \epsilon_{1} (\epsilon_{3} \rho +\delta_{2} ) \\
\nonumber& F_{12} \epsilon_{1} ( \delta_{1} -\delta_{4}) =   F_{23} \epsilon_{3} ( \delta_{3} -\delta_{6})\\
\nonumber& \frac{\epsilon_{3}}{2} ( h +\mu -\nu ) \delta_{3} +\epsilon_{1}\lambda \delta_{4} +\frac{\epsilon_{3}}{2} ( - h +\mu +\nu )\delta_{6}
=\delta_{0} F_{12}\\
\nonumber& \frac{\epsilon_{1}}{2} ( h -\mu -\nu ) \delta_{1} -\epsilon_{3} \rho \delta_{3} -\frac{\epsilon_{1}}{2} ( h -\mu +\nu ) \delta_{4} =\delta_{0}
F_{23}\\
\nonumber& \epsilon_{1} \lambda (\delta_{1} -\delta_{4}) +\epsilon_{3} \nu (\delta_{3} -\delta_{6}) =0\\
\label{5-nondeg} & \epsilon_{1} \mu (\delta_{1} -\delta_{4}) +\epsilon_{3} \rho (\delta_{3} -\delta_{6}) =0
\end{align}
hold.  
\end{lem}

\begin{proof}  From Remark \ref{compare}, we obtain the Dorfman coefficients $B_{ABC}$ ($1\leq A, B, C\leq 6$) 
in the basis $(e_{A})_{0\leq A\leq 6}$ of $\tilde{E} = \mathfrak{g}\oplus \mathfrak{g}^{*}\oplus \mathbb{R}$  determined by the basis
$(v_{a})$  of $\mathfrak{g}$ by replacing $h\mapsto - h$ in the Dorfman coefficients computed in the proof of Proposition 3.5 of \cite{c-k-paper}:
\begin{align}
\nonumber&  B_{145} = -\lambda ,\ B_{156} =\frac{1}{2} ( - h +\mu + \nu),\ B_{246} =\frac{1}{2} ( h - \mu +\nu )\\
\nonumber& B_{345} =-\frac{1}{2} ( h +\mu +\nu ),\ B_{356} = \rho ,\ B_{412} =\lambda \\
\nonumber& B_{423} =-\frac{1}{2} ( h+\mu +\nu ),\ B_{513} = \frac{1}{2} ( h+\mu -\nu )\\ 
\nonumber&  B_{612} =\frac{1}{2} ( - h +\mu +\nu ),\ B_{623} = -\rho\\
\label{1-nondeg} & B_{146} = B_{245} = B_{256} = B_{346} = B_{413} = B_{512} =  B_{523} =B_{613} =0.
\end{align}
The components 
$R_{ia}^{\tilde{\delta}, +}$ 
of the Ricci tensor $\mathrm{Ric}^{\tilde{\delta}, +}$  
from Corollary \ref{forapplic}  are computed from Lemma \ref{ricci-diferit}, 
by using relations  (\ref{sus-jos}) and  (\ref{1-nondeg}):
\begin{align}
\nonumber& R_{41}^{\tilde{\delta}, +} = \frac{\epsilon_{2} \epsilon_{3}}{2} (\nu^{2} - h^{2} -\mu^{2}) +\epsilon_{2} \delta_{2} \lambda \\
\nonumber& R_{42}^{\tilde{\delta}, +} = - \epsilon_{1} \delta_{1} \lambda -\frac{\epsilon_{3} \delta_{3}}{2} ( h +\mu + \nu )\\
\nonumber& R_{43}^{\tilde{\delta}, +} = \frac{\epsilon_{1} \epsilon_{2}}{2} \lambda ( h +\mu - \nu ) +\frac{\epsilon_{2} \epsilon_{3}}{2} \rho
(h -\mu +\nu)+\frac{\epsilon_{2}\delta_{2}}{2} ( h +\mu +\nu )\\
\nonumber& R_{51}^{\tilde{\delta}, +} = \frac{\epsilon_{3} \delta_{3}}{2} ( h +\mu - \nu )\\
\nonumber& R_{52}^{\tilde{\delta}, + }= -\lambda^{2} - \rho^{2} -\frac{\epsilon_{1}\epsilon_{3}}{2} ( ( \mu +\nu )^{2} + h^{2} )\\
\nonumber & R_{53}^{\tilde{\delta}, +} =-  \frac{\epsilon_{1} \delta_{1}}{2} ( h + \mu -\nu )\\
\nonumber& R_{61}^{\tilde{\delta}, + }= - \frac{\epsilon_{1} \epsilon_{2}}{2} \lambda ( h  - \mu + \nu )  - \frac{\epsilon_{2} \epsilon_{3}}{2} 
\rho ( h +  \mu - \nu )+\frac{\epsilon_{2} \delta_{2}}{2} ( - h +\mu +\nu )\\
\nonumber& R_{62}^{\tilde{\delta}, +} = \frac{\epsilon_{1} \delta_{1}}{2} ( h -\mu -\nu ) -\epsilon_{3} \delta_{3}\rho\\
\label{2-nondeg}& R_{63}^{\tilde{\delta}, +} = \frac{\epsilon_{1} \epsilon_{2}}{2} ( \mu^{2} -h^{2} -\nu^{2}) +\epsilon_{2} \delta_{2} \rho
\end{align}
(when $\epsilon_{1} = \epsilon_{2} = -\epsilon_{3}$, they coincide with the Ricci coefficients from the proof of Proposition 3.15 of
\cite{c-k-paper}, with $h$ replaced by $-h$).  
Our claim follows from Corollary \ref{forapplic}.  More precisely, relations (\ref{3-nondeg}) are equivalent to
relations (\ref{rel-1}), while  relations (\ref{4-nondeg}) and (\ref{5-nondeg}) are equivalent to relations (\ref{rel-2}), (\ref{rel-3}) and (\ref{rel-4}).
\end{proof}

\begin{rem}{\rm  i)   Combining  the 2nd with the 4th relation (\ref{3-nondeg}), the 3rd with the 7th relation (\ref{3-nondeg}) and the 6th with the 8th relation (\ref{3-nondeg}),  we will often replace  in the arguments below the 4th, 6th and 7th relations (\ref{3-nondeg}) with 
\begin{align}
\nonumber& \epsilon_{1} \lambda \delta_{1} +\epsilon_{3} \nu \delta_{3} =0\\
\nonumber& \epsilon_{1} \mu \delta_{1} +\epsilon_{3} \rho \delta_{3} =0\\
\label{add} & h ( \delta_{2} + \epsilon_{1} \lambda + \epsilon_{3} \rho ) =0.
\end{align}
Similarly, using the 1st and last relations (\ref{3-nondeg}), we will often  replace the 5th relation (\ref{3-nondeg}) by
\begin{equation}\label{add-after}
\delta_{2} (\epsilon_{1} \lambda +\epsilon_{3} \rho ) = -\lambda^{2} -\frac{\epsilon_{1}\epsilon_{3}}{2} (\mu +\nu )^{2} -\rho^{2}
+\frac{\epsilon_{1}\epsilon_{3}}{2} h^{2}.
\end{equation}
ii)  Moreover,  if  $(\delta_{1} -\delta_{4})(\delta_{3} -\delta_{6})\neq 0$ or $\delta_{1}\delta_{3}\neq 0$, then
\begin{equation}\label{lambda-rho}
\lambda \rho = \mu \nu .
\end{equation}
This follows from the last two relations (\ref{5-nondeg}) and the first two relations (\ref{add}).
}
\end{rem}

In order to determine the generalized Einstein metrics $E_{-}$ we need to solve the system formed by the equations 
(\ref{3-nondeg}), (\ref{4-nondeg})  and (\ref{5-nondeg}). For this, we will consider various cases, namely:  I) $F_{23} =0$; II) $F_{23} \neq 0$ and
$F_{12} =  0$; III) $F_{23} F_{12}\neq 0.$  Recall that $F_{13} =0$, as  the $2$-form $F$ is closed.

\begin{lem} When $F_{23} =0$, the system (\ref{3-nondeg}), (\ref{4-nondeg}) and (\ref{5-nondeg}) leads to the classes  of solutions
1a) and 1b).
\end{lem}

\begin{proof} Since  $F_{23} =0$, we obtain that  $F_{12}\neq 0$ (otherwise $F=0$). 
From the first four relations  (\ref{5-nondeg}) and the 6th relation (\ref{5-nondeg}), 
we obtain
\begin{equation}
\delta_{2} = -\epsilon_{1}\lambda ,\ \delta_{1} = \delta_{4} =0,\  h =\mu ,\ \rho \delta_{3} =0.
\end{equation}
When $\rho =0$ we obtain the class of solutions 1a).  When $\rho \neq 0$ we obtain the class of solutions 1b). 
\end{proof}

\begin{lem} When $F_{23} \neq 0$ and $F_{12} =0$, 
the system (\ref{3-nondeg}), (\ref{4-nondeg}) and (\ref{5-nondeg}) leads (after interchanging $v_{1}$ and  $v_{3}$),  to the classes of solutions
1a) and 1b).
\end{lem}

\begin{proof}
The 1st and 3rd relations (\ref{5-nondeg}) imply  that $ h = -\nu $  and 
$\delta_{2} = -\epsilon_{3} \rho .$ 
From the 3rd relation (\ref{3-nondeg}) we obtain 
$\mu ( \lambda - 2 \epsilon_{1} \epsilon_{3} \rho ) =0$. When $\mu =0$ we obtain (after interchanging  $v_{1}$ and  $v_{3}$)  the class of solutions  
1a).
(Note that under the interchange of $v_1$ and $v_3$ the variable $\lambda$ is interchanged with $\rho$, $\mu$ with $\nu$,  
$\delta_{1}$ with $\delta_{3}$, $\delta_{4}$ with $\delta_{6}$,   $\epsilon_{1}$ with $\epsilon_{3}$ and 
$\delta_{0}$, $\delta_{2}$ and $\delta_{5}$ remain the same while $h$ is replaced by $-h$).
 When $\mu \neq 0$ we obtain   (after interchanging   $v_{1}$ and  $v_{3}$)  the class of solutions 1b). 
\end{proof}

The key remark to treat the case 
$F_{12} F_{23} \neq 0$  is the observation that 
$(\delta_{1} -\delta_{4})$ and $(\delta_{3} -\delta_{6})$ are both zero or both non-zero, and, similarly, for 
$\delta_{1}$ and $\delta_{3}$ (see the  2nd and 4th relation (\ref{5-nondeg})).  Also, from the 3rd relation (\ref{add}), 
either $h =0$ or $\delta_{2} = - (\epsilon_{1} \lambda +\epsilon_{3} \rho ) .$ 
Using these facts  we divide the case 
$F_{12} F_{23} \neq 0$ 
into several lemmas, as follows.

\begin{lem} When $F_{12} F_{23}\neq 0$, $\delta_{1} =\delta_{4}$, $\delta_{3} =\delta_{6}$  and $\delta_{1} =\delta_{3} = h =0$, 
the system (\ref{3-nondeg}), (\ref{4-nondeg}) and (\ref{5-nondeg}) leads  (modulo interchanging $v_{1}$ and  $v_{3}$),
to the classes of solutions 6) and 7). 
 \end{lem}

\begin{proof}
By our assumption, $\delta_{i} =0$ for any $i\notin\{ 0, 2, 5\}$. 
The 2nd relation \eqref{3-nondeg} implies that also $\delta_0=0$. 
When $\delta_{i} =0$ for any $i\notin\{0,  2, 5\}$, relations
(\ref{3-nondeg}) (with the 4th,  5th, 6th and 7th relation (\ref{3-nondeg}) replaced by relations (\ref{add})
and (\ref{add-after})),  together with relations 
(\ref{4-nondeg}) and (\ref{5-nondeg})  become
\begin{equation}\label{syst} 
\delta_{0} =0,\ \delta_2=\delta_5,\ h ( \delta_{2} +\epsilon_{1}\lambda +\epsilon_{3} \rho ) =0
\end{equation}
together with 
\begin{align}
\nonumber&  F_{12} \epsilon_{3} (\delta_{2} +\epsilon_{1}\lambda ) = F_{23} ( h+  \nu )\\
\nonumber& F_{12} (\mu  - h) = F_{23} \epsilon_{1} (\epsilon_{3} \rho +\delta_{2})\\
\nonumber& (F_{12})^{2} =\frac{\epsilon_{3}}{2} ( \nu^{2} -\mu^{2} -  h^{2} ) +\delta_{2} \lambda\\
\nonumber& \delta_{2} (\epsilon_{1} \lambda +\epsilon_{3} \rho ) = -\lambda^{2} -\frac{\epsilon_{1} \epsilon_{3}}{2} ( \mu +\nu )^{2} -\rho^{2}
+\frac{\epsilon_{1}\epsilon_{3}}{2} h^{2}\\
\nonumber&  F_{12}  F_{23}=  \frac{\epsilon_{1}}{2} \lambda ( h - \mu  + \nu )+ \frac{\epsilon_{3}}{2} \rho ( h +\mu - \nu )  
+ \frac{\delta_{2}}{2} ( h - \mu - \nu ) \\
\label{syst1}&  (F_{23})^{2} = \frac{\epsilon_{1}}{2} (\mu^{2} - \nu^{2} - h^{2}) +\delta_{2} \rho .
\end{align}
(Relations (\ref{syst}) and  (\ref{syst1})  with $h\neq 0$ will be used in the proof of the next lemma). Recall now our assumption $h =0$. 
Since $F_{23} F_{12}\neq 0$ and $h=0$,  $\mu$ and $\nu$ cannot be both zero (from the 5th relation (\ref{syst1})). 
When 
$\mu  \neq 0$ and $\nu =0$ we obtain the class of solutions 6).  In fact, when $\nu=0$ the 1st equation of \eqref{syst1} with $h=0$
yields $\delta_2=-\epsilon_1\lambda$, which can be used 
to rewrite the 4th and 6th equations as 
\begin{eqnarray*}-\epsilon_3\rho^2 &=& \epsilon_1(\frac{\mu^2}{2} -\lambda \rho)\\
(F_{23})^2&=& \epsilon_1(\frac{\mu^2}{2} -\lambda \rho).
\end{eqnarray*}
This shows that $(F_{23})^2 = -\epsilon_3\rho^2$ and the remaining equations are easily solved building on that.
When  $\mu =0$ and $\nu \neq 0$ we obtain again (after interchanging  $v_{1}$ and   $v_{3}$) the class of solutions 6). 
When $\mu \nu \neq 0$ the  solutions of the system (\ref{syst1}) with $h=0$ 
correspond  to the class of solutions 7)  from the statement of the theorem.
\end{proof}

\begin{lem}
When $F_{12} F_{23}\neq 0$, $\delta_{1} =\delta_{4}$, $\delta_{3} =\delta_{6}$, $\delta_{1} =\delta_{3}  =0$
and $h\neq 0$, 
the system (\ref{3-nondeg}), (\ref{4-nondeg}) and (\ref{5-nondeg}) leads  (modulo interchanging $v_{1}$ and  $v_{3}$) to the classes of solutions
4a) with $\delta_{0} = \delta_{4} =0$,  2), 3b)   and 5). 
\end{lem}

\begin{proof}
From the previous proof, what we need to solve  is the system (\ref{syst}) and (\ref{syst1}) under the assumption $h\neq 0.$
From  the 3rd relation (\ref{syst})  
and the 1st relation (\ref{syst1}), 
we obtain 
\begin{equation}\label{delta-2-add}
\delta_{2} = -  (\epsilon_{1} \lambda + \epsilon_{3} \rho ),\quad 
F_{12} \rho = - F_{23} ( h +\nu ).
\end{equation}
Since 
$F_{23} F_{12}  \neq 0$, either $ \rho = h +\nu =0$ or $\rho ( h+\nu ) \neq 0.$ When $ \rho = h +\nu =0$
and $\mu =0$,  we obtain the class of solutions 3b) with $(\rho = \mu =0, \epsilon_{1} = -1)$.
When $\rho = h+\nu =0$ and $\mu \neq 0$, we obtain  (after interchanging  $v_{1}$ and  $v_{3}$) 
the class of solutions  2).
Assume now that  $\rho ( h+\nu ) \neq 0$.
Since $\delta_{2} = -(\epsilon_{1} \lambda +\epsilon_{3} \rho )$, the 3rd and 6th relations (\ref{syst1}) imply that
\begin{align}
\nonumber& (F_{12})^{2} = \frac{\epsilon_{3}}{2} (  \nu^{2}- \mu^{2} - h^{2} ) - \epsilon_{1} \lambda^{2} -\epsilon_{3} \lambda \rho\\
\label{patrat-F}& (F_{23})^{2} = \frac{\epsilon_{1}}{2} ( \mu^{2} - \nu^{2} - h^{2}) -\epsilon_{1} \lambda \rho -\epsilon_{3} \rho^{2}
\end{align} 
and the 4th relation (\ref{syst1}) becomes 
\begin{equation}\label{adaus}
4\lambda \rho =  ( \mu + \nu)^{2} -  h^{2}.
\end{equation}
Expressing $\lambda \rho $ in terms of $\mu$, $\nu$ and $h$ (using the above relation), we rewrite relations
(\ref{patrat-F}) in an equivalent way: 
\begin{align}
\nonumber& (F_{12})^{2}  = \frac{\epsilon_{3}}{4} ( - h^{2} - 3\mu^{2} +\nu^{2} - 2\mu \nu ) -\epsilon_{1} \lambda^{2}\\
\label{equivalent}& (F_{23})^{2}  = \frac{\epsilon_{1}}{4} ( - h^{2} - 3\nu^{2} + \mu^{2} - 2\mu \nu ) - \epsilon_{3} \rho^{2}.
\end{align} 
When   $\lambda =0$,  $\delta_{2} = -\epsilon_{3} \rho$ and 
the 2nd  relation (\ref{syst1}) implies   that $ h = \mu$. 
Relation
(\ref{adaus}) becomes  $\nu ( 2\mu  +\nu ) =0$. 
We   arrive at the  class of solutions 4a) 
with $\delta_{0} = \delta_{4} =0$  (when $\nu =0$) and at the  class of  solutions~2) (when $\nu = - 2\mu\neq 0$). 
Assume now that $\lambda \neq 0.$
From the 2nd  relation (\ref{syst1}), 
and $\delta_{2} = -(\epsilon_{1} \lambda +\epsilon_{3} \rho )$, 
we obtain
\begin{equation}\label{F-23-12}
F_{23} = F_{12} ( \frac{h-\mu}{\lambda}),
\end{equation}
which, combined with the 2nd relation 
(\ref{delta-2-add}),   implies
 \begin{equation}\label{prod}
 \lambda \rho = ( \mu - h) ( h+\nu ).
\end{equation}
Replacing   $F_{23}$ with the right-hand side of  (\ref{F-23-12})
we  obtain, from relations 
(\ref{equivalent}) and  the 5th relation (\ref{syst1}), 
 that $(F_{12})^{2}$ satisfies 
\begin{align}
\nonumber& (F_{12})^{2} = \frac{\epsilon_{3}}{4} ( - h^{2} - 3\mu^{2} + \nu^{2} - 2\mu \nu ) -\epsilon_{1} \lambda^{2}\\
\nonumber&  (F_{12})^{2}  (\frac{ h-\mu}{\lambda} ) = \epsilon_{1} \lambda \nu +\epsilon_{3} \rho \mu\\
\label{patrat-F-restr}& (F_{12})^{2} ( \frac{h-\mu}{\lambda })^{2}  = \frac{\epsilon_{1}}{4} ( - h^{2} - 3\nu^{2} +\mu^{2} - 2\mu \nu ) -\epsilon_{3} \rho^{2},
\end{align}
where in the second relation above we used that $\delta_{2}  = -(\epsilon_{1} \lambda +\epsilon_{3}\rho )$.
On the other hand,  remark that relations
(\ref{adaus})
and (\ref{prod}) 
imply  that 
\begin{equation}
( \mu - \nu - h) ( \mu - \nu - 3h) =0.
\end{equation}
When $ h = \mu - \nu$  straightforward computations lead to the class of solutions 3b) with $\rho (\mu - \nu ) \neq 0.$ 
When $ h =\frac{1}{3} ( \mu - \nu )$, we claim that we arrive  at the  class of solutions 5).
The key point of the argument is to solve   the system
\begin{align}
\nonumber& 9\lambda \rho = (\nu + 2\mu ) ( \mu + 2\nu )\\
\nonumber&  (F_{12})^{2} = \frac{\epsilon_{3}}{9} ( - 7 \mu^{2} + 2\nu^{2} - 4\mu \nu )  -\epsilon_{1} \lambda^{2}\\
\nonumber& (F_{12})^{2} = -\frac{3\lambda}{(2\mu +\nu )} ( \epsilon_{1} \lambda \nu  +\epsilon_{3} \rho \mu )\\
\label{compatib-compl}&  (F_{12})^{2} = \frac{ 9\lambda^{2}}{ ( 2\mu  + \nu )^{2}} \left( 
\frac{\epsilon_{1}}{9}  ( - 7\nu^{2} + 2\mu^{2} - 4\mu \nu )  - \epsilon_{3} \rho^{2}\right),
\end{align}
which is obtained from  relations (\ref{adaus}) and (\ref{patrat-F-restr}), using 
$h = \frac{1}{3} ( \mu - \nu )$, under the assumption that $\mu \neq \nu$ (as $h \neq 0$). 
The compatibility between the 2nd  and 3rd relations  (\ref{compatib-compl}), respectively the 
3rd and 4th relations (\ref{compatib-compl}), reduce to 
\begin{align}
\nonumber& 9 \epsilon_{1} \lambda^{2} + \epsilon_{3} ( 2\mu +\nu )^{2} =0\\
\label{compat1}& \epsilon_{1} \lambda ( \mu + 2\nu ) +\epsilon_{3} \rho ( 2\mu + \nu ) =0,
\end{align}
where we have used the 1st relation (\ref{compatib-compl}) to write $\lambda \rho $ in terms of $\mu$ and $\nu $
and $\mu \neq \nu$. 
 From the 1st relation (\ref{compat1}), we obtain that $\epsilon_{3} = -\epsilon_{1}$ and 
\begin{equation}\label{compat2}
\lambda = \frac{\epsilon}{3}  (2\mu +\nu )
\end{equation}
(where $\epsilon \in \{ \pm 1\}$). Then the  2nd  relation (\ref{compat1}) implies 
\begin{equation}
\rho = \frac{\epsilon}{3} ( 2\nu +\mu ).
\end{equation}
Building on these formulae we obtain the class of solutions 5).
\end{proof}

\begin{lem} When $F_{12} F_{23} \neq  0$, 
$\delta_{1} =\delta_{4}$, $\delta_{3} =\delta_{6}$ and $\delta_{1}\delta_{3} \neq 0$, the system
 (\ref{3-nondeg}), (\ref{4-nondeg}) and (\ref{5-nondeg}) leads to the classes of solutions 3a) and 4b).
 \end{lem}

\begin{proof} Since $\delta_{3} \neq 0$, the 2nd relation (\ref{5-nondeg}) expresses  $F_{23}$ in terms of
$F_{12}$: 
\begin{equation}\label{F23-12}
F_{23} = \epsilon_{1} \epsilon_{3} F_{12} \frac{\delta_{1}}{\delta_{3}}.
\end{equation}
Since $\delta_{1} =\delta_{4}$ and $\delta_{3} =\delta_{6}$, the 5th relation (\ref{5-nondeg}) can be written as 
\begin{equation}\label{adaus-2}
\delta_{0} F_{12} = \epsilon_{1} \lambda \delta_{1}  +\epsilon_{3} \mu \delta_{3}.
\end{equation}
Combining relation (\ref{adaus-2}) with the 2nd relation 
(\ref{3-nondeg}) and using the 1st relation (\ref{add}) together with $\delta_{3} \neq 0$ 
we obtain that 
\begin{equation}\label{h-mu-nu}
h = \mu - \nu .
\end{equation}
On the other hand, the 1st relation (\ref{add}) 
and $\delta_{1}\delta_{3}\neq 0$ 
implies that either $\lambda \nu \neq 0$ or  $\lambda = \nu =0.$ 
Assume first that $\lambda \nu \neq 0$.  Then,  from the 1st relation (\ref{add}) 
and relation (\ref{F23-12}), we obtain that 
\begin{equation}\label{delta3-1}
\delta_{3}= - \epsilon_{1} \epsilon_{3} \frac{\lambda}{\nu }\delta_{1},\quad  F_{23} = -\frac{\nu}{\lambda} F_{12}.
\end{equation}
We claim  that we arrive at the class of solutions 3a). To prove the claim, 
we  replace in relations  (\ref{3-nondeg}) (without the 4th,  5th, 6th and 7th relations), (\ref{5-nondeg}),  (\ref{add})
and  (\ref{add-after}), 
$h$, $\delta_{3}$ and  $F_{23}$ by the right-hand  sides of   (\ref{h-mu-nu}) and
(\ref{delta3-1}). 
Recall, from our assumption 
$\delta_{1} \delta_{3} \neq 0$ and 
relation (\ref{lambda-rho}), that 
$\lambda \rho = \mu \nu $. From the 1st relation (\ref{5-nondeg}) we obtain that
$\delta_{2} = -(\epsilon_{1} \lambda +\epsilon_{3} \rho ).$ 
Relations (\ref{add}) 
and relation (\ref{add-after}) give no further restrictions. 
Relations (\ref{5-nondeg}) reduce to 
relation (\ref{adaus-2}), which, from  the 1st relation (\ref{delta3-1}),  is equivalent to 
$\delta_{0} F_{12} = \epsilon_{1} \frac{\lambda}{\nu} \delta_{1} ( \nu - \mu )$. 
Relations (\ref{3-nondeg})  give in addition 
$(F_{12})^{2}  = - (\epsilon_{1} \lambda^{2} +\epsilon_{3} \mu^{2}).$ 
We arrive at the class of solutions 3a), as claimed.
By similar computations, when $\lambda = \nu =0$ we arrive at the class of solutions 4b). 
\end{proof}

\begin{lem}  When $F_{12} F_{23} \neq  0$, 
$\delta_{1} \neq  \delta_{4}$, $\delta_{3} \neq \delta_{6}$ and $\delta_{1} \delta_{3} \neq  0$, the system
 (\ref{3-nondeg}), (\ref{4-nondeg}) and (\ref{5-nondeg}) leads
 to the class of solutions 3c) with  $\delta_{1}\neq 0$ and to  
 the class of solutions  
4a) with $\delta_{0} (\delta_{1} -\delta_{4}) \neq 0$.
\end{lem}

\begin{proof} 
From the 7th relation (\ref{5-nondeg}) and $(\delta_{1} -\delta_{4}) (\delta_{3}-\delta_{6})\neq 0$, either 
$\lambda = \nu =0$ or $ \lambda \nu \neq 0.$ Assume first that $\lambda = \nu =0.$ Then $\rho \neq 0$ (since $\mathrm{tr}\, \mathrm{ad}_{v_{2}} = 
\epsilon_{1} \lambda + \epsilon_{3}
\rho \neq 0$), and, from the last relation
(\ref{5-nondeg}),  also $\mu \neq  0.$ The last relation (\ref{5-nondeg}) becomes
\begin{equation}\label{1436}
\frac{\delta_{1} -\delta_{4}}{ \delta_{3} -\delta_{6}} = -\epsilon_{1} \epsilon_{3} \frac{\rho}{\mu }.
\end{equation}
From  relation (\ref{1436}) combined with the 4th relation (\ref{5-nondeg}),  we obtain that $F_{23} = -\frac{\rho }{\mu}F_{12}$
and, from the 1st relation (\ref{5-nondeg}) we obtain that 
\begin{equation}\label{contr}
\delta_{2} = -\frac{\epsilon_{3} \rho h}{\mu }.
\end{equation}
Replacing these expressions of $F_{23}$ and $\delta_{2}$ into the 3rd relation (\ref{5-nondeg}) we obtain that 
\begin{equation}
( \mu - h) (\mu^{2} + \epsilon_{1} \epsilon_{3} \rho^{2}) =0.
\end{equation}
We claim that $h = \mu .$ If, by contradiction, $h \neq \mu$,  then, from the above relation, 
$\epsilon_{3} = - \epsilon_{1}$, $\mu = \epsilon \rho$ (where $\epsilon \in \{ \pm 1\}$), $F_{12} = - \epsilon F_{23}$ 
and the 5th relation (\ref{3-nondeg}) implies that  $ h = - \mu$. From the 1st relation (\ref{3-nondeg}) we conclude that  $\epsilon_{3} =-1$
and $F_{12} = \tilde{\epsilon} \mu$ (where $\tilde{\epsilon}\in \{ \pm 1\}$).  From relation (\ref{contr}) we obtain that
$\delta_{2} = - \epsilon \mu .$ But 
the 3rd relation (\ref{3-nondeg}) implies that
$\delta_{2} = \epsilon \mu .$ We obtain a contradiction. 
We proved that $h =\mu $  and 
we arrive at the class of solutions 4a) with $\delta_{0} (\delta_{1} -\delta_{4}) \neq 0.$ 
Assume now that $\lambda \nu \neq 0.$ 
Since $\lambda\neq 0$,  the 1st relation (\ref{add}) implies that  $\frac{\delta_{1}}{\delta_{3}} = - \epsilon_{1} \epsilon_{3} \frac{\nu}{\lambda}$.
Then   $F_{23} = -\frac{\nu}{\lambda} F_{12}$, and, from the 
3rd relation (\ref{5-nondeg}), 
$h = \mu +\epsilon_{1}\frac{\nu}{\lambda} ( \delta_{2} + \epsilon_{3} \rho ).$
Replacing this expression of $h$ into the 1st relation (\ref{5-nondeg})  and 
using relation  (\ref{lambda-rho}) 
together with $F_{23} = - \frac{\nu}{\lambda} F_{12}$ and $F_{12}\neq 0$, 
we obtain that 
\begin{equation}\label{auxi:eq}
( 1+\epsilon_{1} \epsilon_{3}  (\frac{\nu}{\lambda})^{2}) (\delta_{2} +\epsilon_{1} \lambda +\epsilon_{3} \rho ) =0.
\end{equation}
If the first factor from  the left-hand side of (\ref{auxi:eq}) vanishes, then $ \epsilon_{3} = -\epsilon_{1}$, $\nu = \epsilon \lambda$
(where $\epsilon \in \{ \pm 1\}$),  $ \rho = \epsilon \mu$ (as $\lambda\rho = \mu \nu$),  $F_{23} = - \epsilon F_{12}$  
and $\delta_{2} = \epsilon \epsilon_{1} h$ (from the 1st relation  (\ref{5-nondeg}). 
The 5th relation 
(\ref{3-nondeg})  implies that $h =\tilde{\epsilon} ( \lambda -\rho )$, where $\tilde{\epsilon}\in \{ \pm 1\} .$   In particular, 
$h\neq 0$ as $\lambda  -\rho = \epsilon_{1} ( \epsilon_{1} \lambda +\epsilon_{3} \rho ) \neq 0$ 
and $\delta_{2} = \epsilon \tilde{\epsilon} \epsilon_{1} ( \lambda -\rho ).$ 
From the 3rd relation (\ref{add}), 
we obtain that $\tilde{\epsilon}  = -\epsilon .$ We arrive at the class of solutions 3c) with 
$\delta_{1}\neq 0$, $\nu = \epsilon \lambda$ and $ \epsilon_{3} = - \epsilon_{1}$.  
If the first  factor from the left-hand side of (\ref{auxi:eq})  is non-zero, we arrive, 
by similar computations, to the class of solutions 3c) with $\delta_{1}\neq 0$
and $ 1+\epsilon_{1} \epsilon_{3} ( \frac{\nu}{\lambda})^{2} \neq 0.$ Therefore, when $\lambda \nu \neq 0$ 
we arrive at the class of solutions 3c) with $\delta_{1}\neq 0.$ 
 \end{proof}

\begin{lem} When $F_{12} F_{23} \neq  0$, 
$\delta_{1} \neq  \delta_{4}$, $\delta_{3} \neq \delta_{6}$ and $\delta_{1} = \delta_{3}  =  0$, the system
 (\ref{3-nondeg}), (\ref{4-nondeg}) and (\ref{5-nondeg}) leads to the class of solutions  
 3c) with $\delta_{0} = \delta_{1} =0$  and to the class of solutions 
 4a) with
 $\delta_{0} =0$ and $\delta_{4}\neq 0.$ 
  \end{lem}

\begin{proof}
Since  $F_{12}\neq 0$ and $\delta_{3} =0$,  the 2nd relation
(\ref{3-nondeg}) implies that
$\delta_{0} =0$. The  6th relation  (\ref{5-nondeg}) implies that $ h =\mu - \nu .$ 
As in the previous proof,  either $\lambda\nu \neq 0$ or $\lambda =\nu =0.$ 
By similar computations, when 
$\lambda\nu \neq 0$ we arrive  at the class of solutions 3c) with $\delta_{0} = \delta_{1} =0$.
When $\lambda = \nu =0$ we arrive at the class of solutions 4a) with $\delta_{0}  =0$ and $\delta_{4} \neq 0$.
\end{proof}

The proof of Theorem \ref{non-deg-thm}  is now completed.

Like in Proposition \ref{identif-lie},  we obtain:

\begin{prop}\label{identif-lie-non-deg} Consider the setting from Theorem \ref{non-deg-thm}.\

i) For the class of solutions 1a), 3) and  4),    $\mathfrak{g}$ is isomorphic to $\tau_{2} (\mathbb{R}) \oplus \mathbb{R}$.

ii) For the class of solutions 1b) and 6), $\mathfrak{g}$ is isomorphic to $\tau_{3, \frac{1}{2}} (\mathbb{R})$. 

iii) For the class of solutions 2) and 5), $\mathfrak{g}$ is isomorphic to $\tau_{3, -\frac{1}{2}} (\mathbb{R}).$ 
\end{prop}

It remains to describe the odd generalized Einstein metrics which correspond to case 7) of Theorem \ref{non-deg-thm}.
This will be done in the next section.

\subsubsection{The case 7) of Theorem \ref{non-deg-thm}}\label{case-9}

We   need to solve the system  (\ref{syst1}) with
$h =0$ and 
 $\mu$, $\nu$, $F_{12}$, $F_{23}$, $\epsilon_{1} \lambda +\epsilon_{3} \rho$   all  non-zero.  It  is convenient to replace
the variables $\lambda$ and $\rho$ with 
$$
\tilde{\lambda }
:= \delta_{2} +\epsilon_{1} \lambda ,\ \tilde{\rho} := \delta_{2} +\epsilon_{3} \rho .
$$
Since   $\epsilon_{1}\lambda + \epsilon_{3} \rho \neq 0$, we obtain that 
\begin{equation}\label{delta-diferit}
\delta_{2}\neq \frac{1}{2} ( \tilde{\lambda} +\tilde{\rho}).
\end{equation}
In terms of the unknowns $(\mu , \nu , F_{12}, F_{23},  \delta_{2}, \tilde{\lambda}, \tilde{\rho })$, the system
(\ref{syst1}) takes the form 
\begin{align}
\nonumber& F_{23} = F_{12} \epsilon_{3} \frac{\tilde{\lambda}}{\nu}\\
\nonumber& \mu \nu =\epsilon_{1} \epsilon_{3} \tilde{\lambda}\tilde{\rho}\\
\nonumber& (F_{12})^{2} = \frac{\epsilon_{3}}{2} (\nu^{2} - \mu^{2}) +\epsilon_{1} \delta_{2} (\tilde{\lambda} -\delta_{2})\\
\nonumber&  \delta_{2} ( \tilde{\lambda} +\tilde{\rho}) = \tilde{\lambda}^{2} +\tilde{\rho}^{2}  +\frac{\epsilon_{1}\epsilon_{3}}{2} 
(\mu + \nu)^{2}\\
\nonumber& (F_{12})^{2} \frac{\tilde{\lambda}}{\nu} = \frac{\epsilon_{3}}{2} (\nu - \mu)  (\tilde{\lambda} -\tilde{\rho} ) -\frac{\epsilon_{3}\delta_{2}}{ 2}
(\mu + \nu )\\
\label{syst1-var}& (F_{12})^{2} (\frac{\tilde{\lambda}}{\nu} )^{2} =\frac{\epsilon_{1}}{2} (\mu^{2} - \nu^{2}) +\delta_{2} \rho .
\end{align}
(The 2nd, 5th and 6th relation (\ref{syst1-var}) are obtained from the 2nd, 5th and 6th relation (\ref{syst1}),  by 
replacing $F_{23}$ with $F_{12} \epsilon_{3} \frac{\tilde{\lambda}}{\nu}$, see  the 1st relation (\ref{syst1-var})).
Remark, from  the second relation (\ref{syst1-var}) and $\mu\nu\neq 0$,  that $\tilde{\lambda}\tilde{\rho}\neq 0.$
The 1st, 2nd and 3rd relations  (\ref{syst1-var}) determine $F_{23}$,   $\tilde{\rho}$ and $F_{12}$
in terms of  $\tilde{\lambda}$,  $\mu$, $\nu$, 
$\delta_{2}.$ 
We are left with the 4th, 5th and 6th relations (\ref{syst1-var}), which are compatibility relations
between the unknows  $\tilde{\lambda}$, $\mu$,  $\nu$, $\delta_{2}$.

\begin{prop}
i) The expression $\tilde{\lambda}+\epsilon_{1}\epsilon_{3} \frac{\mu \nu}{\tilde{\lambda}}$  is non-zero. 
The 4th relation (\ref{syst1-var}) determines  the  unknown  $\delta_{2}$ in terms of 
$\tilde{\lambda}$, $\mu$, $\nu$,  by  
\begin{equation}\label{compatibility-delta}
\delta_{2}  (\tilde{\lambda}   +\epsilon_{1} \epsilon_{3} \frac{\mu\nu}{\tilde{\lambda }}) = \tilde{\lambda}^{2} + (\frac{\mu \nu}{ \tilde{\lambda}})^{2} 
+\frac{\epsilon_{1} \epsilon_{3}}{2} (\mu + \nu )^{2}.
\end{equation}
ii)  The 5th and 6th relations (\ref{syst1-var}) 
hold if and only if 
\begin{equation}\label{compatibility}
(\mu - \nu   ) (\tilde{\lambda}^{2} +\epsilon_{1}\epsilon_{3}  \nu^{2}) 
\left( \tilde{\lambda}^{2} +(\frac{\mu \nu}{\tilde{\lambda}})^{2} +\frac{\epsilon_{1} \epsilon_{3}}{2} (\mu + \nu )^{2} + (\tilde{\lambda} +\epsilon_{1}\epsilon_{3} 
\frac{\mu\nu}{ \tilde{\lambda}})^{2} \right) =0.
\end{equation}
\end{prop}

\begin{proof}
Relation (\ref{compatibility-delta}) follows immediately from the 4th relation (\ref{syst1-var}), 
by using $\tilde{\rho } = \epsilon_{1} \epsilon_{3} \frac{\mu \nu}{\tilde{\lambda}}$.
Using the expression of $(F_{12})^{2}$ 
provided by  the 3rd relation (\ref{syst1-var}) and $\tilde{\rho } = \epsilon_{1} \epsilon_{3} \frac{\mu \nu}{\tilde{\lambda}}$, we now rewrite 
the 5th and 6th relation (\ref{syst1-var}). 
We obtain that the 
5th relation (\ref{syst1-var})  is equivalent to
\begin{equation}\label{5-var}
(\delta_{2})^{2}  - \left( \tilde{\lambda} +\epsilon_{1}\epsilon_{3} \frac{(\mu + \nu) \nu}{2\tilde{\lambda}} \right)\delta_{2} +\frac{\epsilon_{1}\epsilon_{3} }{2} (\mu - \nu)
\mu \left( 1+ \epsilon_{1} \epsilon_{3} \frac{\nu^{2}}{ \tilde{\lambda}^{2}} \right)  =0.
\end{equation}
Similarly,  the 6th relation (\ref{syst1-var}) is equivalent to
\begin{equation}\label{6-var}
(\delta_{2})^{2}  ( \epsilon_{1} \tilde{\lambda}^{2}-\epsilon_{3} \nu^{2}) +\epsilon_{1} ( \frac{\mu \nu^{3}}{ \tilde{\lambda}} -\tilde{\lambda}^{3} )\delta_{2}
+\frac{1}{2} ( \epsilon_{3}  \tilde{\lambda}^{2} +\epsilon_{1}  \nu^{2}) (\mu^{2} -\nu^{2}) =0.  
\end{equation}
(Relation (\ref{6-var}) is obtained in the following way:  from  the 3rd relation (\ref{syst1-var}), 
 \begin{equation}\label{mu-nu}
 \frac{1}{2} (\mu^{2} - \nu^{2}) = -\epsilon_{3} (F_{12})^{2} +\epsilon_{1}\epsilon_{3} \delta_{2} ( \tilde{\lambda} -\delta_{2}).
  \end{equation}
Replacing in the  6th relation  (\ref{syst1-var})  the term $\frac{\epsilon_{1}}{2} ( \mu^{2} - \nu^{2})$ by   the right-hand side of  (\ref{mu-nu}) 
multiplied with  $\epsilon_{1}$ and using that $\rho = \epsilon_{3} ( \tilde{\rho} -\delta_{2})$, 
we obtain
\begin{equation}\label{agove} 
(F_{12})^{2} ( (\frac{\tilde{\lambda}}{\nu })^{2} +\epsilon_{1} \epsilon_{3}) = \epsilon_{3} \delta_{2} (\tilde{\lambda} +\tilde{\rho }) - 2\epsilon_{3} ( \delta_{2})^{2}.
\end{equation}
Using now the expression of $\delta_{2} (\tilde{\lambda }
 +\tilde{\rho})$ provided by  the 4th relation (\ref{syst1-var})
 and the 2nd relation (\ref{syst1-var}), 
 we obtain that relation (\ref{agove})  is equivalent to
 \begin{equation}\label{6-var-initial}
 (F_{12})^{2} ( ( \frac{\tilde{\lambda}}{\nu})^{2} +\epsilon_{1}\epsilon_{3} ) =\epsilon_{3} ( \tilde{\lambda}^{2} +\tilde{\rho}^{2} +\tilde{\lambda}\tilde{\rho} )
 +\frac{\epsilon_{1}}{2} (\mu^{2} +\nu^{2}) - 2\epsilon_{3} (\delta_{2})^{2}.
 \end{equation}
 Replacing in (\ref{6-var-initial})   $(F_{12})^{2}$ by  its expression from the 3rd relation (\ref{syst1-var})  and using  \eqref{compatibility-delta} 
together with $\mu \nu = \epsilon_{1}\epsilon_{3} \tilde{\lambda}\tilde{\rho}$  we obtain  relation (\ref{6-var}), as  needed.)\

We now claim that $\tilde{\lambda}+\epsilon_{1}\epsilon_{3} \frac{\mu \nu}{\tilde{\lambda}}\neq 0.$ Assume, by contradiction, that 
 $\tilde{\lambda}+\epsilon_{1}\epsilon_{3} \frac{\mu \nu}{\tilde{\lambda}} = 0.$ 
Then  the right-hand  side of relation (\ref{compatibility-delta}) vanishes, that is, 
\begin{equation}
\epsilon_{3} = -\epsilon_{1},\ \tilde{\lambda}^{2} + (\frac{\mu \nu}{\tilde{\lambda}})^{2} =\frac{1}{2} ( \mu +\nu )^{2}.
\end{equation}
We deduce  that 
$\mu = \nu $, since $\mu \nu =  \tilde{\lambda}^{2} $.   Hence   $\tilde{\lambda}^{2} = \nu^{2}.$  
But then relation (\ref{5-var}) implies that $\delta_{2} =0$ and  the 3rd relation  (\ref{syst1-var}) implies that $F_{12} =0$, which is  a contradiction. 
The  claim is proved.\

Straightforward computations show that  by multiplying relation 
(\ref{5-var}) with  the non-zero  factor $(\tilde{\lambda}+\epsilon_{1}\epsilon_{3} \frac{\mu \nu}{\tilde{\lambda}})^{2}$
and using  relation (\ref{compatibility-delta}) we obtain  relation  (\ref{compatibility}). Similarly,   multiplying 
relation (\ref{6-var}) with $(\tilde{\lambda}+\epsilon_{1}\epsilon_{3} \frac{\mu \nu}{\tilde{\lambda}})^{2}$
and using  relation (\ref{compatibility-delta}) again we obtain 
$$
(\mu^{2} - \nu^{2}   ) (\tilde{\lambda}^{2} +\epsilon_{1}\epsilon_{3}  \nu^{2}) 
\left( \tilde{\lambda}^{2} +(\frac{\mu \nu}{\tilde{\lambda}})^{2} +\frac{\epsilon_{1} \epsilon_{3}}{2} (\mu + \nu )^{2} + (\tilde{\lambda} +\epsilon_{1}\epsilon_{3} 
\frac{\mu\nu}{ \tilde{\lambda}})^{2} \right) =0.
$$
Remark that the left-hand side of the above relation coincides with the left-hand side of 
(\ref{compatibility}), multiplied with  $\mu +\nu .$ 
This concludes our proof.
\end{proof}

We arrive at the description of odd generalized Einstein metrics from case 7) of Theorem \ref{non-deg-thm}.

\begin{cor}\label{case-7}  There are two classes  i) and ii) (see below) of odd generalized  Einstein metrics generated by case 7) of Theorem
\ref{non-deg-thm}. The corresponding  constants $\tilde{\lambda}$,  $\tilde{\rho}$,  $\mu$, $\nu$,  $\epsilon_{1}$, $ \epsilon_{3}$, 
$F_{12}$, $F_{23}$, $\delta_{2}$
satisfy the following conditions:

i)  $  \mu = \nu$ ,  $\mu \tilde{\lambda }\neq 0$,   $\delta_{2} = \tilde{\lambda}  +\epsilon_{1} \epsilon_{3} \frac{\mu^{2}}{\tilde{\lambda}} $
and 
\begin{equation}
(F_{12})^{2} = -(\frac{\mu}{\tilde{\lambda}})^{2} (\epsilon_{3} \tilde{\lambda}^{2} +\epsilon_{1} \mu^{2}) >0;
\end{equation}

ii) $\tilde{\lambda}^{2}\neq \nu^{2}$, 
$\mu  \tilde{\lambda}\neq 0$, 
$\epsilon_{3} = -\epsilon_{1}$, $\delta_{2} =\frac{\epsilon}{2} (\nu - \mu )$ (where 
$\epsilon \in \{ \pm 1\}$) and
\begin{align}
\nonumber&  \tilde{\lambda}^{2} + \frac{\epsilon}{2} ( \nu -\mu  ) \tilde{\lambda } -\mu \nu =0\\
\label{ec-lambda}& (F_{12})^{2} =\epsilon_{1} (\nu - \mu ) \nu ( \epsilon \frac{\mu}{2\tilde{\lambda}} -1) >0.
\end{align}
For both families, 
\begin{equation}
F_{23} = F_{12}\epsilon_{3} \frac{\tilde{\lambda}}{\nu},\ \tilde{\rho } = \epsilon_{1} \epsilon_{3} \frac{ \mu \nu }{\tilde{\lambda}}.
\end{equation}
For the family i), $\mathfrak{g}$ is isomorphic to $\tau_{2} (\mathbb{R}) \oplus \mathbb{R}.$ For the family ii),
$\mathfrak{g}$ is isomorphic to $\tau_{3, \frac{1}{2} }(\mathbb{R})$.
\end{cor}
  
\begin{proof}
Starting from  relation  (\ref{compatibility}) we distinguish various cases. If  $\mu = \nu$ then we obtain the family i).
If   $\mu \neq \nu$ and 
$\tilde{\lambda}^{2} +\epsilon_{1} \epsilon_{3} \nu^{2}=0$ we arrive at 
$\delta_{2} =\frac{\epsilon}{2} ( \nu -\mu )$, $\tilde{\lambda} =\epsilon \nu$, 
$\tilde{\rho} = -\epsilon \mu$ 
 and $\epsilon_{3} = -\epsilon_{1}.$ 
But then 
relation (\ref{delta-diferit}) is not satisfied. 
Finally,  if 
$\mu \neq \nu$ and $ \tilde{\lambda}^{2} +\epsilon_{1}\epsilon_{3}  \nu^{2}\neq 0$ then relation (\ref{compatibility}) becomes 
\begin{equation}\label{last_comp:eq}
\tilde{\lambda}^{2} +(\frac{\mu \nu}{\tilde{\lambda}})^{2} +\frac{\epsilon_{1} \epsilon_{3}}{2} (\mu + \nu )^{2} + (\tilde{\lambda} +\epsilon_{1}\epsilon_{3} 
\frac{\mu\nu}{ \tilde{\lambda}})^{2} =0.
\end{equation} 
In particular, $\epsilon_{3} =- \epsilon_{1}$, $\tilde{\lambda}^{2} \neq \nu^{2}$ and 
the above relation reduces to the 1st relation (\ref{ec-lambda}). 
In fact,  relation (\ref{last_comp:eq})  
can be rewritten as 
\begin{eqnarray*} \frac12 (\mu +\nu )^2 &=& \tilde{\lambda}^{2} +(\frac{\mu \nu}{\tilde{\lambda}})^{2} + (\tilde{\lambda} - \frac{\mu\nu}{ \tilde{\lambda}})^{2}\\ 
&=&2\left( \tilde{\lambda} - \frac{\mu \nu}{\tilde{\lambda}}\right)^2+2\mu \nu\end{eqnarray*}
and therefore as
\[ \frac12 (\mu -\nu )^2 = 2\left( \tilde{\lambda} - \frac{\mu \nu}{\tilde{\lambda}}\right)^2,\]
which leads to 
\begin{equation} \label{muminusnu:eq} \mu -\nu  = 2 \epsilon (\tilde{\lambda} - \frac{\mu \nu}{\tilde{\lambda}}),\quad \epsilon \in \{ \pm 1\}.\end{equation} 
That is the 1st relation (\ref{ec-lambda}). From  relation \eqref{last_comp:eq} we also see that 
the right-hand side of \eqref{compatibility-delta} is precisely $-(\tilde{\lambda} - \frac{\mu \nu}{\tilde{\lambda}})^2$, which proves that 
\[ \delta_2 = -(\tilde{\lambda} - \frac{\mu \nu}{\tilde{\lambda}})  =\frac{\epsilon}{2} (\nu - \mu ),\]
where in the 2nd  equality we used relation (\ref{muminusnu:eq}). 
We now prove the 2nd relation (\ref{ec-lambda}). 
From the 3rd relation  
(\ref{syst1-var}) we obtain 
\begin{align}
\nonumber(F_{12})^{2} & =  - \frac{\epsilon_{1}}{2} (\nu ^{2} - \mu^{2}) +\frac{\epsilon_{1} \epsilon}{2} ( \nu - \mu ) \left( \tilde{\lambda} - 
\frac{\epsilon}{2} (\nu -\mu )\right)\\
\nonumber& = \frac{\epsilon_{1}}{2} (\nu -\mu  ) ( \epsilon \tilde{\lambda}  -\frac{\mu}{2}  - \frac{3\nu}{2} )
\end{align}
where  in the 1st  relation above we used that  $\epsilon_{3} = -\epsilon_{1}$ and $\delta_{2} = \frac{\epsilon}{2} (\nu - \mu ).$ 
On the other hand, the 1st  relation (\ref{ec-lambda}) implies that 
\begin{equation}
\frac{1}{2}  ( \epsilon \tilde{\lambda} -\frac{\mu}{2}  - \frac{3\nu}{2} )= \nu (\epsilon \frac{\mu}{2\tilde{\lambda}} -1)
\end{equation} 
The 2nd  relation (\ref{ec-lambda}) follows 
and we  obtain 
the family ii).

We remark that 
relation  (\ref{delta-diferit}) is satisfied for both families i) and ii). For example, for the family ii), 
$$
\frac{1}{2} ( \tilde{\lambda} +\tilde{\rho }) =\frac{1}{2} ( \tilde{\lambda} + \epsilon_{1} \epsilon_{3} \frac{\mu \nu}{ \tilde{\lambda  }})=   
 \frac{\epsilon}{4} (\mu - \nu )\neq \delta_{2} 
$$ 
as  $\mu \neq \nu$.\

The statements on the Lie algebra $\mathfrak{g}$ are  straightforward,
by using that the matrix $\mathbf{A}$ is given by (\ref{form-A-deg})
and 
\begin{equation}
\lambda = \epsilon_{1} ( \tilde{\lambda} -\delta_{2}),\ \rho = \epsilon_{1}\frac{ \mu \nu }{ \tilde{\lambda}} -\epsilon_{3} \delta_{2}.
\end{equation} 
For example, for the family ii) the matrix $\mathbf{A}$ is given by
\begin{equation}
\mathbf{A} = \left(\begin{tabular}{cc}
$\tilde{\lambda} - \frac{\epsilon}{2} (\nu - \mu )$ & $ \epsilon_{1} \nu$\\
$ -\epsilon_{1} \mu $ & $ -\frac{\mu\nu}{\tilde{\lambda}} - \frac{\epsilon}{2} (\nu - \mu )$
\end{tabular}\right) .
\end{equation}
Its  the characteristic polynomial  is given by 
\begin{equation}
P_{\mathbf{A}}(x) = x^{2} +\frac{3\epsilon}{2} (\nu - \mu) x +\frac{1}{2} ( \nu - \mu )^{2}.
\end{equation}
As the roots of $P_{\mathbf{A}}$ are given by $x_{\pm}  = \frac{1}{4} (3\epsilon \pm 1) (\mu - \nu )$ 
and $\mu \neq \nu$, 
$\epsilon \in \{ \pm 1\}$, 
we obtain that  $\mathfrak{g}$ is isomorphic to $\tau_{3,\frac{1}{2}} (\mathbb{R})$, as claimed.
\end{proof}

V. Cort\'es: vicente.cortes@math.uni-hamburg.de\

Department of Mathematics and Center for Mathematical Physics, University of Hamburg,  Bundesstrasse 55, D-20146, Hamburg, Germany.\\

L. David: liana.david@imar.ro\

Institute of Mathematics  `Simion Stoilow' of the Romanian Academy,   Calea Grivitei no. 21,  Sector 1, 010702, Bucharest, Romania.

\end{document}